\documentclass[psamsfonts]{amsart}
\usepackage{amssymb,amsfonts,array}
\usepackage{amsmath}
\usepackage{mathtools}
\usepackage{tabulary}
\usepackage[all,arc]{xy}
\usepackage{enumerate}
\usepackage{easyReview}
\usepackage{mathrsfs}
\usepackage{xcolor}
\usepackage[hidelinks]{hyperref}
\hypersetup{
colorlinks=true,
linkcolor=blue,
filecolor=blue,
urlcolor=blue,
citecolor=cyan}

\usepackage{bm}
\usepackage{dsfont}
\usepackage{pgfplots}
\usepackage{tikz}
\usepackage{multirow}
\usepackage{tikz-cd}
\usetikzlibrary{cd,shapes,arrows,positioning,calc}
\usepackage{tikz}
\usetikzlibrary{intersections,through,backgrounds}
\usepackage{tkz-base}
\usepackage{tkz-euclide}
\usepackage{geometry}
\usepackage[OT2,T1]{fontenc}
\DeclareSymbolFont{cyrletters}{OT2}{wncyr}{m}{n}
\DeclareMathSymbol{\Sha}{\mathalpha}{cyrletters}{"58}

\newtheorem{thm}{Theorem}[section]
\newtheorem{cor}[thm]{Corollary}
\newtheorem{prop}[thm]{Proposition}
\newtheorem{lemma}[thm]{Lemma}
\newtheorem{conj}[thm]{Conjecture}

\theoremstyle{definition}
\newtheorem{defn}[thm]{Definition}
\newtheorem{eg}[thm]{Example}
\newtheorem{assump}[thm]{Assumption}

\newtheorem{notation}[thm]{Notation}

\theoremstyle{remark}
\newtheorem{rmk}[thm]{Remark}

\newcommand{\AAA}{\mathbb{A}}
\newcommand{\BB}{\mathbb{B}}
\newcommand{\FF}{\mathbb{F}}
\newcommand{\CC}{\mathbb{C}}
\newcommand{\TT}{\mathbb{T}}
\newcommand{\QQ}{\mathbb{Q}}
\newcommand{\GG}{\mathbb{G}}
\newcommand{\RR}{\mathbb{R}}
\newcommand{\bS}{\mathbb{S}}
\newcommand{\ZZ}{\mathbb{Z}}
\newcommand{\LL}{\mathbb{L}}
\newcommand{\II}{\mathbb{I}}
\newcommand{\XX}{\mathbb{X}}

\newcommand{\cA}{\mathcal{A}}
\newcommand{\cC}{\mathcal{C}}
\newcommand{\cR}{\mathcal{R}}
\newcommand{\cO}{\mathcal{O}}
\newcommand{\cE}{\mathcal{E}}

\newcommand{\cW}{\mathcal{W}}
\newcommand{\cH}{\mathcal{H}}
\newcommand{\cF}{\mathcal{F}}
\newcommand{\cK}{\mathcal{K}}
\newcommand{\cI}{\mathcal{I}}
\newcommand{\cU}{\mathcal{U}}
\newcommand{\cN}{\mathcal{N}}
\newcommand{\cP}{\mathcal{P}}

\newcommand{\fI}{\mathfrak{I}}
\newcommand{\fA}{\mathfrak{A}}
\newcommand{\fm}{\mathfrak{m}}
\newcommand{\fn}{\mathfrak{n}}

\newcommand{\fp}{\mathfrak{p}}

\newcommand{\sE}{\mathsf{E}}
\newcommand{\sV}{\mathsf{V}}
\newcommand{\sF}{\mathsf{F}}

\newcommand{\sD}{\mathsf{D}}
\newcommand{\sG}{\mathsf{G}}
\newcommand{\sX}{\mathsf{X}}

\newcommand{\sK}{\mathsf{K}}
\newcommand{\sT}{\mathsf{T}}

\newcommand{\rG}{\mathrm{G}}
\newcommand{\rU}{\mathrm{U}}
\newcommand{\rM}{\mathrm{M}}

\newcommand{\git}{/\!\!/}
\newcommand{\op}{\mathrm{op}}

\newcommand{\Hk}{\mathrm{Hk}}
\newcommand{\Nt}{\mathrm{Nt}}
\newcommand{\sSh}{\mathrm{Sh}}
\newcommand{\Igs}{\mathrm{Igs}}

\newcommand{\perf}{\mathrm{perf}}
\newcommand{\Frob}{\mathrm{Frob}}
\newcommand{\Isoc}{\mathrm{Isoc}}
\newcommand{\Perf}{\mathrm{Perf}}
\newcommand{\Sht}{\mathrm{Sht}}
\newcommand{\Gal}{\mathrm{Gal}}
\newcommand{\sph}{\mathrm{sph}}
\newcommand{\Rep}{\mathrm{Rep}}
\newcommand{\Ind}{\mathrm{Ind}}
\newcommand{\Pro}{\mathrm{Pro}}
\newcommand{\cInd}{{\mathrm{c}\mbox{-}\mathrm{Ind}}}
\newcommand{\nInd}{{\mathrm{n}\mbox{-}\mathrm{Ind}}}
\newcommand{\std}{\mathrm{std}}
\newcommand{\loc}{\mathrm{loc}}
\newcommand{\End}{\mathrm{End}}
\newcommand{\Hom}{\mathrm{Hom}}
\newcommand{\Adm}{\mathrm{Adm}}
\newcommand{\aff}{\mathrm{aff}}
\newcommand{\GU}{\mathrm{GU}}

\newcommand{\GL}{\mathrm{GL}}
\newcommand{\diag}{\mathrm{diag}}
\newcommand{\rec}{\mathrm{rec}}
\newcommand{\St}{\mathrm{St}}
\newcommand{\Loc}{\mathrm{Loc}}
\newcommand{\unip}{\mathrm{unip}}
\newcommand{\Shv}{\mathrm{Shv}}
\newcommand{\Coh}{\mathrm{Coh}}
\newcommand{\fgen}{\mathrm{f.g.}}
\newcommand{\ad}{\mathrm{ad}}
\newcommand{\CohSpr}{\mathrm{CohSpr}}
\newcommand{\ext}{\mathrm{ext}}
\newcommand{\Sh}{\mathrm{Sh}}
\newcommand{\HH}{\mathrm{HH}}
\newcommand{\spec}{\mathrm{spec}}

\newcommand{\Tr}{\mathrm{Tr}}
\newcommand{\can}{\mathrm{can}}
\newcommand{\gr}{\mathrm{gr}}

\newcommand{\Cone}{\mathrm{Cone}}

\newcommand{\Igssheaf}{{\mathcal{I}gs}}

\newcommand{\Res}{\mathrm{Res}}

\newcommand{\Spec}{\operatorname{Spec}}
\newcommand{\Spf}{\operatorname{Spf}}

\title{On Ihara's lemma for definite unitary groups}
\author{Xiangqian Yang}
\address{Beijing International Center for Mathematical Research, Peking University, Beijing 100871, China}
\email{yangxq@pku.edu.cn}

\begin{document}
\begin{abstract}
	Clozel, Harris, and Taylor proposed a conjectural generalized Ihara's lemma for definite unitary groups. In this paper, we prove their conjecture with banal coefficients under some conditions. As an application, we prove a level-raising result for automorphic forms associated to definite unitary groups.
\end{abstract}
\maketitle
\setcounter{tocdepth}{2}
\tableofcontents

\section{Introduction}

Let $\Gamma=\Gamma_1(N)$ be a congruence subgroup of $\GL_2(\ZZ)$. Let $p$ be a prime number coprime to $N$, and let $\Gamma'=\Gamma\cap\Gamma_0(p)$. Let $X_\Gamma$ (resp. $X_{\Gamma'}$) be compactified modular curve of level $\Gamma$ (resp. $\Gamma'$). There are two degenerate maps
$$\pi_1,\pi_2\colon X_{\Gamma'}\to X_\Gamma$$
induces by the natural inclusion $\Gamma'\hookrightarrow\Gamma$ and its conjugation by the matrix $\begin{pmatrix}0&1\\p&0\end{pmatrix}$. Let $\ell\neq p$ be another prime. We consider the morphism
$$(\pi_1^*,\pi_2^*)\colon H^1(X_{\Gamma},\overline\FF_\ell)^{\oplus2}\to H^1(X_{\Gamma'},\overline\FF_\ell)$$
between \'etale cohomology groups. Ihara's lemma (\cite[Lemma 3.2]{Ihara}) asserts that the kernel of the above map consists only of the Eisenstein classes. This result plays a key role in studying congruences between modular forms, as in \cite{Ribet-congruence}, and has many important arithmetic applications.

In this paper, we study the generalized version of Ihara's lemma for higher dimensional unitary groups, proposed by Clozel, Harris, and Taylor in \cite{CHT08}. We first introduce some notations. Let $\sF^+$ be a totally real number field, and $\sF/\sF^+$ be a quadratic CM extension. Assume that $\sF$ contains an imaginary quadratic field $\sK$. Let $\sD$ be a division algebra over $\sF$ equipped with an involution $\dagger$ compatible with the complex conjugation on $\sF$. Associated to the pair $(\sD,\dagger)$, we can define a unitary similitude group $\sG$ over $\QQ$ . Assume that $\sG(\RR)$ is compact modulo center. 

Let $p$ be a prime number that splits in $\sK$. Fix a place $\wp$ of $\sK$ above $p$. Fix a place $v$ of $\sF$ above $\wp$, and assume that $\sD$ splits at $v$. There is a natural isomorphism:
$$\sG(\QQ_p)\cong \QQ_p^\times\times\GL_n(\sF_v)\times\prod_{\substack{w|\wp\\w\neq v}}\sD_{\sF_w}^\times.$$
Define $H(\QQ_p)=\QQ_p^\times\times\prod_{w|\wp,w\neq v}\sD_{\sF_w}^\times$. Fix a neat open compact subgroup $K^p=\prod_{r\neq p}K_r\subseteq \sG(\AAA_f^p)$ and an open compact subgroup $K_p^v\subseteq H(\QQ_p)$. Denote $K^v=K^pK_p^v\subseteq \sG(\AAA_f^p)\times H(\QQ_p)$. 

Fix a prime number $\ell\neq p$, and let $\Lambda=\overline\FF_\ell$. Let
$$\cA(K^v)\coloneqq H^0(\sG(\QQ)\backslash\sG(\AAA_f)/K^v,\Lambda)$$
be the space of automorphic forms of $\sG$ with coefficients in $\Lambda$. Let $S$ be a set of primes of $\QQ$, including $
\ell$, $p$, and all the finite places $r$ with $K_r$ not hyperspecial. Let $\TT^S=\Lambda[K^S\backslash\sG(\AAA_f^S)/K^S]$ denote the global Hecke algebra. Then $\cA(K^v)$ is equipped with an action of $\TT^S\times\GL_n(\sF_v)$.

We recall the generalized Ihara's lemma in \cite{CHT08}. An irreducible $\GL_n(\sF_v)$-representation over $\Lambda$ is called \emph{Whittaker-generic} if it admits non-zero Whittaker coefficients. This terminology is called \emph{generic} in \cite{CHT08}.

\begin{conj}[{\cite[Conjecture B]{CHT08}}]\label{conj-Ihara}
	Let $\fm\subseteq \TT^S$ be a non-Eisenstein maximal ideal with $\cA(K^v)_\fm\neq 0$. Let $V\subseteq \cA(K^v)_\fm$ be an irreducible $\GL_n(\sF_v)$-submodule. Then $V$ is Whittaker-generic. 
\end{conj}

Fix a maximal ideal $\fm\subseteq \TT^S$ such that $\cA(K^v)_\fm\neq 0$. There is a Galois representation $\rho_\fm\colon \Gal(\overline\sF/\sF)\to\GL_n(\Lambda)$ attached to $\fm$ (See Theorem \ref{thm-assoc-Galois}). Assume that $\cA(K^v)^{\GL_n(O_{\sF_v})}_\fm\neq 0$. Then $\rho_\fm$ is unramified at $v$. Let $\Frob_v\in \Gal(\overline\sF/\sF)$ be a geometric Frobenius at $v$. Our main result is the following theorem.

\begin{thm}[Corollary \ref{cor-CHT}]\label{thm-main-intro}
	Let $\fm\subseteq \TT^S$ be a maximal ideal such that $\cA(K^v)_\fm^{\GL_n(O_{\sF_v})}\neq 0$.
	Suppose the following conditions hold:
	\begin{enumerate}[(i)]
		\item $\ell$ is banal at $v$, i.e. $\ell\nmid |\GL_n(\FF_q)|$, where $q$ is the cardinality of the residue field of $\sF_v$.
		\item $\rho_\fm(\Frob_v)\in \GL_n(\Lambda)$ is regular semisimple.
		\item $\sF_v$ is unramified over $\QQ_p$ and $\sF_v\neq \QQ_p$.
        \item $\fm$ is cohomologically generic (Definition \ref{def-Gal-generic}).
	\end{enumerate}
	Then any irreducible $\GL_n(\sF_v)$-submodule of $\cA(K^v)_\fm$ is Whittaker-generic.
\end{thm}

\begin{rmk}
    Condition (iv) is a replacement of the non-Eisenstein condition in Conjecture \ref{conj-Ihara}. By results of \cite{YZ-torsion-vanishing} (generalizing the results of \cite{Caraiani-Scholze17}, \cite{Koshikawa-generic}, and \cite{Hamann-Lee}), it implies that the \'etale cohomology of certain unitary Shimura varieties with coefficients in $\Lambda$ is concentrated in middle degree. We note that there are Eisenstein maximal ideals $\fm$ that are cohomologically generic, and our result applies to that cases. We expect that the statement of Conjecture \ref{conj-Ihara} still holds true with $\fm$ non-Eisenstein replaced by $\fm$ cohomologically generic. 
    
	Conditions (i) and (ii) together ensure that the stack of local Langlands parameters has simple geometry after localizing at $\rho_\fm|_{W_{\QQ_p}}$. These conditions could potentially be weakened with a deeper understanding of the geometry of the stack of local Langlands parameters.
	
	Condition (iii) is essential for our method. It allows us to find an indefinite inner form $\sG_1$ of $\sG$ such that $\sG_1(\AAA_f)\simeq \sG(\AAA_f)$ with the signatures changing only at $v$. See Lemma \ref{lemma-pure-inner-form}.
\end{rmk}

\begin{rmk}
    Using the method developed in this paper, we can prove generalized Ihara's lemma for Shimura varities associated to indefinite unitary groups under similar assumptions. 
    
    We also expect that the method can be applied to prove analogous versions of Ihara's lemma and (arithmetic) level raising results when the prime $p$ is \emph{inert} in $\sK$. Such results have important arithmetic applications, as in \cite{LTXZZ} and \cite{LTX-Iwasawa}.
\end{rmk}

\begin{rmk}
    Under different assumptions on the prime $\ell$ and the maximal ideal $\fm$, this generalized Ihara's lemma is studied by Boyer in \cite{Boyer-ihara-level-raising} and \cite{Boyer-ihara-limit}.
\end{rmk}

To prove Theorem \ref{thm-main-intro}, we present an alternative formalization of the generalized Ihara's lemma, analogous to the original Ihara's lemma for $\GL_2$. Let $\hat{G}=\GL_n$ be the dual group over $\Lambda$. Let $\hat{T}\subseteq \hat{G}$ be the maximal torus of diagonal matrices, and $W$ be the Weyl group. There is a Chevalley isomorphism $\hat{G}\git\hat{G}\simeq \hat{T}\git W$. Let $\xi_{\fm,v}\in (\hat{T}\git W)(\Lambda)$ denote the conjugacy class of the element $q^{\frac{1-n}{2}}\cdot\rho_\fm(\Frob_v)$. Let $I\subseteq \GL_n(O_{\sF_v})$ denote the Iwahori subgroup consisting of elements that are upper-triangular modulo $v$. Using Bernstein presentation of the Iwahori--Hecke algebra, we can define a decomposition (See Proposition \ref{prop-decomp-R})
$$\cA(K^v)^I_\fm=\bigoplus_{x}\cA(K^v)^{I,x}_\fm,$$
where $x$ ranges points in the preimage of $\xi_{\fm,v}$ in $\hat{T}$. For each such point $x$ appearing in the decomposition, we can define an associated index set $P_x$ (Definition \ref{def-P_x}). The main result is as follows.

\begin{thm}[Theorem \ref{thm-main}]\label{thm-main2-intro}
	Assume that we are in the situation of Theorem \ref{thm-main-intro}. Let $x,x'\in \hat{T}(\Lambda)$ be two lifts of $\xi_{\fm,v}$ such that $P_{x'}\subseteq P_x$. 
	Then we have:
	\begin{enumerate}
		\item There is a surjective map
			$$\alpha_{x,x'}\colon\cA(K^v)_\fm^{I,x}\to \cA(K^v)_\fm^{I,x'}$$
			defined by some Iwahori--Hecke operator.
		\item There is an injective map
			$$\beta_{x,x'}\colon\cA(K^v)_\fm^{I,x'}\to \cA(K^v)_\fm^{I,x}$$
			defined by some Iwahori--Hecke operator.
	\end{enumerate}
\end{thm}

Theorem \ref{thm-main-intro} follows from Theorem \ref{thm-main2-intro} via an explicit computation using the Bernstein--Zelevinsky classifications. As an arithmetic application of Theorem \ref{thm-main2-intro}, we establish level-raising results for automorphic representations of $\sG(\AAA_f)$; see Theorem \ref{thm-level-raising} for details. Roughly speaking, let $\pi$ be an automorphic representation of $\sG(\AAA_f)$ with coefficients in $\overline\QQ_\ell$ that is unramified at $v$. Let $\rho_\pi$ denote the Galois representation associated to $\pi$. If the semisimple reduction $\overline{\rho}_\pi$ of $\rho_\pi$ satisfies the conditions in Theorem \ref{thm-main-intro}, then level-raising at $v$ is possible whenever there is no local obstruction.

We give an overview of the proof of  Theorem \ref{thm-main2-intro}. The basic strategy is not new: we aim to use torsion vanishing of \'etale cohomology of Shimura varieties, along with basic uniformization, to prove Ihara's lemma or level-raising type results. This approach has been employed in earlier works, such as \cite{Thorne-level-raising}, \cite{LTXZZ}, and \cite{LTX-Iwasawa}. In those works, the results were achieved by carefully studying geometry of the special fibers of certain Shimura varieties. In this paper, we introduce a new method. 

In \cite{Tame}, Xinwen Zhu established the unipotent categorical local Langlands correspondences for unramified groups (reviewed in \S\ref{subsection-ucllc}). In particular, the category of unipotent $\GL_n(\sF_v)$-representations admits a fully faithfully embedding into the category of coherent sheaves on the stack $\Loc^{\widehat{\unip}}_{{}^LG,\sF_v}$ of unipotent local Langlands parameters. After localizing at regular semisimple parameter, we can explicitly compute the geometry of $\Loc^{\widehat{\unip}}_{{}^LG,\sF_v}$. This enables us to reformulate Theorem \ref{thm-main2-intro} in terms of coherent sheaves on the stack $\Loc^{\widehat{\unip}}_{{}^LG,\sF_v}$.

By Condition (3) in Theorem \ref{thm-main-intro}, there exists a pure inner form $\sG_1$ of $\sG$ such that there are isomorphisms $\sG_1(\AAA_f)\simeq\sG(\AAA_f)$ and 
$$\sG_1(\RR)\simeq \rG(\rU(1,n-1)\times\rU(n-1,1)\times \rU(0,n)^{[\sF^+:\QQ]-2}).$$
Let $\Sh_K(\sG_1,\sX_1)$ denote the Shimura variety associated to $\sG_1$ (with Iwahori level at $v$). According to \cite{Igusa}, there is a perfect Igusa stack $\Igs^1_{K^v}$ associated to $\Sh_K(\sG_1,\sX_1)$. By \cite[\S 6]{Tame} (See also \cite[\S 4]{YZ-torsion-vanishing}), we can express the \'etale cohomology of $\Sh_K(\sG_1,\sX_1)$ using coherent sheaves on the stack $\Loc^{\widehat{\unip}}_{{}^LG,\sF_v}$. By basic uniformization, the basic locus of $\Igs^1_{K^v}$ is precisely given by the Shimura set associated to $\sG$. Using this, we are able to compare the cohomology of $\Sh_K(\sG_1,\sX_1)$ and $\cA(K^v)$ directly using coherent sheaves on the stack $\Loc^{\widehat{\unip}}_{{}^LG,\sF_v}$. In particular, the cones of the maps in Theorem \ref{thm-main2-intro} can be computed using the cohomology of $\Sh_K(\sG_1,\sX_1)$. This can be viewed as a geometric Jacquet--Langlands correspondence. Finally, applying the main result of \cite{YZ-torsion-vanishing} on torsion vanishing of Shimura varieties, we conclude that the cones of the maps in Theorem \ref{thm-main2-intro} are concentrated in desired cohomological degrees. 

The organization of this paper is as follows: 

In \S\ref{section-rep}, we setup the notations of automorphic forms and recall some properties of representations of $p$-adic groups. We then introduce our main Theorem \ref{thm-main}(=Theorem \ref{thm-main2-intro}).

In \S\ref{section-CLLC}, we construct the Iwahori--Hecke operators needed in Theorem \ref{thm-main}. To do this, we use unipotent categorical local Langlands correspondences developed in \cite{Tame}.

In \S\ref{section-proof}, we prove our main result Theorem \ref{thm-main}. The main tool is the local-global compatibility results on the \'etale cohomology of Shimura varieties developed in \cite[\S 6]{Tame}.

In \S\ref{section-coh-sheaf}, we study the coherent sheaf associated to $\cA(K^v)_\fm$ under the unipotent categorical local Langlands correspondence. Using methods similar to those in \S\ref{section-proof}, we prove that it is a genuine coherent sheaf that is determined by its global sections.

In \S\ref{section-level-raising}, we prove some level-raising results, which follows directly from Theorem \ref{thm-main}.

\subsection{Notations and Conventions}
We use the language of $\infty$-categories. When we say ``category'', we always mean $(\infty,1)$-category unless otherwise stated. In particular, if $R$ is a commutative ring, then $R\mbox{-}\mathrm{mod}$ is the stable $\infty$-category of complexes of $R$-modules, and $R$-linear categories are stable $\infty$-categories enriched over $R\mbox{-}\mathrm{mod}$. All the functors are derived.

However, if we have an ordinary commutative ring $R$ with some elements $x_i\in R, i\in I$, we denote by $R/(x_i)_{i\in I}$ the \emph{classical} quotient ring of $R$ by the ideal $(x_i)_{i\in I}$. We will not use derived quotient ring (usually denoted by $R/^L(x_i)_{i\in I}$) in this paper.

\subsection{Acknowledgements}
I sincerely thank my advisor, Liang Xiao, for his invaluable guidance and constant encouragement. I am especially thankful to Xinwen Zhu for sharing an early draft of his work \cite{Tame} and for very helpful discussions. I thank Yuanyang Jiang for very helpful discussions.

Part of this work was completed while I was visiting Princeton University and Stanford University, and I sincerely thank them for their hospitality.

\section{Representation theoretic preparation}\label{section-rep}
\subsection{Unitary similitude groups}\label{subsection-def-auto}
Let $\sF^+$ be a totally real field. Let $\sF/\sF^+$ be a CM quadratic extension. Assume that $\sF$ contains an imaginary quadratic field $\sK$. We denote by $c\colon \sF\simeq \sF$ the complex conjugation. Let $\sD$ be a central simple $\sF$-algebra of dimension $n^2$. Write $\sV=\sD^\op$ viewed as a left $\sD^\op$-module. Then $\sD=\mathrm{Aut}_{\sD^\op}(\sV)$. Let $\star$ be a positive involution on $\sD^\op$. Let 
$$\langle-,-\rangle\colon \sV\times\sV\to \QQ$$
be a non-degenerated $\star$-Hermitian form on $\sV$ such that $\langle v,w\rangle=-\langle w,v\rangle$ and  $\langle bv,w\rangle=\langle v,b^\star w\rangle$ for $b\in\sD^\op$, $v,w\in \sV$. Then $(\sD^\op,\star,\sV,\langle-,-\rangle)$ defines a PEL structure of type A. The pairing $\langle-,-\rangle$ defines an involution $\dagger$ on $\sD$ that is compatible with the complex conjugation on $\sF$. Define $\sG=\GU(\sD,\dagger)$ to be the unitary similitude group over $\QQ$ whose points in any $\QQ$-algebra $R$ are given by 
$$\sG(R)\coloneqq\{x\in \sD\otimes_\QQ R| x x^\dagger\in R^\times\}.$$
Define the unitary group $\sG^1$ over $\sF^+$ whose points in any $\sF^+$-algebra $R$ are given by
$$\sG^1(R)\coloneqq \{x\in\sD\otimes_{\sF^+}R| xx^\dagger=1\}.$$
There is a short exact sequences 
$$1\to \Res_{\sF^+/\QQ}\sG^1\to \sG\to \GG_m\to 1$$
of groups over $\QQ$. Here $\Res_{\sF^+/\QQ}\sG^1$ is the Weil restriction of $\sG^1$ from $\sF^+$ to $\QQ$.

Fix a complex embedding $\tau_0\colon \sK\hookrightarrow \CC$.
Let $\Sigma_\infty$ be the set of real embeddings of $\sF^+$. For each $\tau\in \Sigma_\infty$, together with $\tau_0$, we obtain a complex embedding $\bar\tau\colon \sF\hookrightarrow \CC$. Let $(a_\tau,n-a_\tau)$ with $ 0\leq a_\tau\leq n$ be the signature of $(\sD,\dagger)$ at $\tau\in\Sigma_\infty$. Denote $J_{a_\tau}=\diag(\underbrace{1,\dots,1}_{a_\tau},\underbrace{-1,\dots,-1}_{n-a_\tau})\in \GL_n(\CC)$. We define the unitary similitude group over $\RR$ with signatures $(a_\tau,n-a_\tau)_{\tau\in \Sigma_\infty}$:
$$\rG(\prod_{\tau\in\Sigma_\infty}\rU(a_\tau,n-a_\tau))\coloneqq\bigg\{(x_\tau)\in \prod_{\tau\in\Sigma_\infty}\GL_n(\CC)\bigg| xJ_{a_\tau}{}^t\bar{x}=s\cdot J_{a_\tau} \text{ for some $s\in\RR^\times$ independent of $\tau$}\bigg\}.$$
Here we identify $\rU(a_\tau,n-a_\tau)$ as a subgroup of $\GL_n(\CC)$ via the embedding $\bar\tau$.
There is an isomorphism of real Lie groups
$$\sG(\RR)\simeq \rG(\prod_{\tau\in \Sigma_\infty}\rU(a_\tau,n-a_\tau)).$$
We note that $\sG(\RR)$ is compact modulo center if and only if $a_\tau=0$ or $a_\tau=n$ for each $\tau\in \Sigma_\infty$.

Let $\bS=\Res_{\CC/\RR}\GG_m$ be the Deligne torus. Consider the homomorphism
$$h\colon \bS\to \sG_\RR,\quad z\mapsto \prod_{\tau\in \Sigma_\infty}\mathrm{diag}(\underbrace{z,\dots,z}_{a_\tau},\underbrace{\bar{z},\dots,\bar{z}}_{n-a_\tau}).$$
Let $\sX$ be set of $\sG(\RR)$-conjugacy classes of $h$. Then $\sX$ is endowed with a natural complex analytic structure. The pair $(\sG,\sX)$ form a Shimura datum of PEL type. Let $\mu\colon\GG_{m,\CC}\to \sG_\CC$ be the Hodge cocharacter defined as the composition of $h$ with the map $\GG_{m,\CC}\to \bS_\CC\cong\GG_{m,\CC}\times\GG_{m,\CC}, z\mapsto (z,1)$. Under the  isomorphism $\sG_\CC\simeq \GG_{m,\CC}\times\prod_{\tau\in\Sigma_\infty}\GL_{n,\CC}$, the Hodge cocharacter $\mu$ is given by
$$\mu(z)= (z,\prod_{\tau\in\Sigma_\infty}\mathrm{diag}(\underbrace{z,\dots,z}_{a_\tau},\underbrace{1,\dots,1}_{n-a_\tau})).$$
Denote by $\sE=\sE_\mu\subset \CC$ the reflex field of $\mu$, i.e. the minimal subfield of $\CC$ such that the conjugacy class of $\mu$ is defined. For each \emph{neat} open compact subgroup $K\subset \sG(\AAA_f)$, we have a Shimura variety $\sSh_K(\sG,\sX)$ defined over $\sE$. The $\CC$-points of $\sSh_K(\sG,\sX)$ is identified with the double quotient set
$$\sG(\QQ)\backslash\sX\times\sG(\AAA_f)/K.$$
The group $\sG$ is anisotropic modulo center. Therefore $\sSh_K(\sG,\sX)$ is proper smooth over $\sE$.

Fix a prime number $p$ that splits in $\sK$. Let $v$ be a finite place of $\sF$ above $p$. Denote by $\wp$ (resp. $v^+$) the place of $\sK$ (resp. $\sF^+$) lying below $v$. Then $v^+$ splits into $v$ and $v^c$ in $\sF$. Let $\Pi_p$ denote the set of places of $\sF$ lying above $\wp$. We have an isomorphism
$$\sD\otimes_\QQ\QQ_p\cong  \prod_{v'\in\Pi_p} (\sD_{\sF_{v'}}\times \sD_{\sF_{v'^c}}).$$
Here if $v'$ is a place of $\sF$, we denote by $v'^c$ the complex conjugate of $v'$. The involution $\dagger$ on $\sD$ defines isomorphisms $\sD_{\sF_{v'}}\simeq \sD^\op_{\sF_{v'}}\otimes_{\sF_{v'},c}\sF_{v'^c} $. Thus there is an isomorphism
$$\sG_{\QQ_p}\cong \GG_m\times \prod_{v'\in \Pi_p}\Res_{\sF_{v'}/\QQ_p}\sD_{\sF_{v'}}^\times$$
We make the following assumptions:
\begin{assump}
	\begin{enumerate}
	\item The group $\sG(\RR)$ is compact modulo center. In other words, the signature $a_\tau$ is either $0$ or $n$ for any real embedding $\tau\in \Sigma_\infty$.
	\item $v$ is unramified over $p$. 
	\item The central simple algebra $\sD$ splits at $v$ and $v^c$. We fix an isomorphism $\sD_{\sF_v}\simeq \rM_n(\sF_v)$.
\end{enumerate}
\end{assump}

Denote $F=\sF_{v}$. The above assumptions induce an isomorphism
$$\sG_{\QQ_p}\cong \Res_{F/\QQ_p}\GL_n\times H$$
where $H\coloneqq \GG_m\times\prod_{v'\in \Pi_p\backslash\{v\}}\Res_{\sF_{v'}/\QQ_p}\sD_{\sF_{v'}}^\times$. By Assumption (1) above, the locally symmetric space $\sX$ is trivial. Hence the Shimura variety $\sSh_K(\sG,\sX)$ is 0-dimensional, and its $\CC$-points are given by the finite discrete set $$\sSh_K(\sG,\sX)(\CC)= \sG(\QQ)\backslash \sG(\AAA_f)/K.$$

Assume that the open compact subgroup $K$ factorizes as $K=\prod_{r}K_r$ where $r$ runs through finite places of $\QQ$ and $K_r\subset\sG(\QQ_r)$ is open compact. We further assume that $K_p=K_vK_p^v$ for open compact subgroups $K_v\subset \GL_n(F)$ and $K_p^v\subset H(\QQ_p)$. Denote $K^p=\prod_{r\neq p}K_r$ and $K^v=K^pK_p^v$. We call $K^v$ a prime-to-$v$ level group of $\sG$.

Fix a prime number $\ell\neq p$. Let $\Lambda=\overline{\FF}_\ell$. 

\begin{defn}
	Fix a neat prime-to-$v$ level $K^v\subseteq \sG(\AAA_f^p)\times H(\QQ_p)$. We define the space of automorphic forms of $\sG$ with coefficients in $\Lambda$ and prime-to-$v$-level $K^v$ to be
	$$\cA=\cA(K^v)\coloneqq H^0(\sG(\QQ)\backslash \sG(\AAA_f)/{K^v},\Lambda).$$
	The space $\cA$ carries a smooth action of $\GL_n(F)$. If $K_v\subset \GL_n(F)$ is an open compact subgroup, we denote by
	$$\cA^{K_v}=H^0(\sG(\QQ)\backslash\sG(\AAA_f)/K_vK^v,\Lambda)$$
	 the subspace of $K_v$-invariant vectors.
\end{defn}

Fix a finite set $S$ of primes of $\QQ$ such that $\sF$ and $\sG$ is unramified away from $S$, and the open compact subgroup $K^S=\prod_{r\notin S}K_r$ is a product of hyperspecial subgroups. Also assume that $S$ contains $p$ and $\ell$. Define the abstract Hecke algebra
$$\TT^S\coloneqq\bigotimes'_{r\notin S} \Lambda [K_r\backslash \sG(\QQ_r)/K_r].$$
Note that $\TT^S$ is a commutative $\Lambda$-algebra under the convolution product, and $\TT^S$ acts on $\cA$ by convolution. 

We need to define some Hecke operators. Let $r\notin S$ be a prime number that splits in $\sK$. Fix a place $w$ of $\sF$ above $r$. Thus we have an isomorphism
$$\sG(\QQ_r)\simeq \QQ_r^\times\times \prod_{w'}\GL_n(\sF_{w'}),$$
where $w'$ runs through places of $\sF$ above $r$ lying over the same place of $\sK$ as $w$. We can assume 
$$K_r=\ZZ_r^\times\times \prod_{w'}\GL_n(O_{\sF_{w'}}).$$
For $1\leq i\leq n$, denote by $T_{i,w}\in \Lambda[K_r\backslash \sG(\QQ_r)/K_r]$ the characteristic function on
$$\ZZ_r^\times\times \GL_n(O_{\sF_{w}})\diag(\underbrace{\varpi_{w},\dots\varpi_{w}}_{i},\underbrace{1,\dots,1}_{n-i})\GL_n(O_{\sF_{w}}) \times\prod_{w'\neq w}\GL_n(O_{\sF_{w'}}).$$

For any maximal ideal $\fm\subset \TT^S$ such that the localization $\cA_\fm$ is non-zero, we can attach a Galois representation.

\begin{thm}\label{thm-assoc-Galois}
	Let $\fm\subset \TT^S$ be a maximal ideal such that the localization $\cA_\fm$ is non-zero. Then there is a (unique) semisimple Galois representation
	$$\rho_\fm\colon \Gal(\overline{\sF}/\sF)\to \GL_n(\Lambda)$$
	unramified outside the places above $S$, such that $\rho_\fm^c\simeq \rho_\fm^\vee \otimes \varepsilon^{1-n}$ where $\varepsilon$ is the cyclotomic character.
    For an prime number $r\notin S$ that splits in $\sK$ and a place $w$ of $\sF$ above $r$, the characteristic polynomial of $\rho_\fm(\Frob_w)$ is given by the image of
	$$X^n- T_{1,w}X^{n-1}+\cdots + (-1)^i q_w^{\frac{i(i-1)}{2}}T_{i,w}X^{n-i}+\cdots + (-1)^n q_w^{\frac{n(n-1)}{2}}T_{n,w}$$
	in $\Lambda[X]$. Here $\Frob_w$ is the geometric Frobenius at $w$, $q_w$ is the cardinality of the residue field of $w$, and $T_{i,w}\in \TT^S$ are Hecke operators defined above.
\end{thm}
\begin{proof}
	This follows from \cite[Proposition 3.4.2]{CHT08} by restricting the automorphic forms to $\Res_{\sF^+/\QQ}\sG^1$. 
\end{proof}

\subsection{Recollection on representations of $\GL_n(F)$ with banal coefficients}\label{subsection-local-rep}
Recall the notation $F=\sF_v$. Let $q$ be the cardinality of the residue field of $F$. Let $\varpi\in O_F$ be a uniformizer.

In \S\ref{subsection-local-rep} and \S\ref{subsection-irrep}, we assume either $\Lambda=\overline\FF_\ell$ for $\ell\nmid |\GL_n(\FF_q)|$, or $\Lambda=\overline{\QQ}_\ell$. We fix a square root $\sqrt{q}\in\Lambda$ of $q$ from now on. 

Denote $G=\GL_{n,F}$. Let $B\subset G$ be the Borel subgroup of upper-triangular matrices. Let $T\subset B$ be the maximal torus of diagonal matrices. Let $\XX^\bullet(T)$ (resp. $\XX_\bullet(T)$) be the weight (resp. coweight) lattice of $T$. Let $\XX_\bullet^+(T)$ (resp. $\XX_\bullet^-(T)$) be the set of dominant (resp. anti-dominant) elements in $\XX_\bullet(T)$ associated to $B$. Let $N\subset B$ be the unipotent radical. Let $W=N_G(T)/T$ be the Weyl group of $G$. Note that $W$ is isomorphic to the symmetric group $\Sigma_n$. If $\mu\colon \GG_m\to T$ is a cocharacter, we denote $\varpi^\mu\coloneqq \mu(\varpi)\in T(F)$. Let $\hat{G}=\GL_n$ be the dual group of $G$ over $\Lambda$. Let $\hat{T}\subset \hat{B}\subset \hat{G}$ be the dual Borel subgroup and dual maximal torus. Let $\varepsilon_i\in \XX^\bullet(\hat{T})=\XX_\bullet(T)$ for $i=1,\dots,n$ be the fundamental characters. Let $\rho\in \frac{1}{2}\XX^\bullet(T)$ be the half sum of positive roots. 

Let $G(O_F)$ be the hyperspecial subgroup of $G(F)$. Denote by 
$$\cH_\sph\coloneqq\Lambda[G(O_F)\backslash G(F)/G(O_F)]$$ 
the \emph{spherical Hecke algebra}. Recall the Satake isomorphism
$$\cH_\sph \cong \Lambda[\XX_\bullet(T)]^W\cong \cO(\hat{T}\git W),$$
where $\hat{T}\git W\simeq \hat{G}\git \hat{G}$ is the GIT quotient of $\hat{G}$ by the adjoint action over $\Lambda$. Let $\wedge^i\std$ be $i$-th exterior product of the standard representation of $\GL_n$. Then $\wedge^i\std$ is the highest weight representation of $\GL_n$ with highest weight $\mu_i=(1^i,0^{n-i})$. Let $\chi_{\wedge^i\std}\in\cO(\hat{T}\git W)$ denote the character of $\wedge^i\std$. Under the Satake isomorphism, the Hecke operator
$$T_i=\text{characteristic function on }G(O_F) \varpi^{\mu_i} G(O_F)$$
is sent to the function $q^{\frac{i(n-i)}{2}}\chi_{\wedge^i\std}\in \cO(\hat{T}\git W)$.

We have the following local-global compatibility result.

\begin{prop}\label{prop-local-global-comp-hyperspecial}
	Assume that the localization $\cA^{G(O_F)}_\fm$ is non-zero. Then the representation $\rho_\fm\colon \Gal(\overline{\sF}/\sF)\to \GL_n(\Lambda)$ is unramified at $v$. Let $\xi_{\fm,v}\in (\hat{T}\git W)(\Lambda)$ be the conjugacy class of $q^{\frac{1-n}{2}}\cdot\rho(\Frob_v)^\mathrm{ss}$. Then as an $\cH_\sph$-module, the space $\cA_\fm^{G(O_F)}$ is set theoretically supported on the point $\xi_{\fm,v}$ under the Satake isomorphism. 
\end{prop}
\begin{proof}
	It follows from \cite[Proposition 3.4.4]{CHT08}.
\end{proof}

Define the Iwahori subgroup $I\subset G(O_F)$ as the preimage of $B(\FF_q)$ under the projection $G(O_F)\to G(\FF_q)$. Denote by
$$\cH\coloneqq \Lambda[I\backslash G(F)/I]$$
the Iwahori--Hecke algebra of $G(F)$. Let $\widetilde{W}\coloneqq N_G(T)(F)/ T(O_F)\cong \XX_\bullet(T)\rtimes W$ be the extended affine Weyl group of $G$. The image of a coweight $\lambda\in \XX_\bullet(T)$ in $\widetilde{W}$ is given by the class of $\varpi^\lambda\in T(F)$. For each $w\in \widetilde{W}$, denote by $IwI$ the $I$-double coset in $G(F)$ corresponding to $w$. Let $T_w=1_{IwI}$ be the characteristic function on $IwI$. Then $T_w$ for $w\in\widetilde{W}$ form a $\Lambda$-basis of $\cH$.

We recall the Bernstein presentation of $\cH$.

\begin{defn}[Bernstein presentation]
	For $\lambda\in \XX_\bullet^-(T)$ anti-dominant, define the element $\theta_\lambda= q^{\langle\rho,\lambda\rangle} 1_{I\varpi^\lambda I}$. In general, if $\lambda\in \XX_\bullet(T)$, we can write $\lambda=\mu-\nu$ for $\mu$ and $\nu$ anti-dominant. Then we define $\theta_\lambda=\theta_\mu\theta_\nu^{-1}$.
\end{defn}
 
\begin{prop}\label{prop-Bernstein-pres}
	For $\lambda\in \XX_\bullet(T)$, the elements $\theta_\lambda$ are well-defined. Let $\cR\subset \cH$ be the $\Lambda$-subalgebra generated by $\theta_\lambda$ for $\lambda\in \XX_\bullet(T)$. Then $\cR$ is commutative and is isomorphic to $\cO(\hat{T})=\Lambda[\XX_\bullet(T)]$. Moreover, under this isomorphism, the center $Z(\cH)$ of $\cH$ is identified with $\Lambda[\XX_\bullet(T)]^W=\cO(\hat{T}\git W)$.
\end{prop}
\begin{proof}
	See for example, \cite[\S 5.1]{Achar-Riche}. The isomorphism $\cR\cong \Lambda[\XX_\bullet(T)]$ is given by $\theta_\lambda\mapsto e^\lambda$. We note that the proposition holds when $\Lambda$ is an arbitrary $\ZZ[\frac{1}{p},\sqrt{q}]$-algebra.
\end{proof}

The local-global compatibility can be extended to the Iwahori level. Let $\Rep^{\widehat\unip}(G(F),\Lambda)$ be the category of unipotent $G(F)$-representations over $\Lambda$ as in \cite[Definition 4.116]{Tame}. Since $\ell$ is a banal prime and $G(F)=\GL_n(F)$, the object $\cInd_I^{G(F)}\Lambda$ is a compact projective generator of $\Rep^{\widehat\unip}(G(F),\Lambda)$. More concretely, a representation $V\in \Rep(G(F),\Lambda)$ is unipotent if and only if each cohomology $H^i(V)$ is generated by its $I$-invariant elements.

Recall that an object $V\in \Rep(G(F),\Lambda)$ is called \emph{admissible} if $\Hom_{\Rep(G(F),\Lambda)}(C,V)$ is a perfect $\Lambda$-module for any compact objects in $\Rep(G(F),\Lambda)$, cf. \cite[Example 7.31 (4)]{Tame}. We denote by $\Rep(G(F),\Lambda)^\Adm$ the category of admissible representations. In particular, an unipotent representation $V\in\Rep^{\widehat\unip}(G(F))$ is admissible if and only if $V^I$ is a perfect $\Lambda$-module. 

We define the category of unipotent admissible representations with fixed central characters.
\begin{defn}\label{def-rep-localize}
	Let $\xi\in(\hat{T}\git W)(\Lambda)$ be a $\Lambda$-point. Define
	$$\Rep^{\widehat\unip}(G(F),\Lambda)^\Adm_{\xi}$$
	to be the category of $G(F)$-representations $V$ such that the following conditions hold:
	\begin{enumerate}
		\item $V$ is admissible and unipotent.
		\item When viewed as a $Z(\cH)$-module,  $V^I$ is set theoretically supported at the point $\xi$. 
	\end{enumerate}
	Note that $\Rep^{\widehat\unip}(G(F),\Lambda)^\Adm_{\xi}$ is a block in the category of unipotent admissible $G(F)$-representations.
\end{defn}

Since $\Rep^{\widehat\unip}(G(F),\Lambda)$ has a compact generator $\cInd_I^{G(F)}\Lambda$, the Bernstein center $Z_{G(F)}^{\widehat\unip}$ of $\Rep^{\widehat\unip}(G(F),\Lambda)$ is identified with $Z(\cH)$. In Definition \ref{def-rep-localize}, the condition (2) is equivalent to that the action of $Z_{G(F)}^{\widehat\unip}$ on $V$ is set theoretically supported on $\xi$. Therefore the above definition agree with the definition in \cite[\S 3.2.3]{YZ-torsion-vanishing}. We record the following useful lemma.

\begin{lemma}\label{lemma-adm-unip-imply-compact}
	Admissible objects in $\Rep^{\widehat\unip}(G(F),\Lambda)$ are compact.
\end{lemma}
\begin{proof}
	Let $V\in\Rep^{\widehat\unip}(G(F),\Lambda)^\Adm$. We may assume that $V$ lies in the heart of $\Rep^{\widehat\unip}(G(F),\Lambda)^\Adm$. There is a surjective morphism
	$$V^I\otimes_\Lambda\cInd_I^{G(F)}\Lambda\twoheadrightarrow V$$
	of $G(F)$-representations. Let $W$ be the kernel of the above map. Taking $I$-invariance induces a short exact sequence
	$$0\to W^I\to V^I\otimes_\Lambda\cH\to V^I\to 0$$
	of $\cH$-modules. As $\cH$ is a finitely generated module over $Z(\cH)$, we see that $W^I$ is a finitely generated module over $Z(\cH)$. Pick a finite set of generators $S$ of $W^I$ over $Z(\cH)$, we obtain a surjective morphism 
	$$\Lambda[S]\otimes_\Lambda \cInd_I^{G(F)}\Lambda\twoheadrightarrow W.$$
	Continuing this process, we construct a projective resolution $C_\bullet$ of $V$ with each term equals to a finite direct sum of $\cInd_I^{G(F)}\Lambda$. By \cite[Lemma 3.59]{Tame}, the $\ell$-cohomological dimension of $G(F)$ is finite, as $\ell$ is banal for $G(F)$. Say the $\ell$-cohomological dimension is $d$. Let $D_d$ be the kernel of $C_{d-1}\to C_{d-2}$. Thus $D_d$ is projective in $\Rep^{\widehat\unip}(G(F),\Lambda)^\heartsuit$ by the standard homological algebra argument. The finite resolution
	$$0\to D_d\to C_{d-1}\to \cdots\to C_0\to V\to 0$$
	satisfies that each term $C_i$ with $i\in 0,\dots,d-1$ is a finite direct sum of $\cInd_I^{G(F)}\Lambda$, and $D_d$ is a direct summand of a finite direct sum of $\cInd_I^{G(F)}\Lambda$. It follows that $V$ is compact in $\Rep^{\widehat\unip}(G(F),\Lambda)$.
\end{proof}

\begin{prop}\label{prop-local-global-auto}
	If $\cA_\fm^{G(O_F)}$ is non-zero, then $\cA_\fm$ lies in the category $\Rep^{\widehat\unip}(G(F),\Lambda)^\Adm_{\xi_{\fm,v}}$ with $\xi_{\fm,v}$ defined in Proposition \ref{prop-local-global-comp-hyperspecial}.  
\end{prop}
\begin{proof}
Let $\kappa$ be a finite extension of $\FF_\ell$ such that the Galois representation $\rho_\fm$ is defined with coefficients in $\kappa$. Let $L$ be a finite extension of $\QQ_\ell$ with residue field $\kappa$. Consider the space of automorphic forms for $\sG$ with coefficients in $O_L$:
	$$\cA_{O_L}= H^0(\sG(\QQ)\backslash \sG(\AAA_f)/{K^v},O_L).$$
	Let $\TT^S_{O_L}=\bigotimes'_{r\notin S} O_L [K_r\backslash \sG(\QQ_r)/K_r]$ be the global Hecke algebra with $O_L$-coefficients. Then $\fm$ can be viewed as a maximal ideal of $\TT^S_{O_L}$. For any open compact subgroup $K_v\subset \GL_n(F)$, the space $\cA^{K_v}_{O_L,\fm}$ of $K_v$-invariant elements is a finite free $O_L$-module. 
	
	Fix a compact open subgroup $K_v\subset \GL_n(F)$. Denote by $\TT^{S,\mathrm{im}}_{O_L}(K_v)$ the image of $\TT_{O_L}^S\to \End_{O_L}(\cA^{K_v}_{O_L})$. Let $\fp\subset \TT^{S,\mathrm{im}}_{O_L}(K_v)$ be a minimal ideal. It cuts out an automorphic representation
	$$\pi=\boxtimes'_{r\nmid \infty}\pi_r$$
	of $\sG(\AAA_f)=\prod'_{r\nmid\infty}\sG(\QQ_r)$ with coefficients in $\overline\QQ_\ell$. We can further write 
	$$\pi_p=\pi_v\boxtimes\pi_p^v$$ 
	as a representations of $\sG(\QQ_p)=\GL_n(F)\times H(\QQ_p)$. After enlarging $L$ if necessary, by \cite[Proposition 3.3.4]{CHT08}, there is a Galois representation 
	$$\rho_\fp\colon \Gal(\overline{\sF}/\sF)\to \GL_n(L)$$
	associated to $\pi$ which is compatible with local Langlands correspondence for $\GL_n$ at split places. In particular, 
	let $\rec(\pi_v)\colon \Gal(\overline{\sF}_v/\sF_v)\to\GL_n(L)$ be the Galois representation associated to $\pi_v$ under local Langlands correspondence. Then there is an isomorphism
	$$(\rho_\fp|_{\Gal(\overline{\sF}_v/\sF_v)})^{\mathrm{ss}}\simeq \rec(\pi_v)^\mathrm{ss}\otimes \varepsilon^{\frac{1-n}{2}}$$
	after semisimplifications.
	
	By \cite[Proposition 3.4.2]{CHT08}, for any minimal prime ideal $\fp\subset \TT^{S,\mathrm{im}}_{O_L}(K_v)$ contained in $\fm$, the Galois representation $\rho_\fm\colon \Gal(\overline{\sF}/\sF)\to\GL_n(\Lambda)$ is isomorphic to the semisimple reduction of $\rho_\fp$. By Proposition \ref{prop-local-global-comp-hyperspecial}, $\rho_\fm$ is unramified at $v$. Let $I_F\subset \Gal(\overline{F}/F)$ be the inertia subgroup and $P_F\subset I_F$ be the wild inertia subgroup. Since $\ell$ is a banal prime, the representation $\rho_\fp|_{\Gal(\overline{F}/F)}$ is trivial on $P_F$ and sends elements in $I_F$ to unipotent matrices. Thus the representation $\pi_v$ is unipotent. By \cite[Corollary 3.3.3]{CHT08}, the $G(F)$-module $\cA_{O_L,\fm}\otimes_{O_L} \overline\QQ_\ell$ is semisimple. Therefore $\cA_{O_L,\fm}\otimes_{O_L} \overline\QQ_\ell$ is unipotent as all the irreducible submodules of $\cA_{O_L,\fm}\otimes_{O_L} \overline\QQ_\ell$ is unipotent. By \cite[1.5]{Vig01}, any irreducible subquotient of $\cA_\fm$ is unipotent. Therefore $\cA_{\fm}$ is unipotent.
	
	Let $\cH_{O_L}$ be the Iwahori--Hecke algebra with coefficients in $O_L$. Fix a minimal prime ideal $\fp$ of $\TT^{S,\text{im}}_{O_L}$ contained in $\fm$. By local-global compatibility in characteristic 0 coefficients, we know that the action of $Z(\cH_{O_L})$ on $\pi_v$ factors through the homomorphism $\xi_{\fp,v}\colon Z(\cH_{O_L})\to \overline\QQ_\ell$ defined by the conjugacy classe of $q^{\frac{1-n}{2}}\rho_\fp(\Frob_v)$ 
	in $(\hat{T}\git W)(\overline\QQ_\ell)$. By congruence, the homomorphism $\xi_{\fp,v}\colon Z(\cH_{O_L})\to \overline\QQ_\ell$ factors through $Z(\cH_{O_L})_{(\xi_{\fm,v})}$, where $Z(\cH_{O_L})_{(\xi_{\fm,v})}$ is the localization of $Z(\cH_{O_L})$ at the maximal ideal defining $\xi_{\fm,v}$. As $\fp$ runs through minimal prime ideals of $\TT^{S,\text{im}}_{O_L}$ lying below $\fm$, we see that the action of $Z(\cH_{O_L})$ on $\cA^{I}_{O_L,\fm}$ factors through the localization $Z(\cH_{O_L})_{(\xi_{\fm,v})}$. Since $\cA^{I}_{\fm}$ is finite-dimensional over $\Lambda$, we see that it is set theoretically supported on the point $\xi_{\fm,v}$.
\end{proof}

An element $\xi\in (\hat{T}\git W)(\Lambda)$ is called \emph{regular semisimple} if it is the conjugacy classes of an regular semisimple elements in $\hat{T}(\Lambda)$. Concretely, an element $x=\diag(x_1,\dots,x_n)\in\hat{T}(\Lambda)$ is regular semisimple if and only if $x_i\neq x_j$ for any $i\neq j$.

The Iwahori invariance of objects in $\Rep^{\widehat\unip}(G(F),\Lambda)^\Adm_{\xi}$ can be decomposed using the action of $\cR$. Let $\varsigma\colon \hat{T}\to \hat{T}\git W$ denote the projection. 

\begin{prop}\label{prop-decomp-R}
	Let $\xi\in (\hat{T}\git W)(\Lambda)$ be an element. Let $V\in\Rep^{\widehat\unip}(G(F),\Lambda)^\Adm_{\xi}$. There is a natural decomposition
	$$V^I=\bigoplus_{x\in \varsigma^{-1}(\xi)} V^{I,x},$$
	compatible with $\cR$-actions such that $V^{I,x}$ is set theoretically supported at $x\in \hat{T}(\Lambda)$ when viewed as an $\cR$-module.
\end{prop}
\begin{proof}
	This is clear because the action of $\cR$ on $V^I$ is supported on the finite set $\varsigma^{-1}(\xi)$.
\end{proof}

For $x\in \hat{T}(\Lambda)$, denote
$$\widehat{\delta}_x=(\cInd_I^{G(F)}\Lambda)\otimes_{\cR}\cR_x^\wedge$$
as a Pro-object in $\Rep^{\widehat{\unip}}(G(F),\Lambda)$, where $\cR_x^\wedge$ is the completion of $\cR$ at $x$. We see that
$$(\cInd_I^{G(F)}\Lambda)\otimes_{Z(\cH)}(Z(\cH))^\wedge_\xi\cong \bigoplus_{x\in\varsigma^{-1}(\xi)}\widehat{\delta}_x$$
as $G(F)$-representations. By definition, for $V\in \Rep^{\widehat\unip}(G(F),\Lambda)^\Adm_\xi$, there is an isomorphism
$$V^I\cong\Hom_{\Pro\Rep^{\widehat{\unip}}(G(F),\Lambda)}((\cInd_I^{G(F)}\Lambda)\otimes_{Z(\cH)}Z(\cH)^\wedge_\xi, V).$$
The decomposition in Proposition \ref{prop-decomp-R} can be alternatively defined as
$$V^{I,x}=\Hom_{\Pro\Rep^{\widehat{\unip}}(G(F),\Lambda)}(\widehat{\delta}_x,V).$$

\subsection{Irreducible unipotent $\GL_n(F)$-representations with regular semisimple parameters}\label{subsection-irrep}

Let $n=n_1+\cdots+n_t$ be a partition of $n$. Let $P=P_{n_1,\dots,n_t}$ be the standard parabolic subgroup of $G$ with Levi subgroup $M\cong \GL_{n_1}\times\cdots\times\GL_{n_t}$. Let $W_M$ be the Weyl group $M$. For a smooth representation $V$ of $M(F)$, define the normalized parabolic induction as 
$$\nInd_{P(F)}^{G(F)}V\coloneqq \Ind_{P(F)}^{G(F)}(V\otimes\delta_P^{1/2}), $$
where $\delta_P^{1/2}$ is the square root of the modulus character of $P$ (determined by our choice of $\sqrt{q}\in\Lambda$). More concretely, the character $\delta_P^{1/2}$ factors though $M$ and is given by
$$\delta_P^{1/2}\colon (m_1,\dots,m_t)\mapsto |\det(m_1)|_F^{\frac{n-n_1}{2}}|\det(m_2)|_F^{\frac{n-2n_1-n_2}{2}}\cdots |\det(m_t)|_F^{\frac{n-2n_1-\cdots-2n_{t-1}-n_t}{2}}$$
for $(m_1,\dots,m_t)\in \GL_{n_1}(F)\times\cdots\times\GL_{n_t}(F)$. Here $|\cdot|_F\colon F^\times\to \Lambda$ is the unramified character sending $\varpi$ to $q^{-1}$. Let $v_F\colon F\to \ZZ\sqcup\{\infty\}$ be the valuation with $v_F(\varpi)=1$. Then $|\cdot|=q^{-v_F(\cdot)}$.

\begin{lemma}\label{lemma-jacquet-comp}
    Fix $x_1,\dots,x_t\in \Lambda^\times$. For each $i=1,\dots,t$, let $\chi_i\colon\GL_{n_i}(F)\to \Lambda^\times$ be an unramified character sending $m_i\in \GL_{n_i}(F)$ to $x_i^{v_F(\det(m_i))}$. Denote $\chi=\chi_1\otimes\cdots\otimes\chi_t\colon M(F)\to \Lambda^\times$. Then
	$$\big(\nInd_{P(F)}^{G(F)}\chi\big)_{N^-(F)}^\mathrm{ss}\simeq \bigoplus_{w\in {}^{W_M}W}\delta_B^{-1/2}(\tilde{\chi}\circ w)$$
	as $T(F)$-representations.
	Here $N^-$ is the unipotent radical of the opposite Borel subgroup $B^-$, ${}^{W_M}W$ is the set of elements in $W$ that have minimal length in their left $W_M$-cosets, and $\tilde\chi\colon T(F)\to\Lambda^\times$ is the unramified character corresponding to the element
	$$x_{\tilde{\chi}}=\diag(x_1q^{\frac{1-n_1}{2}},\dots,x_1q^{\frac{n_1-1}{2}},x_2q^{\frac{1-n_2}{2}},\dots,x_2q^{\frac{n_2-1}{2}},\dots,x_tq^{\frac{1-n_t}{2}},\dots,x_tq^{\frac{n_t-1}{2}})$$
    in $\hat{T}(\Lambda)$.
\end{lemma}
\begin{proof}
	The proof is similar to the proof of \cite[Theorem 3.5]{Cartier-rep}. By Bruhat decomposition, we have $G(F)=\bigsqcup_{w\in {}^{W_M} W} P(F)wB^-(F)$. For $w\in {}^{W_M} W$, denote $d(w)=\dim P\backslash PwB^-$. This defines a $B^-(F)$-stable filtration $I_\bullet$ of $\nInd_{P(F)}^{G(F)}(\chi_1\otimes\dots\otimes\chi_t)$ with
	$$I_s=\{f\in \nInd_{P(F)}^{G(F)}(\chi_1\otimes\dots\otimes\chi_t)\,|\, f\text{ vanishes on }\bigsqcup_{d(w)< s}P(F)wB^-(F)\}.$$
	The grading pieces are given by $I_s/I_{s+1}=\bigoplus_{d(w)=s}J_w$, where $J_w$ is the space of functions $f$ on $P(F)wB^-(F)$ that are compactly supported modulo $P(F)$ and satisfy $f(pwb)=\delta_P^{1/2}(p)\chi(p)f(wb)$ for $p\in P(F)$ and $b\in B(F)$. Denote $B^-(w)=B^-\cap w^{-1}Pw$. We see that $J_w\cong \cInd_{B^-(w)(F)}^{B^-(F)}((\delta_{P}^{1/2}\chi\circ w)|_{T(F)})$. By the proof of \cite[Theorem 3.5]{Cartier-rep}, we know that
	$$(J_w)_{N^-(F)}\simeq (\delta_{P}^{1/2}\chi|_{T(F)}\circ w)\cdot \frac{\delta_{B^-(F)}}{\delta_{B^-(w)(F)}}$$
	as $T(F)$-representations. 
	
	Let $\Phi$ (resp. $\Phi_M$) be the set of roots of $G$ (resp. $M$). Let $\Phi_{U_P}$ be the set of roots appearing in the Lie algebra of the unipotent radical $U_P$ of $P$. Denote $\Phi_P=\Phi_M\sqcup\Phi_{U_P}$ Note that $\Phi=\Phi_M\sqcup \Phi_{U_P}\sqcup (-\Phi_{U_P})$. For a weight $\lambda\in \XX^\bullet(T)$, denote by $|\lambda|$ the composition $T(F)\xrightarrow{\lambda}F^\times\xrightarrow{|\cdot|_F}\Lambda^\times$. We see that 
	$$\delta_P|_{T(F)}=\prod_{\alpha\in \Phi_{U_P}}|\alpha|,\quad \delta_{B^-(F)}=\prod_{\alpha<0}|\alpha|,\quad\text{and}\quad\delta_{B^-(w)(F)}=\prod_{\alpha<0,w\alpha\in\Phi_P}|\alpha|.$$
	It follows that
	$$\begin{aligned}
		(\delta_P^{1/2}|_{T(F)}\circ w)\cdot \frac{\delta_{B^-(F)}}{\delta_{B^-(w)(F)}} & =\prod_{w\alpha\in \Phi_{U_P}}|\alpha|^{1/2}\cdot\prod_{\alpha<0}|\alpha|\cdot\prod_{\alpha<0,w\alpha\in \Phi_P}|\alpha|^{-1}\\
		&=\prod_{w\alpha\in -\Phi_{U_P}}|\alpha|^{-1/2}\cdot\prod_{\alpha<0,w\alpha\in -\Phi_{U_P}}|\alpha|\\
		&=\prod_{\alpha<0,w\alpha\in -\Phi_{U_P}}|\alpha|^{1/2}\cdot\prod_{\alpha<0,w\alpha\in \Phi_{U_P}}|\alpha|^{1/2}\\
		&=\delta_{B}^{-1/2}\cdot\prod_{\alpha>0,w\alpha\in \Phi_M}|\alpha|^{1/2}.
	\end{aligned}$$
	Since $w\in {}^{W_M}W$, we have $\prod_{\alpha>0,w\alpha\in \Phi_M}|\alpha|^{1/2}=\delta^{1/2}_{B(F)\cap M(F)}\circ w$. The Lemma follows as $\tilde\chi=\chi|_{T(F)}\cdot \delta^{1/2}_{B(F)\cap M(F)}$.
\end{proof}

\begin{lemma}\label{lemma-Jacquet-vs-Iwahori}
	Let $(V,\iota)\in\Rep(G(F),\Lambda)$ be an admissible representation of $G(F)$ concentrated in degree $0$. Let $\lambda\in \XX_\bullet^-(T)$ be an anti-dominant coweight. Then for any $y\in V^I$, we have
	$$\iota(1_{I\varpi^\lambda I})y-q^{-\langle2\rho,\lambda\rangle} \iota(\varpi^{\lambda}) y\in V(N^-(F)).$$
	Here $V(N^-(F))$ is the $\Lambda$-submodule of $V$ generated by $nz-z$ for $n\in N^-(F)$ and $z\in V$.
\end{lemma}
\begin{proof}
	Take $y\in V^I$. For $\lambda\in \XX_\bullet^-(T)$, we have $\varpi^{-\lambda}(N(F)\cap I)\varpi^{\lambda}\subset N(F)\cap I$.  Therefore $\iota(\varpi^\lambda)y\in V^{N(F)\cap I}$ for $\lambda\in \XX_\bullet^-(T)$. By Iwahori decomposition, we have
	$$\int_{I}\iota(g\varpi^\lambda)y\mathrm{d} g=\int_{I\cap N^-(F)}\iota(n\varpi^\lambda)y\mathrm{d}n.$$
	It follows that 
	$$\iota(1_{I\varpi^\lambda I})y-\mathrm{Vol}(I\varpi^\lambda I) \iota(\varpi^\lambda)y\in V(N^-(F)),$$
	where the measure on $G(F)$ is normalized so that $\mathrm{Vol}(I)=1$. The claim follows as $\mathrm{Vol}(I\varpi^\lambda I)=q^{-\langle \lambda,2\rho\rangle}.$
\end{proof}

\begin{cor}\label{cor-parabolic-R}
	Let notations be as in Lemma \ref{lemma-jacquet-comp}. Assume that $x_{\tilde\chi}\in\hat{T}(\Lambda)$ is regular semisimple. Then there is an isomorphism
	$$(\nInd_{P(F)}^{G(F)}\chi)^I\simeq \bigoplus_{w\in {}^{W_M}W}\Lambda_{w^{-1}(x_{\tilde{\chi}})}$$
	of $\cR$-modules. Here $\Lambda_{w^{-1}(x_{\tilde{\chi}})}$ is the skyscraper $\cR$-module supported at $w^{-1}(x_{\tilde{\chi}})$.
\end{cor}
\begin{proof}
	By \cite[Theorem 3.7]{Cartier-rep}, we have $\nInd_{P(F)}^{G(F)}\chi=(\nInd_{P(F)}^{G(F)}\chi)^I\oplus (\nInd_{P(F)}^{G(F)}\chi)(N^-(F))$. The claim follows from Lemma \ref{lemma-jacquet-comp} and Lemma \ref{lemma-Jacquet-vs-Iwahori}. Note that all $w^{-1}(x_{\tilde{\chi}})$'s are distinct as $x_{\widetilde\chi}$ is regular semisimple. Hence there is no need to take semisimplications.
\end{proof}

We define some index sets, which will be important to state our main result.

\begin{defn}\label{def-P_x}
    Let $\xi\in (\hat{T}\git W)(\Lambda)$ be a regular semisimple element. Let $E_{\xi}\subseteq \Lambda^\times$ be the set of eigenvalues of an element in the conjugacy class $\xi$. Define a set
$$Q_{\xi}=\{(a,b)\in E_{\xi}\times E_{\xi}\,|\, a=qb\}.$$
Let $x=\diag(x_1,\dots,x_n)\in \varsigma^{-1}(\xi)$ be a lift of $\xi$ in $\hat{T}(\Lambda)$. Define  
$$P_x=\{(a,b)\in Q_{\xi}\,|\, \text{if $x_i=a$ and $x_j=b$, then }i<j \}$$
as a subset of $Q_{\xi}$.
\end{defn}

By our banal assumption, there is no loop in $Q_{\xi}$, i.e. there does not exist $a_1,\dots,a_t\in E_{\xi}$ such that $(a_1,a_2), (a_2,a_3),\dots,(a_{t-1},a_t), (a_t,a_1)\in Q_{\xi}$.

We recall the notion of Whittaker-genericness.

\begin{defn}
	We say an irreducible smooth representation $V$ of $G(F)$ over $\Lambda$ is \emph{Whittaker-generic} if it has non-zero Whittaker coefficients, i.e. $\Hom_{N(F)}(V,\psi)\neq 0$ where $\psi\colon N(F)\to \Lambda^\times$ is a non-degenerate character.
\end{defn}
\begin{rmk}
	The notion of Whittaker-generic is usually called generic, for example, in \cite{CHT08}. We choose to use a different name to distinguish it with \emph{cohomologically generic} defined in Definition \ref{def-Gal-generic}.
\end{rmk}

We show that the decomposition in Proposition \ref{prop-decomp-R} detects Whittaker-generic irreducible representations.

\begin{prop}\label{prop-Whit-generic-comp}
	Let $\xi\in (\hat{T}\git W)(\Lambda)$ be a regular semisimple element. There is a unique (up to isomorphism) Whittaker-generic irreducible smooth $G(F)$-representation in $\Rep^{\widehat\unip}(\GL_n(F),\Lambda)_{\xi}^{\Adm,\heartsuit}$, denoted by $\St_{\xi}$. Let $x\in \varsigma^{-1}(\xi)$ be an element with $P_x=Q_{\xi}$. Then for any irreducible representation $V\in \Rep^{\widehat\unip}(\GL_n(F),\Lambda)_{\xi}^{\Adm,\heartsuit}$, we have
	$$\dim_\Lambda V^{I,x} =\left\{\begin{aligned}
	0, \text{ if }V\not\simeq \St_{\xi}, \\
	1, \text{ if }V\simeq \St_{\xi}.
\end{aligned}\right.$$
\end{prop}
\begin{proof}
	The first half of the proposition follows from Bernstein--Zelevinsky classification. See \cite{BZ-classification-torsion} for Bernstein--Zelevinsky classification in torsion coefficients. Let $\psi\colon N(F)\to \Lambda^\times$ be a non-degenerate character. Let $x'\in \hat{T}(\Lambda)$ be a lift of $\xi$. Let $\chi_{x'}$ be the unramified character of $T(F)$ associated to ${x'}$. The set of irreducible factors of $\nInd_{B(F)}^{G(F)}\chi_{x'}$ is equal to the set of isomorphism classes of irreducible representations in $\Rep^{\widehat\unip}(G(F),\Lambda)^\Adm_{\xi}$. By Bruhat decomposition, we have $$\dim\Hom_{N(F)}(\nInd_{B(F)}^{G(F)}\chi_{x'},\psi)=1.$$ 
    As the functor $\Hom_{N(F)}(-,\psi)$ is $t$-exact, there is a unique Whittaker-generic irreducible representation, denoted by $\St_{\xi}$, in the composition series of $\nInd_{B(F)}^{G(F)}\chi_{x'}$. Hence $\St_{\xi}$ is also the unique Whittaker-generic irreducible representation in  $\Rep^{\widehat\unip}(G(F),\Lambda)^{\Adm,\heartsuit}_{\xi}$. By Bernstein--Zelevinsky classification, $\St_{\xi}$ is the unique irreducible quotient of $\nInd_{B(F)}^{G(F)}\chi_{x'}$ for $x'\in \varsigma^{-1}(\xi)$ with $P_{x'}=Q_{\xi}$. 
	
	Let $V$ be an irreducible object in $\Rep^{\widehat\unip}(G(F),\Lambda)^{\Adm,\heartsuit}_{\xi}$ that is not Whittaker-generic. Then $V$ appears in the composition series of a parabolic induction $\nInd_{P(F)}^{G(F)}\chi$ where $P\neq B$ is a standard parabolic subgroup, and $\chi\colon M(F)\to \Lambda^\times$ is an unramified character such that the associated element $x_{\tilde\chi}\in\hat{T}(\Lambda)$ as in Lemma \ref{lemma-jacquet-comp} lies in the conjugacy class $\xi$. By Corollary \ref{cor-parabolic-R}, we know that there is an isomorphism
		$$(\nInd_{P(F)}^{G(F)}\chi)^I\simeq \bigoplus_{w\in {}^{W_M}W} \Lambda_{w^{-1}(x_{\widetilde{\chi}})}$$
	of $\cR$-modules. Because of the form of $x_{\tilde{\chi}}$ and $w\in {}^{W_M}W$, there exists $i<j$ such that $(w^{-1}x_{\tilde{\chi}})_j=q(w^{-1}x_{\tilde\chi})_i$. In particular, the $\cR$-module $(\nInd_{P(F)}^{G(F)}\chi)^I$ is not supported on the points $x\in \varsigma^{-1}(\xi)$ with $P_x=Q_{\xi}$. Thus $V^{I,x}=0$ if $V$ is not Whittaker-generic and $P_x=Q_{\xi}$. 
		
	By Corollary \ref{cor-parabolic-R}, we have $\dim_\Lambda (\nInd_{B(F)}^{G(F)}\chi_{x'})^{I,x}=1$ for any $x,x'\in \varsigma^{-1}(\xi)$. Moreover, the Whittaker-generic constituent $\St_{\xi}$ in $\nInd_{B(F)}^{G(F)}\chi_{x'}$ is the only constituent that contributes to $(\nInd_{B(F)}^{G(F)}\chi_{x'})^{I,x}$ if $P_x=Q_{\xi}$. The result now follows.
\end{proof}
\begin{rmk}
	Let $V\subseteq \Rep^{\widehat\unip}(\GL_n(F),\Lambda)_{\xi}^{\Adm,\heartsuit}$ be an irreducible representation. Then the support of $V^I$ as $\cR$-modules are closely related to the monodromy of the associated local Galois representation. See Lemma \ref{lemma-rep-char-0}.
\end{rmk}

\begin{eg}
	Let $x=\diag(q^{\frac{n-1}{2}},\dots,q^{\frac{1-n}{2}})\in \hat{T}(\Lambda)$. Denote by $\xi$ the image of $x$ in $\hat{T}\git W$. Irreducible representations in $\Rep^{\widehat\unip}(G(F),\Lambda)^{\Adm,\heartsuit}_{\xi}$ are called generalized Steinberg representations: For a partition $\underline{n}$ of $n$, denote by $\St_{\underline{n}}$ the unique irreducible quotient of the (un-normalized) parabolic induction $\Ind_{P_{\underline{n}}(F)}^{G(F)}\Lambda$. Then $\St_{\underline{n}}$ gives a set of representatives of irreducible representations in $\Rep^{\widehat\unip}(G(F),\Lambda)^{\Adm,\heartsuit}_{\xi}$. For example, $\St_{(n)}=\Lambda$ is the trivial representation and $\St_{(1,1,\dots,1)}=\St_{\xi}$ is the unique Whittaker-generic one. We have $\dim (\Lambda)^I=\dim (\St_{\xi})^I=1$. As $\cR$-modules, $(\Lambda)^{I}$ is supported at the point $\diag(q^{\frac{1-n}{2}},\dots,q^{\frac{n-1}{2}})$, and $(\St_{\xi})^{I}$ is supported at the point $\diag(q^{\frac{n-1}{2}},\dots,q^{\frac{1-n}{2}})$.
\end{eg}

\subsection{Statement of the main result}\label{subsection-statement}

Recall that the global Hecke algebra $\TT^S$ acts on the space of automorphic forms $\cA$. Let $\fm\subseteq\TT^S$ be a maximal ideal with $\cA_\fm^{G(O_F)}\neq 0$.

\begin{defn}\label{def-Gal-generic}
	We say a maximal ideal $\fm\subseteq\TT^S$ is \emph{cohomologically generic} if there exists a prime number $r\notin S$ such that either of the following conditions holds:
	\begin{enumerate}
		\item The prime $r$ splits in $\sK$. For any place $w$ of $\sF$, let $x_1,\dots,x_n$ be eigenvalues of $\rho_\fm(\Frob_w)$. Then  $x_i/x_j\neq q_w$ for any $i\neq j$.
        \item The prime $r$ is inert in $\sK$. The assocaited unitary Satake parameter of $\fm$ at $r$ is generic in the sense of \cite[Definition 1.1]{YZ-torsion-vanishing}.
	\end{enumerate}
\end{defn}

\begin{rmk}
    If $\fm$ is cohomologically generic, then we can apply \cite[Theorem 1.4]{YZ-torsion-vanishing} to certain compact unitary Shimura varieties constructed latter to show that the cohomology concentrated in middle degree after localizing at $\fm$. This generalize the results of \cite{Caraiani-Scholze17}, \cite{Koshikawa-generic}, and \cite{Hamann-Lee}.
\end{rmk}

Now we can state our main result.

\begin{prop}\label{prop-Hecke-dim}
	Let $\xi\in (\hat{T}\git W)(\Lambda)$ be a regular semisimple element. Then for any $x,x'\in\varsigma^{-1}(\xi)$, then space
	$$\Hom_{\Pro\Rep^{\widehat{\unip}}(G(F),\Lambda)}(\widehat\delta_x,\widehat\delta_{x'})$$
	is concentrated in degree zero, and is a free module over $Z(\cH)^\wedge_\xi$ of rank one.
\end{prop}

The proof of Proposition \ref{prop-Hecke-dim} will be given in \S\ref{subsection-construct-of-morphism}.

\begin{thm}\label{thm-main}
	Assume the following conditions hold:
	\begin{enumerate}[(i)]
		\item $\ell$ is banal at $v$, i.e. $\ell\nmid |\GL_n(\FF_q)|$ where $q$ is the cardinality of the residue field of $F=\sF_v$.
		\item $\rho_\fm(\Frob_v)$ is regular semisimple. Let $\xi_{\fm,v}$ be the conjugacy class of $q^{\frac{1-n}{2}}\rho_\fm(\Frob_v)$ as in Proposition \ref{prop-local-global-auto}.
		\item $F$ is unramified over $\QQ_p$, and $F\neq \QQ_p$.
		\item $\fm$ is cohomologically generic.
	\end{enumerate}
	Let $x,x'\in \hat{T}(\Lambda)$ be two lifts of $\xi_{\fm,v}$ such that $P_{x'}\subseteq P_x$. 
	Then we have:
	\begin{enumerate}
		\item If $\alpha^\mathrm{Hk}_{x,x'}$ is a generator in $\Hom_{\Pro\Rep^{\widehat{\unip}}(G(F),\Lambda)}(\widehat\delta_x,\widehat\delta_{x'})$ (cf. Proposition \ref{prop-Hecke-dim}), then the induced morphism
			$$\alpha_{x,x'}\colon\cA_\fm^{I,x}\to \cA_\fm^{I,x'}$$
			is surjective.
		\item If $\beta^\mathrm{Hk}_{x,x'}$ is a generator in $\Hom_{\Pro\Rep^{\widehat{\unip}}(G(F),\Lambda)}(\widehat\delta_x,\widehat\delta_{x'})$ (cf. Proposition \ref{prop-Hecke-dim}), then the induced morphism
			$$\beta_{x,x'}\colon\cA_\fm^{I,x'}\to \cA_\fm^{I,x}$$
			is injective.
	\end{enumerate}
\end{thm}

\begin{rmk}
	The statement of the Theorem is independent of the choices of the generators $\alpha^\mathrm{Hk}_{x,x'}$ and $\beta^\mathrm{Hk}_{x,x'}$ as they are well-defined up to a unit in $Z(\cH)^\wedge_\xi$.
\end{rmk}

\begin{rmk}
	Let us compare our results with Ihara's lemma for $\GL_2$. If $n=2$, and for $V$ in $\Rep^{\widehat\unip}(G(F),\Lambda)^\Adm_\xi$ with $\xi$ regular semisimple, we have $V^I=V^{I,x}\oplus V^{I,x'}$ for two points $x,x'\in \varsigma^{-1}(\xi)$. Assume that $P_{x'}\subseteq P_x$. We can show that $V^{I,x'}$ is isomorphic to $V^{G(O_F)}$ in a functorial way. Let $\omega=\begin{pmatrix}0&p\\ 1&0\end{pmatrix}$. Then $\omega^{-1} I\omega=I$, i.e. $\omega^{-1}$ lies in the normalizer of $I$. Hence there is an action $\omega\colon V^I\to V^I$ for any $V$. Let $\mathrm{Av}\colon V^I\to V^{G(O_F)}$ be the averaging map. Ihara's lemma asserts that the morphism
	$$(\mathrm{Av},\mathrm{Av}\circ \omega)\colon \cA_\fm^I\to (\cA_\fm^{G(O_F)})^{\oplus 2}$$
	is surjective.
	By an explicit computation, we can see that the natural map $(\mathrm{Av},\mathrm{Av}\circ \omega)\colon V^{I}\to (V^{G(O_F)})^{\oplus 2}$  is surjective if and only if the map $\alpha_{x,x'}\colon V^{I,x}\to V^{I,x'}$ is surjective for $\alpha_{x,x'}$ in Theorem \ref{thm-main} (1). Dually, let $\mathrm{Res}\colon V^{G(O_F)}\to V^{I}$ be the inclusion. 
	Then the natural map $(\mathrm{Res},\omega\circ\mathrm{Res})\colon (V^{G(O_F)})^{\oplus 2}\to V^I$ is injective if and only if the map $\beta_{x,x'}\colon V^{I,x'}\to V^{I,x}$ is injective. This shows that Theorem \ref{thm-main} is equivalent to Ihara's lemma in $\GL_2$ case.
\end{rmk}

As a corollary, we confirm some cases of \cite[Conjecture B]{CHT08}.

\begin{cor}\label{cor-CHT}
	Assume that we are in the situation of Theorem \ref{thm-main}. Then any irreducible $G(F)$-submodule of $\cA_\fm$ is Whittaker-generic.
\end{cor}
\begin{proof}
	Let $V\hookrightarrow \cA_\fm$ be an irreducible $G(F)$-submodule. Let $x\in \varsigma^{-1}(\xi_{\fm,v})$ be an element with $P_x=Q_{\xi}$. Let $x'\in \varsigma^{-1}(\xi_{\fm,v})$. Let $\beta_{x,x'}^\Hk$ be a generator of $\Hom_{G(F)}(\widehat\delta_x,\widehat\delta_{x'})$. By Theorem \ref{thm-main}, there is a commutative diagram
	$$\begin{tikzcd}
		V^{I,x'} \ar[r,"\beta_{x,x'}"]\ar[d] & V^{I,x} \ar[d] \\
		\cA_\fm^{I,x'} \ar[r,"\beta_{x,x'}"] & \cA_\fm^{I,x},
	\end{tikzcd}$$
	where horizontal morphisms are induced by $\beta^\Hk_{x,x'}$. By Theorem \ref{thm-main} (2), the lower horizontal arrow is injective. Two vertical arrows are clearly injective. It follows that $\beta_{x,x'}\colon V^{I,x'}\to V^{I,x}$ is injective for any $x'$. As $V\in\Rep^{\widehat\unip}(G(F),\Lambda)_{\xi_{\fm,v}}^\Adm$, there exists $x'\in \varsigma^{-1}(\xi_{\fm,v})$ such that $V^{I,x'}$ is non-zero. This implies that $V^{I,x}\neq 0$. By Proposition \ref{prop-Whit-generic-comp}, we see that $V$ must be Whittaker-generic.
\end{proof}

\section{Unipotent categorical local Langlands correspondence}\label{section-CLLC}

In this section, we construct the morphisms in Theorem \ref{thm-main} (1). For this, we use the unipotent categorical local Langlands correspondence recently established in \cite{Tame}, and transfer the objects $\widehat{\delta}_x$'s to the spectral side.

\subsection{General theory}\label{subsection-ucllc}
We recall the unipotent categorical local Langlands correspondence in \cite{Tame}.

In this subsection, let $F$ be a non-archimedean local field of mixed characteristic $(0,p)$. Let $O_F$ be the ring of integers in $F$. Let $\varpi\in O_F$ be a uniformizer and $k_F=O_F/(\varpi)$ be the residue field. Denote $q=|k_F|$. Let $k=\overline{k}_F$ be the algebraic closure $k_F$. For a perfect $k_F$-algebra $R$, denote by 
$$W_{O_F}(R)\coloneqq W(R)\otimes_{W(k_F)}O_F$$ 
the ring of $O_F$-Witt vectors of $R$. 

Let $\Perf^\aff_k$ denote the site of perfect $k$-algebras with \'etale topology. For an object $R\in\Perf^\aff_k$, denote by $\sigma_R\colon\Spec(R)\to\Spec(R)$ the absolute $q$-Frobenius. By abuse of notations, we also denote by $\sigma_R$ the $q$-Frobenius on $\Spec(W_{O_F}(R))$ and $\Spec(W_{O_F}(R)[\frac1p])$.

Let $G$ be a quasi-split connected reductive group over $F$. Fix a pinning $(B,T,e)$ of $G$. It defines an Iwahori subgroup $I$ of $G(F)$. Denote by $\cI$ the Iwahori group scheme over $O_F$ associated to $I$.

Recall the definitions of (positive) loop groups.

\begin{defn}
	Let $LG$ be the group-valued sheaf on $\Perf^\aff_k$ that sends $R\in\Perf^\aff_k$ to the group $G(W_{O_F}(R)[\frac 1p])$. Then $LG$ is represented by an ind perfect scheme. Let $L^+\cI$ be the perfect affine group scheme sending $R\in\Perf^\aff_k$ to $\cI(W_{O_F}(R))$.
\end{defn}

\begin{defn}
	Let $\Isoc_G$ be the \'etale sheafification of the quotient prestack $[LG/\mathrm{Ad}_\sigma LG]$, where $\mathrm{Ad}_\sigma$ is the $\sigma$-conjugation given by $\mathrm{Ad}_\sigma(g)(h)=gh\sigma(g)^{-1}$. For $R\in\Perf_k^\aff$, the $R$-point of $\Isoc_G$ is identified with the groupoid of $G$-isocrysrals $(\cE,\sigma_\cE)$, where $\cE$ is a $G$-torsor on $\Spec(W_{O_F}(R)[\frac1p])$ that is trivial \'etale locally on $R$, and $\sigma_\cE\colon \sigma_R^*\cE\xrightarrow{\sim}\cE$ is an isomorphism of $G$-torsors.
\end{defn}

\begin{defn}
	Let $\Sht^\loc_\cI$ be the \'etale sheafification of the quotient prestack $[LG/\mathrm{Ad}_\sigma L^+\cI]$. For $R\in \Perf^\aff_k$, the $R$-points of $\Sht^\loc_\cI$ is identified with the groupoid of $\cI$-local shtukas $(\cE,\sigma_\cE)$, where $\cE$ is an $\cI$-torsor over $\Spec(W_{O_F}(R))$ and  $\sigma_\cE\colon \sigma_R^*(\cE|_{\Spec(W_{O_F}(R)[\frac1p])})\xrightarrow{\sim}\cE|_{\Spec(W_{O_F}(R)[\frac1p])}$ is an isomorphism of $G$-torsors over $\Spec(W_{O_F}(R)[\frac1p]$.
\end{defn}

It is proved in \cite[\S3]{Tame} that the underlying topological space $|\Isoc_G|$ of $\Isoc_G$ is isomorphic to the Kottwitz set $B(G)=G(\breve{F})/(g\sim hg\sigma(h)^{-1})$, where $\breve{F}=\widehat{F}^\mathrm{unr}$ is the completion of the maximal unramified extension of $F$. For an element $b\in B(G)$, define the $\sigma$-centralizer of $b$ as $G_b(F)=\{x\in G(\breve{F})|xb\sigma(x)^{-1}=b\}$. The collections $\{G_b(F)\}_{b\in B(G)}$ of $p$-adic groups are called extended pure inner forms of $G$. Each group $G_b$ is an inner form of a Levi subgroup of $G$.

For a point $b\in B(G)$, let $i_b\colon \Isoc_{G,b}\hookrightarrow \Isoc_G$ be the locally closed substack associated to $b$. After choosing a point in $\Isoc_{G,b}$, there is an isomorphism $\Isoc_{G,b}\simeq \BB_{\textrm{pro\'et}}G_b(F)$. Here $\BB_{\textrm{pro\'et}}G_b(F)$ is the pro-\'etale classifying stack of the locally profinite group $G_b(F)$. 

Let $\Shv(-,\Lambda)$ be the $\Lambda$-\'etale ind-constructible (co)sheaf theory defined in \cite[\S10]{Tame}. By \cite[\S 3.4]{Tame}, the category $\Shv(\Isoc_G,\Lambda)$ has a semi-orthogonal decomposition into $\Shv(\Isoc_{G,b},\Lambda)\simeq \Rep(G_b(F),\Lambda)$, $b\in B(G)$. Moreover, an object $A\in \Shv(\Isoc_G,\Lambda)$ is compact if and only if $(i_b)^*A=0$ for all but finitely many $b\in B(G)$, and $(i_b)^*A\in \Rep(G_b(F),\Lambda)$ is compact for any $b\in B(G)$, or equivalently, if and only if $(i_b)^!A=0$ for all but finitely many $b\in B(G)$, and $(i_b)^!A\in\Rep(G_b(F),\Lambda)$ is compact for any $b\in B(G)$. Here, an object $V\in\Rep(G_b(F),\Lambda)$ is compact if and only if it lies in the idempotent complete $\Lambda$-linear category generated by compact inductions $\cInd^{G_b(F)}_{K}\Lambda$, where $K$ runs through pro-$p$ open compact subgroups of $G_b(F)$; see \cite[\S 3.3]{Tame}.

Let $\Shv_{\fgen}(\Isoc_G,\Lambda)\subseteq \Shv(\Isoc_G,\Lambda)$ be the subcategory of finitely generated sheaves as in \cite[\S10.6]{Tame}. By \cite[Proposition 3.96]{Tame}, an object $A\in\Shv(\Isoc_G,\Lambda)$ is finitely generated if and only if $(i_b)^!A=0$ for all but finitely many $b\in B(G)$ and $(i_b)^!A$ is finitely generated for any $b\in B(G)$, or equivalently, if and only if $(i_b)^*A=0$ for all but finitely many $b\in B(G)$ and $(i_b)^*A$ is finitely generated for any $b\in B(G)$. Here, an object $V\in\Rep(G_b(F),\Lambda)$ is finitely generated if and only if it lies in the idempotent complete $\Lambda$-linear category generated by compact inductions $\cInd^{G_b(F)}_{K}\Lambda$, where $K$ runs theough all open compact subgroups of $G_b(F)$; see \cite[\S 3.3]{Tame}. Note that if $\ell$ is banal for any $G_b(F)$ (which is the case we are interested in latter), then $\Rep_\fgen(G_b(F),\Lambda)=\Rep(G_b(F),\Lambda)^\omega$ for any $b\in B(G)$, and hence $\Shv_\fgen(\Isoc_G,\Lambda)=\Shv(\Isoc_G,\Lambda)^\omega$.

If $G$ splits over an unramified extension of $F$, then we define the subcategory $\Rep^{\widehat\unip}(G(F),\Lambda)\subseteq\Rep(G(F),\Lambda)$ (resp. $\Rep^\unip_\fgen(G(F),\Lambda)\subseteq \Rep_\fgen(G(F),\Lambda)$) of unipotent representations following \cite[Definition 4.116]{Tame}. Let $\Shv^{\widehat\unip}(\Isoc_G,\Lambda)\subseteq \Shv(\Isoc_G,\Lambda)$ (resp. $\Shv^\unip_\fgen(G(F),\Lambda)\subseteq \Shv_\fgen(\Isoc_G,\Lambda)$) be the subcategory of of objects $A\in\Shv(\Isoc_G,\Lambda)$ (resp. $A\in\Shv_\fgen(\Isoc_G,\Lambda)$) such that $(i_b)^!A\in \Rep^{\widehat\unip}(G(F),\Lambda)$ (resp. $(i_b)^!A\in \Rep^{\unip}_\fgen(G(F),\Lambda)$) for any $b\in B(G)$. Then $\Ind\Shv_\fgen^{\unip}(\Isoc_G,\Lambda)$ contains $\Shv^{\widehat{\unip}}(\Isoc_G,\Lambda)$ as a full subcategory. If $\ell$ is banal for all $G_b(F)$, $b\in B(G)$, then the natural embedding induces an equivalence $\Shv^{\widehat{\unip}}(\Isoc_G,\Lambda)\simeq\Ind\Shv_\fgen^{\unip}(\Isoc_G,\Lambda)$.

Let $\hat{G}$ be the dual group of $G$ over $\Lambda$. It is endowed with a pinning $(\hat{B},\hat{T},\hat{e})$. The tuple $(\hat{G},\hat{B},\hat{T},\hat{e})$ carries an action of the Weil group $W_F$. Let ${}^LG=\hat{G}\rtimes W_F$ be the Langlands dual group. Denote by $\mathrm{pr}\colon {}^LG\to W_F$ the projection map. We define the notion of Langlands parameters with general coefficients in an arbitrary $\ZZ_\ell$-algebra following \cite[\S2.2.1]{Tame}. See also \cite{DHKM-L-par} and \cite{FS}.

\begin{defn}
	For an $\Lambda$-algebra $R$, a \emph{Langlands parameter} (or \emph{$L$-parameter}) of $G$ with coefficients in $R$ is a strongly continuous homomorphism $\varphi\colon W_F\to \hat{G}(R)\rtimes W_F$ such that the composition $\mathrm{pr}\circ\varphi\colon W_F\to W_F$ is the identity map. We can write $\varphi=(\varphi_0,\mathrm{pr})$ where $\varphi_0\colon W_F\to \hat{G}(R)$ is a strongly continuous 1-cocycle. 
\end{defn}

	We refer the definition of strongly continuous to \cite[\S 2.1.1]{Tame}. Let $\Loc_{{}^LG,F}$ be the stack of $L$-parameters of $G$ over $\Lambda$ as in \cite[\S 2.1]{Tame}, or equivalently \cite[\S VIII]{FS}.
\begin{rmk}
	As we have fixed $\sqrt{q}\in\Lambda$, the stack of $L$-parameters is canonically isomorphic to the stack of $C$-parameters $\Loc_{{}^cG,F}$. See \cite[Remark 2.4]{Tame}.
\end{rmk}

Assume that the group $G$ split over an unramified extension of $F$. Let $I_F\subseteq W_F$ be the inertia subgroup and $P_F\subseteq W_F$ be the wild inertia subgroup.  An $L$-parameter $\varphi\colon W_F\to {}^LG(\Lambda)$ is called \emph{tame} if $\varphi_0(x)=1$ for all $x\in P_F$. A tame $L$-parameter $\varphi$ is called \emph{unipotent} if moreover $\varphi_0(x)$ is unipotent in $\hat{G}(\Lambda)$ for any $x\in I_F$. For a general $\Lambda$-algerbra $R$, an $L$-parameter $\varphi$ over $R$ is tame (resp. unipotent) if it is tame (resp. unipotent) after pullback to any geometric point of $\Spec(R)$.
	
Let $\Loc_{{}^LG,F}^{\widehat\unip}$ be stack of unipotent $L$-parameters as in \cite[\S2.2.2]{Tame}. Then $\Loc_{{}^LG,F}^{\widehat\unip}$ is a quasi-compact open and closed substack of $\Loc_{{}^LG,F}$. Note that $\Loc_{{}^LG,F}^{\widehat\unip}$ is identified with the (derived) $\phi$-fixed points stack $\mathcal{L}_\phi(\cU^\wedge_{\hat{G}}/\hat{G})$, where $\cU^\wedge_{\hat{G}}$ is the completion of $\hat{G}$ along the unipotent cone $\cU_{\hat{G}}$. Here $\phi$ acts on $\cU^\wedge_{\hat{G}}$ by sending $g$ to $\sigma(g)^{1/q}$, where $\sigma$ is the arithmetic Frobenius in $W_F$. See \cite[\S 2.2]{Tame} for details.
	
Let $\hat{B}^-$ be the Borel subgroup of $\hat{G}$ opposite to $\hat{B}$. Define $\Loc^\unip_{{}^LB^-,F}=\mathcal{L}_\phi(\hat{U}^-/\hat{B}^-)$. The natural morphism $\hat{U}^-/\hat{B}^-\to\cU^\wedge_{\hat{G}}/\hat{G}$ induces a morphism $q^\spec\colon\Loc_{{}^LB^-,F}^{{\unip}}\to \Loc_{{}^LG,F}^{\widehat{\unip}}$. 

\begin{defn}
	Define the \emph{(unipotent) coherent Springer sheaf} as 
	$$\CohSpr=\CohSpr_{{}^LG,F}^{\unip}\coloneqq (q^\spec)_*\omega_{\Loc_{{}^LB^-,F}^{\unip}},$$
	where $\omega_{\Loc_{{}^LB^-,F}^{\unip}}$ is the dualizing sheaf on $\Loc_{{}^LB^-,F}^{\unip}$.
\end{defn}
\begin{rmk}
    The opposite Borel subgroup ${}^LB^-$ appears because we are using a different normalization comparing with \cite{Tame}. We define the Wakimoto sheaf $J_\lambda=\nabla_\lambda$ if $\lambda$ is \emph{dominant}, and use the \emph{opposite} Borel subgroup $\hat{B}^-$ on the spectral side everywhere. Then we set $\BB^\unip(J_\lambda)=\omega_{S_1^\unip}(\lambda)$ in \cite[Theorem 5.1 (7)]{Tame}.
\end{rmk}

We recall the unipotent categorical local Langlands correspondences.
\begin{thm}[{\cite[Theorem 5.4]{Tame}}]\label{thm-unip-cllc}
	Assume that $G$ is unramified over $F$ and is endowed with a pinning. Also assume that $\ell$ is large relative to $G$. Then there is a fully faithful embedding
	$$\LL^\unip\colon \Shv^\unip_\fgen(\Isoc_G,\Lambda)\hookrightarrow \Coh(\Loc_{{}^LG,F}^{\widehat\unip})$$
	of $\Lambda$-linear categories.
	Under the above embedding, the object $(i_1)_*\cInd_I^{G(F)}\Lambda\in\Shv_\fgen^{\widehat\unip}(G(F),\Lambda)$ corresponds to the coherent Springer sheaf $\CohSpr\in\Coh(\Loc_{{}^LG,F}^{\widehat\unip})$. Here $i_1\colon \BB G(F)\simeq \Isoc_{G,1}\hookrightarrow\Isoc_G$ is the Newton stratum indexed by the identity $1\in B(G)$.
\end{thm}

\subsection{Compatibility with Weil restrictions}
In this subsection we prove that the unipotent categorical local Langlands correspondence in Theorem \ref{thm-unip-cllc} is compatible with Weil restriction along unramified extension. Let $E/F$ be an unramified extension of non-archimedean local fields. Let $G$ be an unramified reductive group over $F$ with a pinning. Denote $G_0=\Res_{E/F}G$. Let $\cI$ be the Iwahori group scheme of $G$ over $O_F$ and $\cI_0=\Res_{E/F}\cI$ be the Iwahori group scheme of $G_0$.

\begin{lemma}\label{lemma-functoriality-trace}
	Let $A$ be an associative algebra in a symmetric monoidal category $(\cC,\otimes)$. Assume that $\cC$ admits all colimits and $\otimes$ commutes with colimits in each variables. Let $\phi_A\colon A\to A$ be an endomorphism of algebras. Let $f\geq 1$ be an integer. Let $B=A^{\otimes f}$ with natural associative algebra structure. Let $\phi_B\colon B\to B$ be the endomorphism given by $\phi_B(a_1\otimes\cdots\otimes a_f)=\phi_A(a_f)\otimes a_1\otimes\cdots\otimes a_{f-1}$. Then there is a natural morphism
	$$\can\colon\Tr(A,\phi_A)\to \Tr(B,\phi_B).$$
	Here $\Tr(A,\phi_A)$ is the vertical trace (or Hochschild homology) defined in \cite[\S 7.3.1]{Tame}.
\end{lemma}
\begin{proof}
	For an associative algebra $R\in \cC$, together with an endomorphism $\phi_R\colon R\to R$, the vertical trace $\Tr(R,\phi_R)$ is computed by the geometrization of the simplicial object 
	$$\HH_\bullet(R,\phi_{R})\colon\begin{tikzcd}
		\cdots\ar[r,shift left=0.6em]\ar[r,leftarrow,shift left=0.4em]\ar[r,shift left=0.2em]\ar[r,leftarrow]\ar[r,shift right=0.2em]\ar[r,leftarrow,shift right=0.4em]\ar[r,shift right=0.6em] & R\otimes R\otimes R\ar[r,shift left=0.4em]\ar[r,leftarrow,shift left=0.2em]\ar[r]\ar[r,leftarrow,shift right=0.2em]\ar[r,shift right=0.4em] & R\otimes R \ar[r,shift left=0.2em]\ar[r,leftarrow]\ar[r,shift right=0.2em] & R
	\end{tikzcd}$$
	where face maps are given by
	$$d_i\colon R^{\otimes n+1}\to R^{\otimes n},\ a_0\otimes\cdots\otimes a_n\mapsto \left\{\begin{aligned}
	& a_0\otimes\cdots\otimes(a_{i-1}a_{i})\otimes \cdots\otimes a_n ,&& \quad \text{if }0<i<n;\\
	& a_1\otimes\cdots\otimes a_{n-1}\otimes (a_na_0), && \quad\text{if }i=0; \\
	& a_1\otimes\cdots\otimes a_{n-2}\otimes (\phi_R(a_{n-1})a_n), && \quad\text{if }i=n.
\end{aligned}\right.$$
It suffices to construct a morphism $\can_\bullet\colon\HH_\bullet(B,\phi_B)\to \HH_\bullet(A,\phi_A)$ of simplicial objects. The map is given by $\can_n\colon (A^{\otimes f})^{\otimes n}\to A^{\otimes n}$, 
$$\can_n\big((a_0^1\otimes\cdots \otimes a_0^f)\otimes (a_1^1\otimes\cdots\otimes a_1^f)\otimes\cdots\otimes(a_n^1\otimes\cdots\otimes a_n^f)\big)=a_0^f\otimes a_1^f\otimes\cdots \otimes a_{n-1}^f\otimes (a_n^1a_0^1\cdots a_{n-1}^1a_n^2a_0^2\cdots a_{n-1}^2\cdots a_{n-1}^{f-1}a_n^f).$$
We can check that $\can_\bullet$ is a morphism of simplicial objects.
\end{proof}

\begin{prop}\label{prop-Weil-restriction}
	There are natural isomorphisms
	$$\Isoc_{G}\cong\Isoc_{G_0} \quad\text{and}\quad \Loc^{\widehat\unip}_{{}^LG,E}\cong \Loc^{\widehat\unip}_{{}^LG_0,F},$$
	and the functor $\LL^\unip$ is compatible with the above isomorphisms.
\end{prop}
\begin{proof}
	Consider the following situations:
	\begin{enumerate}
		\item $A=\mathrm{Hk}_{\cI}=L^+\cI\backslash LG/L^+\cI$, $\cC=$ category of correspondences between ind sifted placid perfect stacks over $k$.
		\item $A=\mathrm{St}_{\hat{G}}$, $\cC=$ category of correspondences between algebraic stacks over $\Lambda$.
		\item $A=\Ind\Shv_\fgen(\mathrm{Hk}_{\cI},\Lambda)$, $\cC=$ category of presentable $\Lambda$-linear categories.
		\item $A=\Ind\Coh(\St_{\hat{G}})$, $\cC=$ category of presentable $\Lambda$-linear categories.
	\end{enumerate}
	Then $A$ are associative algebras in $\cC$ equipped with an action of the geometric Frobenius $\phi$. Let $f=[E:F]$ and $q$ be the cardinality of residue field of $F$. If we take $B=A^{\otimes f}$ together with $\phi$-action as in Lemma \ref{lemma-functoriality-trace}, then we recover the parallel notions for $G_0$. For example, we have $LG_{0}\simeq \prod_{i=0}^{f-1} LG^{(i)}$ where $LG^{(i)}=LG\otimes_{k,x\mapsto x^q}k$ is the $q$-Frobenius twist of $LG$. Because we are considering perfect stacks, there are natural isomorphisms $LG\cong LG^{(i)}$. This induces an isomorphism $LG_0\cong \prod_{i=0}^{f-1}LG$ with the action of $\phi$ as in Lemma \ref{lemma-functoriality-trace}.
	
	By the functoriality of the construction in Lemma \ref{lemma-functoriality-trace}, we obtain a diagram that is commutative without two dotted arrows. 
	$$\begin{tikzcd}
		 \Tr(\Ind\Shv_\fgen(\mathrm{Hk}_{\cI_0},\Lambda),\phi)\ar[rr,"\Tr(\BB_{G_0}^\unip)","\cong"swap]\ar[rd,"\can"]\ar[dd,"\cong"] && \Tr(\Ind\Coh(\St_{\hat{G}_0}),\phi)\ar[rd,"\can"]\ar[dd,hook'] \\
		& \Tr(\Ind\Shv_\fgen(\mathrm{Hk}_{\cI},\Lambda),\phi)\ar[rr,"\Tr(\BB_{G}^\unip)","\cong"swap, near start]\ar[dd,"\cong", near start] && \Tr(\Ind\Coh(\St_{\hat{G}}),\phi)\ar[dd,hook'] \\
		\Ind\Shv_\fgen^{\unip}(\Isoc_{G_0},\Lambda) \ar[rr,"\LL_{G_0}^{\unip}",dotted, near start]\ar[rd,"\can","\cong"swap] && \Ind\Coh(\Loc_{{}^LG_0,F}^{\widehat\unip})\ar[rd,"\can","\cong"swap] \\
		& \Ind\Shv_\fgen^{\unip}(\Isoc_{G},\Lambda) \ar[rr,"\LL_{G}^{\unip}",dotted] && \Ind\Coh(\Loc_{{}^LG,E}^{\widehat\unip})
	\end{tikzcd}$$
	Here $\BB_{G}^\unip\colon \Ind\Shv_\fgen(\mathrm{Hk}_{\cI},\Lambda)\cong \Ind\Coh(\St_{\hat{G}})$ is the Bezrukavnikov equivalence. The canonical morphisms $\can\colon \Isoc_{G_0}\to\Isoc_G$ and $\can\colon\Loc_{{}^LG_0,F}^{\widehat\unip}\to \Loc_{{}^LG,F}^{\widehat\unip}$ on stacks are isomorphisms. It follows that the diagram is commutative. 
\end{proof}

\subsection{Localizing at regular semisimple parameters}\label{subsection-construct-of-morphism}
Let $G=\GL_{n,F}$ in this subsection. Assume that $\ell$ is banal for $\GL_n(F)$, i.e. $\ell$ does not divide the order of $\GL_n(k_F)$. In particular, we have $\Shv^{\widehat{\unip}}(\Isoc_G,\Lambda)^\omega=\Shv^\unip_\fgen(\Isoc_G,\Lambda
)$ and hence $\Shv^{\widehat{\unip}}(\Isoc_G,\Lambda)=\Ind\Shv^\unip_\fgen(\Isoc_G,\Lambda)$.

We fix a lift of geometric Frobenius $\phi\in W_F$ and a tame generator $t\in I_F/P_F$. It induces an isomorphisms 
$$\Loc_{{}^LG,F}^{\widehat\unip}\simeq\{(x,y)\in\hat{G}\times\cU^\wedge_{\hat{G}}\,|\,xy^qx^{-1}=y\}/\hat{G}$$
given by sending an $L$-parameter $\varphi$ to $(\varphi(\phi),\varphi(t))$. Here $\cU^\wedge_{\hat{G}}$ is the formal completion of $\hat{G}$ along the unipotent cone $\cU_{\hat{G}}\subseteq \hat{G}$. There is a morphism $\Loc_{{}^LG,F}^{\widehat\unip}\to \hat{G}\git\hat{G}\cong\hat{T}\git W$ given by sending $(x,y)$ to the conjugacy class $\xi$ of $x$. Note that this morphism is independent of the choice of $\phi$. 

Let $\xi\in (\hat{T}\git W)(\Lambda)$ be a regular semisimple element. Denote by
$$\widehat{L}_{\xi}\coloneqq \Loc_{{}^LG,F}^{\widehat\unip}\times_{\hat{T}\git W}(\hat{T}\git W)^\wedge_{\xi}$$
the formal completion of $\Loc_{{}^LG,F}^{\widehat\unip}$ along the preimage of $\xi$. Then $\Coh(\widehat{L}_{\xi})$ is the full subcategory of $\Coh(\Loc_{{}^LG,F}^{\widehat{\unip}})$ consisting of objects that are set-theoretically supported on the preimages of $\xi$.

\begin{prop}\label{prop-fix-parameter}
	The functor $\LL^\unip\circ (i_1)_*\colon\Rep^{\widehat{\unip}}(G(F),\Lambda)\to \Ind\Coh(\Loc_{{}^LG,F}^{\widehat{\unip}})$ sends objects in $\Rep^{\widehat\unip}(G(F),\Lambda)^\Adm_{\xi}$ to the full subcategory $\Coh(\widehat{L}_{\xi})\subseteq\Coh(\Loc_{{}^LG,F}^{\widehat{\unip}})$. 
\end{prop}
\begin{proof}
	By Lemma \ref{lemma-adm-unip-imply-compact}, objects in $\Rep^{\widehat\unip}(G(F),\Lambda)^\Adm_{\xi}$ are sent to the category of coherent sheaves. Recall the Bernstein center $Z^{\widehat{\unip}}_{G(F)}$ of $\Rep^{\widehat{\unip}}(G(F))$ is identified with $Z(\cH)$. By \cite[Proposition 4.43]{YZ-torsion-vanishing}, the $Z_{G(F)}^{\widehat\unip}$-action on an object $V$ in $\Rep^{\widehat\unip}(G(F),\Lambda)^\Adm_{\xi}$ agrees with the action of $\cO(\hat{T}\git W)$ on $\LL^\unip((i_1)_*V)$. Thus $\LL^\unip((i_1)_*V)$ is set theoretically supported on the preimage of $\xi$.
\end{proof}

Since $\ell$ is banal, there is an element $x_\std=\diag(x_1,\dots,x_n)\in \hat{T}(\Lambda)$ lifting $\xi$ such that $x_i=qx_j$ only when $j=i+1$. In particular, $P_{x_\std}=Q_{\xi}$. We fix such a choice of $x_\std$ in this subsection. Using $x_\std$, we identify $Q_{\xi}$ with a subset of $\{1,\dots,n-1\}$ by
$$Q_{\xi}= \{i\in\{1,\dots,n-1\}| x_i=qx_{i+1}\}.$$

\begin{prop}\label{prop-completion}
	There is an isomorphism of formal stacks
	$$\widehat{L}_{\xi}\simeq \Spf\bigg(\frac{\Lambda[v_i]_{i\in Q_{\xi}}[[u_1,\dots,u_n]]}{(u_iv_i)_{i\in Q_{\xi}}}\bigg)/\hat{T},$$
	where $\hat{T}$ acts on $u_1,\dots,u_n$ trivially and acts on $v_i$ with $i\in Q_{\xi}$ via the character $\alpha_i$. Here $\alpha_i=\varepsilon_i-\varepsilon_{i+1}\in \XX^\bullet(\hat{T})$ is the $i$-th simple root of $\hat{G}=\GL_n$.
\end{prop}
\begin{proof}
	Let $[q]\colon\hat{T}\git W\to \hat{T}\git W$ be the morphism induced by $x\mapsto x^q$ on $\hat{T}$. Let $(\hat{T}\git W)^{[q]}$ be the fixed point locus of $[q]$. It is easy to see that $(\hat{T}\git W)^{[q]}$ is always finite over $\Lambda$. Let $(\hat{T}\git W)^{[q]}_{e}$ be the connected component of $(\hat{T}\git W)^{[q]}$ at identity. 
	Since $\ell$ is banal, we can check that $(\hat{T}\git W)^{[q]}_e=\Spec(\Lambda)$ is reduced. The morphism $\Loc_{{}^LG,F}^{\widehat\unip}\to \hat{T}\git W$ sending $(x,y)$ to the conjugacy class of $y$ factors through $(\hat{T}\git W)^{[q]}_{e}$. Therefore the morphism $\Loc_{{}^LG,F}^{\widehat\unip}\to \hat{G}/\hat{G}$ given by $(x,y)\mapsto y$ factors through $\cU_{\hat{G}}/\hat{G}=(\hat{G}/\hat{G})\times_{\hat{T}\git W}\{e\}$. We see that
	$$\Loc_{{}^LG,F}^{\widehat\unip}=\{(x,y)\in\hat{G}\times\cU_{\hat{G}}|xy^qx^{-1}=y\}/\hat{G}.$$
	Let $\cN_{\hat{G}}$ be the nilpotent cone in the Lie algebra of $\hat{G}$. The exponential map (which is well-defined by the banal assumption) defines an isomorphism $\exp\colon \cN_{\hat{G}}\cong \cU_{\hat{G}}$ which is compatible with multiplication-by-$q$ map on $\cN_{\hat{G}}$ and $q$-power map on $\cU_{\hat{G}}$. We see that
	$$\Loc_{{}^LG,F}^{\widehat\unip}\simeq\{(x,Y)\in \hat{G}\times\cN_{\hat{G}}|\mathrm{Ad}_x(Y)=q^{-1}Y\}/\hat{G}.$$
	Denote $\widehat{L}_{\xi}^\square=\{(x,Y)\in (\hat{G})_{\xi}^\wedge\times \cN_{\hat{G}}|\mathrm{Ad}_x(Y)=q^{-1}Y\}$, where $(\hat{G})^\wedge_{\xi}$ is the completion of $\hat{G}$ along the semisimple conjugacy class $\xi$.	Define $\cW=\{(x,Y)\in (\hat{T})^\wedge_{x_{\std}}\times\cN_{\hat{G}}|\mathrm{Ad}_x(Y)=q^{-1}Y\}$. The conjugation action of $\hat{G}$ on $\hat{T}$ defines an isomorphism
	$$\hat{G}/\hat{T}\times(\hat{T})^\wedge_{x_{\std}}\to (\hat{G})^\wedge_{\xi},\quad (g\hat{T},t)\mapsto \mathrm{Ad}_g(t).$$
	Therefore we have is an isomorphism $\hat{G}\times^{\hat{T}}\cW\xrightarrow{\sim} \widehat{L}_{\xi}^\square$. It suffices to compute $\cW$. Let $x=\diag(\tilde{u}_1,\dots,\tilde{u}_n)\in \hat{T}^\wedge_{x_\std}(R)$ for some $\Lambda$-algebra $R$. Then $x$ is congruent to $x_\std$ modulo the nilpotent radical of $R$. A direct computation shows that if $(x,Y)\in \cW(R)$, then $Y$ has the form
	$$Y=\begin{pmatrix}
		0 & \\ v_1 & 0 \\ &\ddots & \ddots \\ && v_{n-1} & 0
	\end{pmatrix}$$
	subject to the relations $(\tilde{u}_i-q\tilde{u}_{i+1})v_i=0$ for $i=1,\dots,n-1$. We see that $v_i=0$ unless $i\in Q_{\xi}$. Put $u_i=(\tilde{u}_i-\frac{x_{i+1}}{x_i}\tilde{u}_{i+1})$ if $i=1,\dots, n-1$ and $u_n=\tilde{u}_n-x_n$. We obtain the presentation as in the proposition. 
\end{proof}

As $\Loc^{\unip}_{{}^LB^-,F}$ is a derived fixed points stack, there is a natural isomorphism $\omega_{\Loc^{\unip}_{{}^LB^-,F}}\simeq \cO_{\Loc^{\unip}_{{}^LB^-,F}}$ on $\Loc^{\unip}_{{}^LB^-,F}$. Thus we have an isomorphism $\CohSpr\cong (q^\spec)_* \cO_{\Loc^{\unip}_{{}^LB^-,F}}$. Let $j_{\xi}\colon \widehat{L}_{\xi}\to \Loc^{\widehat\unip}_{{}^LG,F}$ be the natural embedding.

For a subset $P\subset Q_{\xi}$, let $X_P\subseteq \widehat{L}_{\xi}$ be the closed substack defined by the equations $v_i=0$ for every $i\in Q_{\xi}\backslash P$. Recall that $\varsigma\colon \hat{T}\to \hat{T}\git W$ is the natural projection.

\begin{prop}\label{prop-Springer}
	There is a natural isomorphism 
	$$(j_{\xi})^*\CohSpr\cong \bigoplus_{x\in \varsigma^{-1}(\xi)}\cO_{X_{P_{x}}}$$
	of pro-coherent sheaves on $\widehat{L}_{\xi}$ and a natural isomorphism
	$$(j_{\xi})^!\CohSpr\cong \bigoplus_{x\in \varsigma^{-1}(\xi)}\omega_{X_{P_{x}}}$$
	of ind-coherent sheaves on $\widehat{L}_{\xi}$.
\end{prop}
\begin{proof}
	Denote $\widehat{L}_{\xi,B^-}=\Loc_{{}^LB^-,F}^{\unip}\times_{\Loc_{{}^LG,F}^{\widehat\unip}} \widehat{L}_{\xi}$ with a morphism $\widehat{q}^\spec\colon \widehat{L}_{\xi,B^-}\to \widehat{L}_{\xi}$. We claim that there is a natural isomorphism
	$$\widehat{L}_{\xi,B^-}\cong \bigsqcup_{x\in \varsigma^{-1}(\xi)}X_{P_{x}}.$$
	Note that $\widehat{L}_{\xi,B^-}$ is a priori derived. We first compute the underlying classical stack $(\widehat{L}_{\xi,B^-})^{\mathrm{cl}}$. Let $\widehat{L}_{\xi,B^-}^\square$ denote the framed version of $\widehat{L}_{\xi,B^-}$, which classifies the data $(x,y,\hat{B'})$ where $(x,y)\in\widehat{L}_\xi^\square$ and $\hat{B}'$ is a Borel subgroup of $\hat{G}$ containing $x$ and $y$. Let $(x,y,\hat{B}')$ be an $R$-point of $\widehat{L}_{\xi,B^-}^\square$ for an ordinary $\Lambda$-algebra $R$. By Proposition \ref{prop-completion}, after conjugation, we can assume that $x$ lies in the maximal torus $\hat{T}$ and is congruent to $x_\std$ modulo the nilpotent radical of $R$. Because $x_\std$ is regular semisimple, the Borel subgroups containing it are exactly $w\hat{B}^-w^{-1}$ for $w\in W$. It defines a decomposition
	$$(\widehat{L}_{\xi,B^-}^\square)^{\mathrm{cl}}=\bigsqcup_{w\in W} \hat{G}\times^{\hat{T}} \{(x,y)\in \cW|y\in w\hat{B}^-w^{-1}\},$$
    and hence a decomposition
    $$(\widehat{L}_{\xi,B^-})^{\mathrm{cl}}=\bigsqcup_{w\in W} \{(x,y)\in \cW|y\in w\hat{B}^-w^{-1}\}/\hat{T}.$$
	By the description $\widehat{L}_{\xi}$ in Proposition \ref{prop-completion}, we have $\{(x,y)\in \cW|y\in w\hat{B}^-w^{-1}\}/\hat{T}=X_{P_{w^{-1}(x)}}$. Therefore we have an isomorphism
	$$(\widehat{L}_{\xi,B^-})^\mathrm{cl}=\bigsqcup_{x\in \varsigma^{-1}(\xi)}X_{P_x}.$$
	By \cite[Proposition 2.3.7]{Zhu-coherent-sheaf}, the formal stack $\widehat{L}_{\xi,B^-}$ is quasi-smooth of virtual dimension 0. We have just shown that $(\widehat{L}_{\xi,B^-})^{\mathrm{cl}}$ is of dimension 0. Thus the stack $\widehat{L}_{\xi,B^-}$ is classical. This proves the claim.
	
	Now the results follow from base change.
\end{proof}

Now let $V\in \Rep^{\widehat\unip}(G(F),\Lambda)^\Adm_{\xi}$. We have 
$$\begin{aligned}
	V^I &= \Hom_{\Rep^{\widehat\unip}(G(F),\Lambda)^\Adm_{\xi}}(\cInd_I^{G(F)}(\Lambda),V) \\
	 &\cong\Hom_{\Shv^{\widehat{\unip}}(\Isoc_G,\Lambda)}((i_1)_*\cInd_I^{G(F)}(\Lambda),(i_1)_*V) \\
	 &\cong\Hom_{\Coh(\Loc^{\widehat{\unip}}_{{}^LG,F})}(\CohSpr,\LL^\unip((i_1)_*V))\\
	 &\cong\Hom_{\Pro\Coh(\widehat{L}_{\xi})}((j_{\xi})^*\CohSpr,\LL^\unip((i_1)_*V)) \\
	 &\cong\bigoplus_{x\in \varsigma^{-1}(\xi)} \Hom_{\Pro\Coh(\widehat{L}_{\xi})}(\cO_{X_{P_x}},\LL^\unip((i_1)_*V)).
\end{aligned}$$

\begin{lemma}\label{lemma-decomp-agree}
	The decomposition above agrees with the decomposition defined in Proposition \ref{prop-decomp-R}, i.e. there are natural isomorphisms $V^{I,x}\cong \Hom_{\Pro\Coh(\widehat{L}_{\xi})}(\cO_{X_{P_x}},\LL^\unip((i_1)_*V))$ for $x\in \varsigma^{-1}(\xi)$.
\end{lemma}
\begin{proof}
	It suffices to show that the action of $\cR$ on $\Hom_{\Pro\Coh(\widehat{L}_{\xi})}(\cO_{X_{P_x}},\LL^\unip((i_1)_*V))$ is set theoretically supported on the point $x\in\hat{T}(\Lambda)$. For $\lambda\in \XX_\bullet(T)$, let $J_\lambda\in \Shv_\fgen(L^+\cI\backslash LG/L^+\cI)$ be the normalized Wakimoto sheaf. Concretely, if $\lambda$ is anti-dominant, then $J_\lambda=\Delta_\lambda$, and if $\lambda$ is dominant, then $J_\lambda=\nabla_\lambda$. In general, define $W_\lambda=W_\mu* W_\nu$ for $\lambda=\mu+\nu$ with $\mu$ dominant and $\nu$ anti-dominant.
	By \cite[Proposition 4.42]{YZ-torsion-vanishing}, the element $\theta_\lambda\in \cH$ acts on $(i_1)_*\cInd_I^{G(F)}\Lambda$ via the $S$-operator $S_{J_\lambda}$ corresponds to the Wakimoto sheaf $J_\lambda$. It suffices to compute the action of the $S$-operator $S^\spec_{J_\lambda}$ on the spectral side. 
	
	Let $\mathrm{St}_{\hat{G}}\coloneqq\hat{U}^-/\hat{B}^-\times_{\hat{G}/\hat{G}}\hat{U}^-/\hat{B}^-$ be the Steinberg variety. We have a correspondence
	$$\BB\hat{T}\xleftarrow{\gamma} \hat{U}^-/\hat{B}^-\xrightarrow{\delta} \St_{\hat{G}}.$$
	Under Bezrukavnikov equivalence, the Wakimoto sheaf $J_\lambda$ (together with Frobenius structure) corresponds to the object $J_\lambda^\spec=\delta_*\gamma^!(\omega_{\BB\hat{T}}(\lambda))$. The functor $\delta_*\gamma^!\colon \Coh(\BB\hat{T})\to \Coh(\St_{\hat{G}})$ is monoidal. Here $\Coh(\BB\hat{T})$ is endowed with $!$-tensor product $\otimes^!$, and $\Coh(\St_{\hat{G}})$ is endowed with the monoidal structure given by $!$-convolution. Taking traces on $\phi_*$, it induces a commutative diagram
	$$\begin{tikzcd}[column sep= huge ]
		\Tr(\Ind\Coh(\BB\hat{T}),\phi_*)\ar[r,"\mathrm{Tr}(\delta_*\gamma^!)"]\ar[d,"\cong"swap] & \Tr(\Ind\Coh(\St_{\hat{G}}),\phi_*)\ar[d,hook'] \\
		\Ind\Coh(\hat{T}/\hat{T}) \ar[r,"(q^\spec)_*(p^\spec)^!"] & \Ind\Coh(\Loc_{{}^LG,F}^{\widehat\unip}),
	\end{tikzcd}$$
	where the lower arrow is induced by the correspondence
    $$\hat{T}/\hat{T}\xleftarrow{p^\spec} \Loc_{{}^LB^-,F}^{\unip}\xrightarrow{q^\spec} \Loc_{{}^LG,F}^{\widehat\unip}.$$
	The functor $(q^\spec)_*(p^\spec)^!$ sends the unit object $\omega_{\hat{T}/\hat{T}}$ in $\Ind\Coh(\hat{T}/\hat{T})$ to the unit object $\CohSpr$ in $\Ind\Coh(\Loc_{{}^LG,F}^{\widehat\unip})$. By the same proof as \cite[Proposition 4.43]{YZ-torsion-vanishing}, the action of the $S$-operator $S_{\omega_{\BB\hat{T}}(\lambda)}^\spec$ on $\omega_{\hat{T}/\hat{T}}$ is given by $\lambda\in \cO(\hat{T}/\hat{T})=\End(\omega_{\hat{T}/\hat{T}})$. It induces the action of $S_{J_\lambda}^\spec$ on $\CohSpr$ through the functor $(q^\spec)_*(p^\spec)^!$. After taking completion at $\xi$, the morphism
	$$\widehat{L}_{\xi,B^-}\xrightarrow{\widehat{p}^\spec} \hat{T}/\hat{T}$$
	restricts to 
	$$X_{P_{w^{-1}(x_\std)}}\xrightarrow{(x,y)\mapsto w^{-1}(x)}\hat{T}/\hat{T}$$ 
	on each compotent $X_{P_{w^{-1}(x_\std)}}$ for $w\in W$. Thus the action of $S_{J_\lambda^\spec}$ on $\CohSpr$ preserves the decomposition in Proposition \ref{prop-Springer}, and acts by $\lambda(w^{-1}(x))$ on the component $\cO_{X_{P_{w^{-1}(x_\std)}}}$. The result now follows. 
\end{proof}

Now we are ready to prove Proposition \ref{prop-Hecke-dim}. Although it is possible to prove it using classical methods, we choose to use the machinery of categorical local Langlands correspondence because we can construct generators needed in Theorem \ref{thm-main} directly on the spectral side.

\begin{proof}[Proof of Proposition \ref{prop-Hecke-dim}]
	By Lemma \ref{lemma-decomp-agree}, we see that
	$$\LL^\unip((i_1)_* (\widehat\delta_x))\cong \cO_{X_{P_x}}$$
	as pro-objects in $\Coh(\widehat{L}_\xi)$ for any $x\in \varsigma^{-1}(\xi)$. Let $x,x'\in \varsigma^{-1}(\xi)$ be two elements. It suffices to show that
	$$\Hom_{\Pro\Coh(\widehat{L}_\xi)}(\cO_{X_{P_x}},\cO_{X_{P_{x'}}})\cong \Hom_{G(F)}(\widehat{\delta}_x,\widehat{\delta}_{x'})$$
	is free of rank one over $Z(\cH)^\wedge_\xi\cong \cO(\widehat{L}_\xi)$. The right hand side is concentrated in degree 0 because $\ell$ is banal. Thus it suffices to compute $H^0\Hom_{\Pro\Coh(\widehat{L}_\xi)}(\cO_{X_{P_x}},\cO_{X_{P_{x'}}})$. We have $\cO_{X_{P_x}}=\cO_{\widehat{L}_\xi}/(v_j)_{j\in Q_\xi\backslash P_x}$. Thus there is an inclusion
	$$H^0\Hom_{\Pro\Coh(\widehat{L}_\xi)}(\cO_{X_{P_x}},\cO_{X_{P_{x'}}})\hookrightarrow H^0\Hom_{\Pro\Coh(\widehat{L}_\xi)}(\cO_{\widehat{L}_\xi},\cO_{X_{P_{x'}}})$$
	It is easy to see that $H^0\Hom_{\Pro\Coh(\widehat{L}_\xi)}(\cO_{\widehat{L}_\xi},\cO_{X_{P_{x'}}})\cong \cO(\widehat{L}_\xi)=\Lambda[[u_1,\dots,u_n]]$. Denote $P=P_x\cap P_{x'}$. An element $f\in \Lambda[[u_1,\dots,u_n]]$ defines a homomorphism $\cO_{X_{P_x}}\to \cO_{X_{P_{x'}}}$ if and only if $f$ is divisible by $u_j$ for any $j\in P_x\backslash P$. It follows that $H^0\Hom_{\Pro\Coh(\widehat{L}_\xi)}(\cO_{X_{P_x}},\cO_{X_{P_{x'}}})$ is isomorphic to the ideal in $\Lambda[[u_1,\dots,u_n]]$ generated by $\prod_{j\in P_x\backslash P}u_j$, which is free of rank one.
	
	A generator of $H^0\Hom_{\Pro\Coh(\widehat{L}_\xi)}(\cO_{X_{P_x}},\cO_{X_{P_{x'}}})$ can be constructed as follows: Denote $P=P_{x'}\cap P_x$, then we have a surjection
	$$\beta\colon \cO_{X_{P_x}}\twoheadrightarrow \cO_{X_P}$$
	and an injection
	$$\alpha=\prod_{j\in P_x\backslash P}u_j\colon \cO_P\hookrightarrow \cO_{P_{x'}}.$$
	Then the composition $\alpha\circ\beta$ is a generator in $H^0\Hom_{\Pro\Coh(\widehat{L}_\xi)}(\cO_{X_{P_x}},\cO_{X_{P_{x'}}})$.
\end{proof}

\begin{eg}\label{eg-GL2-spr}
	Consider the case $G=\GL_{2,F}$ and $\xi=$ conjugacy class of $\diag(q^{\frac12},q^{-\frac12})$. Then $Q_\xi=\{1\}$ is a singleton and $\widehat{L}_{\xi}\simeq \Spf\big(\frac{\Lambda[v_1][[u_1,u_2]]}{(u_1v_1)}\big)/\GG_m^2$. Denote $S=\frac{\Lambda[v_1][[u_1,u_2]]}{(u_1v_1)}$. Then $\cO_{X_{\emptyset}}=S/v_1$ and $\cO_{X_{\{1\}}}=S$. There are exact sequences
	$$0\to (S/u_1)(\alpha_1)\xrightarrow{v_1}S\to S/v_1\to 0$$
	and
	$$0\to S/v_1\xrightarrow{u_1}S\to S/u_1\to 0.$$
	Here $(S/u_1)(\alpha_1)$ means $S/u_1$ tensored with the line bundle $\cO(-\alpha_1)$ pulled back from $V_{\xi}\to \BB \hat{T}$. The homomorphism $S\to S/v_1$ (resp. $S/v_1\stackrel{u_1}{\hookrightarrow}S$) gives a generator in $\Hom(\cO_{X_{\{1\}}},\cO_{X_{\emptyset}})$ (resp. $\Hom(\cO_{X_\emptyset},\cO_{X_{\{1\}}})$).
\end{eg}

For $x,x'\in \varsigma^{-1}(\xi)$ with $P_x\supseteq P_{x'}$, the morphism $\alpha_{x,x'}^\Hk$ in Theorem \ref{thm-main} (1) can be chosen to be associated to
$$\prod_{j\in P_x\backslash P_{x'}}u_j\colon \cO_{X_{P_{x'}}}\hookrightarrow \cO_{X_{P_x}}$$
on the spectral side, and the morphism $\beta_{x,x'}^\Hk$ in Theorem \ref{thm-main} (2) can be chosen to be associated to the natural projection
$$\cO_{X_{P_{x}}}\twoheadrightarrow \cO_{X_{P_{x'}}}.$$
We fix such choices of $\alpha_{x,x'}^\Hk$ and $\beta_{x,x'}^\Hk$ from now on.

By the construction of the morphisms, they have the following properties.
\begin{prop}\label{prop-composition}
	If $x,x',x''$ are three elements in $\varsigma^{-1}(\xi)$ with $P_{x''}\subseteq P_{x'}\subseteq P_x$, then we have 
	$$\alpha_{x',x''}\circ\alpha_{x,x'}=\alpha_{x,x''}\colon V^{I,x}\to V^{I,x''}$$
	and 
	$$\beta_{x,x'}\circ\beta_{x',x''}=\beta_{x,x''}\colon V^{I,x''}\to V^{I,x}$$
	for any $V\in\Rep^{\widehat\unip}(G(F),\Lambda)^\Adm_{\xi}$.
\end{prop}
\begin{proof}
	Clear by construction.
\end{proof}
\begin{prop}\label{prop-comp-unip}
	If $x,x'\in \varsigma^{-1}(\xi)$ with $P_{x'}\subsetneq P_x$. Let $V\in \Rep^{\widehat\unip}(G(F),\Lambda)^\Adm_{\xi}$. Then the compositions 
	$$\alpha_{x,x'}\circ\beta_{x,x'}\colon V^{I,x'}\to V^{I,x'}$$
	$$\beta_{x,x'}\circ\alpha_{x,x'}\colon V^{I,x}\to V^{I,x}$$
	are nilpotent on the cohomologies of $V$.
\end{prop}
\begin{proof}
	This follows as $\alpha_{x,x'}\circ\beta_{x,x'}$ is induced by $\prod_{i\in P_x\backslash P_{x'}}u_i\colon \cO_{X_{P_{x'}}}\to \cO_{X_{P_{x'}}}$, and each $u_i\in \cO(\widehat{L}_\xi)=Z(\cH)^\wedge_\xi$ acts nilpotently on the cohomologies of $V^I$. Similar for $\beta_{x,x'}\circ\alpha_{x,x'}$.
\end{proof}

\begin{prop}
	Let $V\in \Rep^{\widehat\unip}(G(F),\Lambda)^\Adm_{\xi}$. If $x,x'\in \varsigma^{-1}(\xi)$ with $P_x=P_{x'}$, then the morphisms
	$$\alpha_{x,x'}\colon V^{I,x}\to V^{I,x'}\quad\text{and}\quad\beta_{x,x'}\colon V^{I,x'}\to V^{I,x}$$
	are mutually inverse to each other.
\end{prop}
\begin{proof}
	Clear by construction.
\end{proof}

\section{Proof of Theorem \ref{thm-main}}\label{section-proof}
In this section, we prove Theorem \ref{thm-main}. For this, we need to introduce some axillary groups $\sG_a$ for $0\leq a\leq n$ that are inner forms of $\sG$, and use the geometry of the Shimura variety associated to $\sG_a$.

\subsection{Changing signatures}
Assume that we are in the situation of Theorem \ref{thm-main}. Let $\sG$ be the definite unitary group over $\QQ$ as in \S\ref{subsection-def-auto}. Fix an isomorphism $\overline{\QQ}_p\simeq\CC$ such that the complex embedding $\tau_0\colon \sK\hookrightarrow\CC$ and the embedding $\sK\hookrightarrow \overline\QQ_p$ defined by $\wp$ (the place of $\sK$ below $v$) are identified. It defines a bijection $\Sigma_\infty\simeq \Sigma_p$ between the set of real embeddings of $\sF^+$ and the set of embeddings of $\sF^+$ into $\overline{\QQ}_p$. Let $(a_\tau)_{\tau\in\Sigma_\infty\simeq \Sigma_p}$ be signatures of $\sG$. By assumptions, $a_\tau=0$ or $n$ for all $\tau$. 

We first prove the following result, which allows us to change signatures at infinite places.

\begin{lemma}\label{lemma-change-sign}
	Fix a real embedding $\tau_1\colon \sF^+\hookrightarrow\RR$. There is a $\star$-Hermetian form $\langle-,-\rangle'$ on $\sV$ that defines a PEL datum $(\sD^\op,\star,\sV,\langle-,-\rangle')$ as \S\ref{subsection-def-auto}. It defines a definite unitary similitude group $\sG'$, such that the following conditions hold:
	\begin{enumerate}
		\item There is an isomorphism $\sG'\simeq\sG$ of reductive groups over $\QQ$.
		\item The $\RR$-points of $\sG'$ is compact modulo center. Let $(a'_\tau)_{\tau\in\Sigma_\infty}$ be the signatures of $\sG'$. Then $a_\tau'=a_\tau$ if $\tau\neq \tau_1$, and $a'_{\tau_1}=n-a_{\tau_1}$.
	\end{enumerate} 
\end{lemma}
\begin{proof}
	We assume that $a_{\tau_1}=0$. The case $a_{\tau_1}=n$ can be treated similarly. By \cite[\S 8]{Helm-JL}, it suffices to construct a cohomology class in $\alpha\in H^1(\sF^+,\sG^1)$ that is given by the class $[\rU(n,0)]\in H^1(\sF^+_{\tau_1},\sG^1)$ associated to the pure inner form $\rU(n,0)$ of $\sG^1_{\sF^+_{\tau_1}}\simeq \rU(0,n)$ at $\tau_1$, such that the associated inner class $\alpha_\ad\in H^1(\sF^+,\sG_\ad^1)$ is trivial. There is an exact sequence
	$$H^1(\sF^+,\sG^1)\to \bigoplus_{u\text{ place of }\sF^+}H^1(\sF^+_{u},\sG^1)\to \pi_0(Z(\hat{\sG}^1)^{\Gal(\overline\sF^+/\sF^+)})^D,$$
	where $(-)^D$ is the Pontryagin dual. Here, the second morphism is defined by taking sum of local invariances
	$$\alpha_{\sG^1,u}\colon H^1(\sF^+_{u},\sG^1)\to \pi_0(Z(\hat{\sG}^1)^{\Gal(\overline\sF_{u}^+/\sF_{u}^+)})^D$$
	for all places $u$ of $\sF^+$. Moreover, if $u$ is a finite place, then $\alpha_{\sG^1,u}$ is an isomorphism. The image of the class $[\rU(n,0)]\in H^1(\sF^+_{\tau_1},\sG^1)$ under $\alpha_{\sG^1,\tau_1}$ is $(-1)^{n}\in \pi_0(Z(\hat{\sG}^1)^{\Gal(\overline\sF^+_{\tau_1}/\sF^+_{\tau_1})})^D\cong \ZZ/2\ZZ$. Thus if $n$ is even, then we can find a class $\alpha\in H^1(\sF^+,\sG^1)$ that is trivial at all the places except at $\tau_1$, and is given by $[\rU(n,0)]$ at $\tau_1$. It defines a pure inner form $\sG'$ with the desired properties.

    Now we treat the case when $n$ is odd. By Chebotarev density theorem, there is a finite place $w$ of $\QQ$ such that $w$ splits completely in $\sF^+$ and is unramified inert in $\sK$. Fix a place $\tilde{w}$ of $\sF^+$ above $w$. We first construct a unitary similitude group $\sG'$ satisfying (2) and is isomorphic to $\sG$ when restricted to all local fields $\QQ_{w'}$ with $w'\neq w$.  We have $\pi_0(Z(\hat{\sG}^1)^{\Gal(\overline\sF^+/\sF^+)})\cong\ZZ/2\ZZ$  and $H^1(\sF^+_{\tilde{w}},\sG^1)\cong \ZZ/2\ZZ$. Thus we can find a class $\alpha\in H^1(\sF^+,\sG^1)$ such that it is trivial away from $\tau_1$ and $\tilde{w}$, it restricts to $-1\in H^1(\sF^+_{\tilde{w}},\sG^1)\cong\ZZ/2\ZZ$ at $\tilde{w}$, and it restricts to $[\rU(n,0)]$ at $\tau_1$. 
	
	Note that the class $\alpha_\ad\in H^1(\sF^+,\sG^1_\ad)\simeq H^1(\QQ,\sG_\ad)$ is trivial at all places: The group $H^1(\sF^+_{\widetilde{w}},\sG^1_\ad)\simeq \pi_0(Z(\hat{\sG}^1_\ad)^{\Gal(\overline\sF^+_{\tilde{w}}/\sF^+_{\tilde{w}})})^D$ is trivial. Over the infinity place $\tau_1$, the class $\alpha_\ad$ is given by unique compact inner form, and thus is trivial. By Hasse principle for adjoint groups, the inner class $\alpha_\ad$ is trivial, and thus there is an isomorphism $\Psi\colon \sG\simeq\sG'$ of reductive groups over $\QQ$.
\end{proof}

Recall that $v^+$ is the $p$-adic place of $\sF^+$ lying below $v$. Let $\Sigma_{v^+}\subseteq\Sigma_p$ be the subset of embeddings $\sF^+\hookrightarrow \overline\QQ_p$ that induce the valuation $v^+$ on $\sF^+$. By Lemma \ref{lemma-change-sign}, we can assume that there is exactly one element $\tau^\circ\in \Sigma_{v^+}$ such that $a_{\tau^\circ}=n$, and for any other element $\tau\in \Sigma_p\backslash \{\tau^\circ\}$, $a_{\tau}=0$. By assumption $F=\sF_v\neq \QQ_p$, the set $\Sigma_{v^+}$ contains at least two elements. Fix an element $\tau^\bullet\in \Sigma_{p,v^+}\backslash \{\tau^\circ\}$.

\begin{lemma}\label{lemma-pure-inner-form}
	Let $0\leq a\leq n$ be an integer. There is a unitary similitude group $\sG_a$ induced from a PEL datum $(\sD^\op,\star,\sV,\langle-,-\rangle_a)$ as in \S\ref{subsection-def-auto}, such that the following conditions hold:
	\begin{enumerate}
		\item $\sG_a$ is a pure inner twist of $\sG$, i.e. there is a $\sG$-torsor $Z$ over $\QQ$ such that $\sG_a=\mathrm{Aut}_\sG(Z)$. Moreover, $Z$ is trivial over $\AAA_f$. In particular, it induces an isomorphism $\sG_a(\AAA_f)\simeq \sG(\AAA_f)$ on $\AAA_f$-points.
		\item The signatures $(a'_\tau)_{\tau\in\Sigma_p}$ of $\sG_a$ are given as follows:
			$$a'_{\tau^\bullet}=a, \quad a'_{\tau^\circ}=n-a\quad\text{and}\quad a'_\tau=0 \text{ for }\tau\neq \tau^\bullet,\tau^\circ.$$
	\end{enumerate}
\end{lemma}
\begin{proof}
	This is a special case of \cite[Lemma 2.9]{Helm-Tian-Xiao}. 
\end{proof}

Let $0\leq a\leq n$ be an integer. Let $(\sD^\op,\star,\sV,\langle-,-\rangle_a)$ be a PEL datum defined in Lemma \ref{lemma-pure-inner-form}. Let $(\sG_a,\sX_a)$ be the Shimura datum of PEL type associated to it. Let $\mu_a$ be the Hodge cocharacter of $(\sG_a,\sX_a)$. Under the identification $\XX_\bullet(\sT)=\ZZ\times \prod_{\tau\in\Sigma_p}\ZZ^n$, the Hodge cocharacter is given by
$$\mu_a=(1,(\underbrace{1^a,0^{n-a}}_{\text{at }\tau^\bullet}),(\underbrace{1^{n-a},0^{a}}_{\text{at }\tau^\circ}),(0^n),\dots,(0^n)).$$

\subsection{Cohomology of Shimura varieties via coherent sheaves}

Let $0\leq a\leq n$ be an integer. Let $\sG_a$ be the pure inner form of $\sG$ defined in Lemma \ref{lemma-pure-inner-form}. 

Recall that there is a decomposition $\sG_{\QQ_p}=\Res_{F/\QQ_p}\GL_n\times H$. Denote $G_0=\Res_{F/\QQ_p}\GL_n$ and $\cI_0=\Res_{O_F/\ZZ_p}\cI$. Assume for now that $K_p=I\times K_p^v\subseteq G(F)\times H(\QQ_p)$ is an Iwahori subgroup. Let $K^p\subseteq \sG(\AAA_f^p)$ be a neat open compact subgroup. Denote $K=K^pK_p$ and $K^v=K^pK_p^v$. Let $\Sh_K(\sG_a,\sX_a)$ be the Shimura variety associated to $(\sG_a,\sX_a)$ of level $K$, which is a compact smooth variety over the reflection field $\sE=\sE_{\mu_a}$. Let
$$\Sh_K^\ext(\sG_a,\sX_a)=\bigsqcup_{\ker^1(\QQ,\sG_a)}\Sh_K(\sG_{a},\sX_{a})$$
be the extended Shimura variety as in \cite{Xiao-Zhu}. Here $\ker^1(\QQ,\sG_a)$ is the kernel of $$H^1(\QQ,\sG_a)\to \bigoplus_{w} H^1(\QQ_w,\sG_a),$$ 
where $w$ runs through all places of $\QQ$. We remark that $\sSh_K^\mathrm{ext}(\sG_a,\sX_a)$ is the moduli space of abelian varieties with PEL structure $(\sD^\op,\star,\sV,\langle-.-\rangle_a)$. The fixed isomorphism $\overline\QQ_p\simeq \CC$ defines a $p$-adic place $\fp$ of $\sE$. Let $\mathscr{S}^\ext_K(\sG_a,\sX_a)$ be the integral canonical model of $\sSh^\ext_{K^pK_p}(\sG_a,\sX_a)$ over $O_{\sE_\fp}$. Denote by
$$\mathscr{S}^\ext_{K}(\sG_a,\sX_a)_{k}^\perf$$
the prefection of the special fiber of $\mathscr{S}^\ext_{K}(\sG_a,\sX_a)$ over $k=\overline{\FF}_p$.

Let $\mu_{a,1}\colon \GG_{m,\overline{\QQ}_p}\to (G_0)_{\overline\QQ_p}$ be the projection of $\mu_a$ to the first component and $\mu_{a,2}\colon\GG_{m,\overline{\QQ}_p}\to H_{\overline\QQ_p}$ the projection of $\mu_a$ to the second component. We need a slight variant of Igusa stacks constructed in \cite{Igusa}. If $\mu$ is a dominant coweight of $G_0$, we denote by $\Isoc_{G_0,\leq\mu}$ the closed substack that is the union of $\Isoc_{G_0,b}$ with $b\in B(G,\mu)$. Let $i_{\leq \mu}\colon \Isoc_{G_0,\leq\mu}\hookrightarrow \Isoc_{G_0}$ be the closed embedding. Let $\mu^*$ denote the dominant Weyl translate of $-\mu$. Let $\Sht^\loc_{\cI_0,\mu}$ denote the closed substack of $\Sht^\loc_{\cI_0}$ classifying local shtukas $\sigma^*\cE\dashrightarrow \cE$ bounded by the $\mu^*$. See \cite[\S 6.1.2]{Tame} for the definition (but note that this stack is denoted by $\Sht^\loc_{\cI_0,\mu^*}$ in \emph{loc. cit.}). 

\begin{prop}\label{prop-Igusa}
	There is a perfect stack $\Igs^a_{K^v}$ over $k$ that fits into a Cartesian diagram
	$$\begin{tikzcd}
		\mathscr{S}^\ext_K(\sG_a,\sX_a)_{k}^\perf \ar[r,"\loc_v"]\ar[d,"\Nt^{\mathrm{glob}}"swap] & \Sht^\loc_{\cI_0,\mu_{a,1}} \ar[d,"\Nt"] \\
		\Igs_{K^v}^a \ar[r,"\loc_v^0"] & \Isoc_{G_0,\leq \mu^*_{a,1}}.
	\end{tikzcd}$$
\end{prop}
\begin{proof}
	It is proved in \cite[\S6.5.3]{Igusa} (see also \cite[Proposition 6.2]{Tame}) that there exists a perfect stack $\Igs_{K^p}^a$ and a Cartesian diagram
	$$\begin{tikzcd}
		\mathscr{S}^\ext_K(\sG_a,\sX_a)_{k}^\perf \ar[r,"\loc_p"]\ar[d,"\Nt^{\mathrm{glob}}"swap] & \Sht^\loc_{\cK_{p},\mu_{a}} \ar[d,"\Nt"] \\
		\Igs_{K^p}^a \ar[r,"\loc_p^0"] & \Isoc_{\sG_{\QQ_p},\leq \mu_a^*}.
	\end{tikzcd}$$
	Moreover, the morphism $\Nt^{\mathrm{glob}}$ is a ind-proper and surjective. Here $\cK_p$ is the Iwahoric group scheme over $\ZZ_p$ associated to $K_p$. Since the cocharacter $\mu_{a,2}$ is central, we have
	$$\Sht^\loc_{\cK_{p},\mu_{a}}\simeq \Sht_{\cI_0,\mu_{a,1}}\times \BB_{\textrm{pro\'et}}K_p^v.$$
	Let $b_{2}\in B(H)$ be the basic element in $B(H,\mu_{a,2}^*)$. Since $\mu_{a,2}$ is a central cocharacter defined over $\QQ_p$, the element $b_{2}$ is the $\sigma$-conjugacy class of $\mu^*_{a,2}(p)\in H(\breve{\QQ}_p)$. Thus the associated extended purely inner form $H_{b_2}$ is canonically identified with $H$. Hence $ \Isoc_{\sG_{\QQ_p},\leq \mu_{a,1}^*}\simeq \Isoc_{G_0,\leq\mu_{a,1}^*}\times \BB_{\textrm{pro\'et}}H(\QQ_p)$. Pullback along the pro-\'etale cover $*\to \BB_{\textrm{pro\'et}}H(\QQ_p)$, we get a Cartesian diagram
	$$\begin{tikzcd}
		\mathscr{S}^\ext_{K^pI}(\sG_a,\sX_a)_{k}^\perf \ar[r,"\loc_v"]\ar[d,"\Nt^\mathrm{glob}"swap]& \Sht^\loc_{\cI_0,\mu_{a,1}}\ar[d,"\Nt"] \\
		\widetilde\Igs^a_{K^p} \ar[r,"\loc_v^0"] & \Isoc_{G_0,\leq \mu_{a,1}^*}
	\end{tikzcd}$$
	where $\mathscr{S}^\ext_{K^pI}(\sG_a,\sX_a)_{k}^\perf\to \mathscr{S}^\ext_K(\sG_a,\sX_a)_{k}^\perf$ is a pro-\'etale $K_p^v$-torsor and $\widetilde\Igs^a_{K^p}\to \Igs^a_{K^p}$ is a pro-\'etale $H(\QQ_p)$-torsor. By taking quotient of $\mathscr{S}^\ext_{K^pI}(\sG_a,\sX_a)_{k}^\perf$ and $\widetilde\Igs^a_{K^p}$ by $K_p^v$-actions, we get the desired Cartesian diagram.
\end{proof}

Note that $\mathscr{S}^\ext_{K^pI}(\sG_a,\sX_a)_{k}^\perf$ carries an action of $H(\QQ_p)$. If $K_p^v\subset H(\QQ_p)$ is an arbitrary open compact subgroup, we define 
$$\mathscr{S}^\ext_{K}(\sG_a,\sX_a)_{k}^\perf=\mathscr{S}^\ext_{K^pI}(\sG_a,\sX_a)_{k}^\perf/K_p^v\quad\text{and}\quad \Igs^a_{K^v}=\widetilde\Igs^a_{K^p}/K_p^v$$
as pro-\'etale quotients. Proposition \ref{prop-Igusa} holds in this generality. Thus we can drop the assumptions on $K_p^v$ of being an Iwahori subgroup in the sequel.
 
If $a=0$, then $\sG_a=\sG$ and $\Sh_{\mu_0}$ is a disjoint union of $|\ker^1(\QQ,\sG)|$-copies of $\sG(\QQ)\backslash\sG(\AAA_f)/K^vI$.  Let $b_1$ be the basic element in $B(G_0,\mu_{0,1}^*)$. Since $[\mu_{a,1}^*]=[\mu_{0,1}^*]\in \pi_1(Z(G_0))_{\Gal(\overline\QQ_p/\QQ_p)}$, the element $b_1$ is the basic element in $B(G_0,\mu_{a,1}^*)$ for any $0\leq a\leq n$.
As $\mu_{0,1}^*$ is central, the extended pure inner form $(G_0)_{b_1}$ associated to $b_1$ is identified to $G_0$. Thus $\Isoc_{G_0,b_1}\simeq \BB_{\textrm{pro\'et}}G(F)$. By definition, we have
$$\Igs_{K^v}^0=\bigsqcup_{\ker^1(\QQ,\sG)} \big(\sG(\QQ)\backslash\sG(\AAA_f)/K^v\big)/G(F).$$
The morphism $\loc_v^0\colon\Igs_{K^v}^0\to \BB_{\textrm{pro\'et}}G(F)$ is the obvious morphism. 

We need the following results on exotic Hecke correspondences.
\begin{prop}\label{prop-exotic-Hecke}
	Let $0\leq a\leq n$. There is a natural isomorphism
	$$\Igs_{K^v}^a\times_{\Isoc_{G_0,\leq\mu_{a,1}^*}}\Isoc_{G_0,b_1}\simeq \Igs_{K^v}^0.$$
\end{prop}
\begin{proof}
	We can assume that $K_p^v$ is an Iwahori subgroup. Then this follows from \cite{Sempliner-van-Hoften}. Alternatively, in this special case, it follows from Rapoport--Zink uniformization as in \cite[Proposition 7.3.5]{Xiao-Zhu}.
\end{proof}

We recall the local-global compatibility for cohomology of Shimura varieties studied in \cite{Tame} and \cite{YZ-torsion-vanishing}. Let $\omega_{\Igs_{K^v}^a}\in\Shv(\Igs_{K^v}^a,\Lambda)$ be the dualizing sheaf.

\begin{defn}
	Define the \emph{Igusa sheaf} $$\Igssheaf^a_{K^v}\coloneqq(i_{\leq\mu_{a,1}^*})_*(\loc_v^0)_\flat(\omega_{\Igs_{K^v}^a})\in \Shv(\Isoc_{G,\leq\mu_{a,1}^*},\Lambda).$$ 
    Here $(\loc_v^0)_\flat$ is the right adjoint of $(\loc_v^0)^!\colon \Shv(\Igs^a_{K^v})\to\Shv(\Isoc_{G_0,\leq\mu_{a,1}^*})$.
\end{defn}

By \cite[Proposition 6.7]{Tame}, $\Igssheaf_{K^v}^a$ is an admissible object in $\Shv(\Isoc,\Lambda)$. 

\begin{defn}
    Let $\cP^{\widehat\unip}$ be the right adjoint of the embedding $\Shv^{\widehat{\unip}}(\Isoc_G,\Lambda)\hookrightarrow\Shv(\Isoc_G,\Lambda)$. Recall the unipotent local Langlands functor $\LL^\unip$ in Theorem \ref{thm-unip-cllc}. For $0\leq a\leq n$, define the \emph{(unipotent) coherent Igusa sheaf} 
	$$\fI^a\coloneqq\LL^\unip(\cP^{\widehat\unip}(\Igssheaf^a_{K^v}))\in \Ind\Coh(\Loc_{{}^LG_0,\QQ_p}^{\widehat\unip}).$$
\end{defn}

The Igusa sheaf $\Igssheaf_{K^v}^a$ carries an action of the global Hecke algebra $\TT^S$. Thus $\fI^a$ also carries an action of $\TT^S$. By \cite[Corollary 1.9]{Tame}, the coherent Springer sheaf $\CohSpr\in\Coh(\Loc^{\widehat{\unip}}_{{}^LG_0,\QQ_p})$ carries an action of the Iwahori--Hecke algebra $\cH$. Let $V_{\mu_{a,1}}$ is the irreducible representation of $\hat{G}_0$ with highest weight $\mu_{a,1}$, and $\widetilde{V}_{\mu_{a,1}}$ be the pullback of $V_{\mu_{a,1}}$ along $\Loc_{{}^LG_0,\QQ_p}^{\widehat\unip}\to \BB\hat{G}_0$. We have the following results.

\begin{thm}\label{thm-coh-formula-Sh}
	For $0\leq a\leq n$, there is an isomorphism of $\Lambda$-modules
	$$R\Gamma(\Sh_K^\ext(\sG_a,\sX_a)_{\overline\QQ_p},\Lambda)[d]\simeq \Hom_{\Ind\Coh(\Loc_{{}^LG_0,\QQ_p}^{\widehat\unip})}(\widetilde{V}_{\mu_{a,1}}\otimes\CohSpr,\fI^a)$$
	compatible with $\TT^S\times\cH$-actions. Here $d=2a(n-a)$ is the dimension of $\sSh_{K^vI}(\sG_a,\sX_a)$. 
\end{thm}
\begin{proof}
    The existence of such an isomorphism and compatibility of $\TT^S$-actions follow from \cite[Corollary 6.17]{Tame}. Compatibility of $\cH$-actions follows from \cite[Theorem 4.20]{YZ-torsion-vanishing}.
\end{proof}

In particular, if $a=0$, then we obtain an isomorphism
$$\bigoplus_{\ker^1(\QQ,\sG)}\cA^I\simeq \Hom_{\Coh(\Loc_{{}^LG_0,\QQ_p}^{\widehat\unip})}(\CohSpr\otimes\widetilde{V}_{\mu_{0,1}},\fI^0).$$
Note that $V_{\mu_{0,1}}$ is a character of $\hat{G}_0$.

We also have the following result, which can be viewed as a geometric Jacquet--Langlands correspondence.

\begin{prop}\label{prop-coh-formula-set}
	For $0\leq a\leq n$, there is an isomorphism of $\TT^S\times \cH$-modules
	$$\bigoplus_{\ker^1(\QQ,\sG)}\cA^I\simeq \Hom_{\Coh(\Loc_{{}^LG_0,\QQ_p}^{\widehat\unip})}(\CohSpr\otimes\widetilde{V}_{\mu_{0,1}},\fI^a).$$
\end{prop}
\begin{proof}
	It follows from Proposition \ref{prop-exotic-Hecke} and the proof of \cite[Corollary 6.17]{Tame}.
\end{proof}

\subsection{Finishing the proof}\label{subsection-finish-proof}

Applying Proposition \ref{prop-Weil-restriction} to $G=\GL_{n,F}$ and $G_0=\Res_{F/\QQ_p}G$, we can use $\Loc^{\widehat\unip}_{{}^LG,F}$ instead of $\Loc^{\widehat\unip}_{{}^LG_0,\QQ_p}$ in the formulae of Theorem \ref{thm-coh-formula-Sh} and Proposition \ref{prop-coh-formula-set}. Explicitly, the isomorphism $\can\colon\Loc^{\widehat\unip}_{{}^LG_0,\QQ_p}\cong\Loc^{\widehat\unip}_{{}^LG,F}$ is given as follows: Note that $\hat{G}_0=(\hat{G})^{s}$ where $s=[F:\QQ_p]$.  Fix a geometric Frobenius $\phi$ and a tame generator $t$ in $W_{\QQ_p}$. Thus $\phi$ acts on $\hat{G}_0=(\hat{G})^{s}$ by cyclic permutation on factors. we have an isomorphism
$$\Loc_{{}^LG_0,\QQ_p}^{\widehat\unip}\simeq \{(x_1,\dots,x_s,y_1,\dots,y_s)\in (\hat{G})^s\times (\cU_{\hat{G}})^s|\mathrm{Ad}_{x_i}(y_{i+1}^q)=y_i,\  i=1,\dots,s\}/(\hat{G})^s$$
where $y_{s+1}=y_1$. The isomorphism
$$\Loc_{{}^LG_0,\QQ_p}^{\widehat\unip}\xrightarrow{\sim} \Loc_{{}^LG,F}^{\widehat\unip}$$
is defined by sending $(x_1,\dots,x_s,y_1,\dots,y_s)$ to $(x_1x_2\cdots x_s,y_1)\in \Loc_{{}^LG,F}^{\widehat\unip}$. 

\begin{lemma}
	There is a direct sum decomposition
	$$\fI^a=\bigoplus_{\fn,\xi}\fI^a_{\fn,\xi},$$
	where $\xi$ runs through $\Lambda$-points of $\hat{T}\git W$ and $\fn\subset\TT^S$ runs through maximal ideals of $\TT^S$ with residue field $\Lambda$, such that the component $\fI^a_{\fn,\xi}$ is set theoretically supported at $(\fn,\xi)\in (\Spec\TT^S\times\hat{T}\git W)(\Lambda)$ when considering $\TT^S\otimes\cO(\hat{T}\git W)$-actions, and $\fI^a_{\fn,\xi}=0$ for all but finitely many $(\fn,\xi)$.
\end{lemma}
\begin{proof}
	Let $C_b$ be a compact generator of $\Rep^{\widehat\unip}(G_b(F),\Lambda)$ for $b\in B(G)$. We know that $\Hom((i_b)_!C_b,\fI^a)=0$ for $a\notin B(G,\mu_{a,1}^*)$.
	Denote $C=\bigoplus_{b\in B(G,\mu^*_{a,1})}(i_b)_!C_b$. Then the perfect $\Lambda$-module $\Hom(C,\fI^a)$ carries an action of $\cO(\hat{T}\git W)\otimes\TT^S$. Therefore $\Hom(C,\fI^a)$ is supported on finitely many closed points $(\fn,\xi)\in (\Spec\TT^S\times\hat{T}\git W)(\Lambda)$. Denote $\fI^a_{\fn,\xi}=\fI^a\otimes_{\cO(\hat{T}\git W)\otimes\TT^S}(\cO(\hat{T}\git W)\otimes\TT^S)_{(\fn,\xi)}$, where $(\cO(\hat{T}\git W)\otimes\TT^S)_{(\fn,\xi)}$ is the localization at $(\fn,\xi)$. Then
    $$\Hom(C,\fI^a_{\fn,\xi})\simeq \Hom(C,\fI^a)_{(n,\xi)},$$
    and the natural map $\fI^a\to \bigoplus_{\fn,\xi}\fI^a_{\fn,\xi}$ is an isomorphism.
\end{proof}

Denote $\bar{\fI}^a_\fm\coloneqq\fI^a_{\fm,\xi_{\fm,v}}$. By Theorem \ref{thm-coh-formula-Sh}, we have an isomorphism
$$\overline{R\Gamma(\sSh_{K^vI}^\ext(\sG_a,\sX_a)_{\overline\QQ_p},\Lambda)_{\fm}}[d]\simeq \Hom_{\Coh(\Loc_{{}^LG,F}^{\widehat\unip})}(\CohSpr\otimes\widetilde{V}_{\mu_{a,1}},\overline\fI^a_\fm),$$
where $\overline{R\Gamma(\sSh_{K^vI}^\ext(\sG_a,\sX_a)_{\overline\QQ_p},\Lambda)_{\fm}}[d]$ is the direct summand of $R\Gamma(\sSh_{K^vI}^\ext(\sG_a,\sX_a)_{\overline\QQ_p},\Lambda)_{\fm}[d]$ where the action of $\cO(\hat{T}\git W)\simeq Z(\cH)$ is set theoretically supported at $\xi_{\fm,v}$.

\begin{rmk} 
	Using the similar method as Proposition \ref{prop-local-global-auto}, we can prove that $$\overline{R\Gamma(\sSh_{K^vI}^\ext(\sG_a,\sX_a)_{\overline\QQ_p},\Lambda)_{\fm}}[d]=R\Gamma(\sSh_{K^vI}^\ext(\sG_a,\sX_a)_{\overline\QQ_p},\Lambda)_{\fm}[d]$$
    is indeed an equality. However, it is not needed for our purpose.
	
	In fact, we expect a stronger statement to be true: the coherent sheaf $\fI^a_\fm$ should be supported at the preimage of $\xi_{\fm,v}$ along $\Loc_{{}^LG,F}^{\widehat{\unip}}\to\hat{T}\git W$. 
\end{rmk}

Take $a=1$. By Proposition \ref{prop-completion} and Proposition \ref{prop-Springer}, we have 
$$\widehat{L}_{\xi_{\fm,v}}\simeq \Spf\bigg(\frac{\Lambda[v_i]_{i\in Q_{\xi_{\fm,v}}}[[u_1,\dots,u_n]]}{(u_iv_i)_{i\in Q_{\xi_{\fm,v}}}}\bigg)/\hat{T}$$
and
$$(j_{\xi_{\fm,v}})^*\CohSpr=\bigoplus_{x\in \varsigma^{-1}(\xi_{\fm,v})}\cO_{X_{P_x}}.$$
The composition $\widehat{L}_{\xi_{\fm,v}}\hookrightarrow \Loc_{{}^LG,F}^{\widehat\unip}\to \BB\hat{G}$ factors through $\BB\hat{T}\to \BB\hat{G}$. Recall that 
$$\mu_{1,1}=(1^1,0^{n-1})\times(1^{n-1},0^1)\times (0^n)^{s-2}\in \XX^\bullet(\hat{T}_0)=(\ZZ^n)^s.$$
Therefore $V_{\mu_{1,1}}=\std\otimes\wedge^{n-1}(\std)$ and we have
$$(j_{\xi_{\fm,v}})^*(\widetilde{V}_{\mu_{1,1}})=\bigoplus_{1\leq i,j\leq n}\cO(\chi+\varepsilon_i-\varepsilon_j)\in \Pro\Coh(\widehat{L}_{\xi_{\fm,v}}),$$
where we denote $\chi=(1,\dots,1)\in \XX^\bullet(\hat{T})$.
It implies that there is a decomposition
$$\overline{R\Gamma(\sSh_{K^vI}^\ext(\sG_1,\sX_1)_{\overline\QQ_p},\Lambda)}_{\fm}[d]\simeq \bigoplus_{\substack{1\leq i,j\leq n \\ x\in \varsigma^{-1}(\xi_{\fm,v})}}\Hom_{\Pro\Coh(\widehat{L}_{[x_{\fm,v]}})}(\cO_{X_{P_x}}(\varepsilon_i-\varepsilon_j),\overline{\fI}_{\fm}^1(-\chi)).$$
By \cite[Theorem 1.5]{YZ-torsion-vanishing}, the cohomologically genericness of $\fm$ implies that $R\Gamma(\sSh_{K^vI}^\ext(\sG_1,\sX_1)_{\overline\QQ_p},\Lambda)_{\fm}[d]$ is concentrated in cohomological degree $0$. Thus $\overline{R\Gamma(\sSh_{K^vI}^\ext(\sG_1,\sX_1)_{\overline\QQ_p},\Lambda)}_{\fm}[d]$ is also concentrated in degree 0. For $x\in \varsigma^{-1}(\xi_{\fm,v})$ and $1\leq i,j\leq n$, we denote
$$H_{i,j}^x\coloneqq \Hom_{\Pro\Coh(\widehat{L}_{[x_{\fm,v]}})}(\cO_{X_{P_x}}(\varepsilon_i-\varepsilon_j),\overline{\fI}_{\fm}^1(-\chi)).$$
In particular, each $H_{i,j}^x$ is concentrated in degree 0. 

By Proposition \ref{prop-coh-formula-set}, we have an isomorphism
$$\bigoplus_{\ker^1(\QQ,\sG)}\cA_\fm^I\simeq \Hom_{\Pro\Coh(\widehat{L}_{[x_{\fm,v]}})}((j_{\xi_{\fm,v}})^*\CohSpr,\overline{\fI}_{\fm}^1(-\chi)).$$
We obtain a decomposition
$$\bigoplus_{\ker^1(\QQ,\sG)}\cA_\fm^I\simeq \bigoplus_{x\in \varsigma^{-1}(\xi_{\fm,v})}\Hom_{\Pro\Coh(\widehat{L}_{[x_{\fm,v]}})}(\cO_{X_{P_x}},\overline{\fI}_{\fm}^1(-\chi)).$$ 
By Lemma \ref{lemma-decomp-agree}, there are canonical isomorphisms $\cA_\fm^{I,x}\cong \Hom_{\Pro\Coh(\widehat{L}_{[x_{\fm,v]}})}(\cO_{X_{P_x}},\overline{\fI}_{\fm}^1(-\chi))$, and the morphisms $\alpha_{x,x'}$ and $\beta_{x,x'}$ in Theorem \ref{thm-main} are induces by $\cO_{X_{P_{x'}}}\hookrightarrow \cO_{X_{P_x}}$ and $\cO_{X_{P_x}}\twoheadrightarrow \cO_{X_{P_{x'}}}$ respectively.

By Proposition \ref{prop-composition}, it suffices to consider the case $P_x=P_{x'}\sqcup\{j\}$ for some $j\in Q_{\xi_{\fm,v}}\backslash P_{x'}$. Define the closed substack 
$$Y=\{v_i=0\text{ if }i\in P_{x'}, u_j=0\}$$
of $\widehat{L}_{\xi_{\fm,v}}$. There are short exact sequences
\begin{equation}\label{eq-exact-seq-1}\tag{S1}
	0\to \cO_Y(\alpha_j)\xrightarrow{v_j}\cO_{X_{P_x}}\to \cO_{X_{P_{x'}}}\to 0
\end{equation}
\begin{equation}\label{eq-exact-seq-2}\tag{S2}
	0\to \cO_{X_{P_{x'}}}\xrightarrow{u_j}\cO_{X_{P_x}}\to \cO_Y\to 0
\end{equation}
of pro-coherent sheaves on $\widehat{L}_{\xi_{\fm,v}}$. Hence we have isomorphisms
$$\Cone(\bigoplus_{\ker^1(\QQ,\sG)}\cA_\fm^{I,x}\xrightarrow{\alpha_{x,x'}}\bigoplus_{\ker^1(\QQ,\sG)}\cA_\fm^{I,x'})\simeq \Hom_{\Pro\Coh(\widehat{L}_{\xi_{\fm,v}})}(\cO_{Y},\overline\fI^1_\fm(-\chi))[1],$$
$$\Cone(\bigoplus_{\ker^1(\QQ,\sG)}\cA_\fm^{I,x'}\xrightarrow{\beta_{x,x'}}\bigoplus_{\ker^1(\QQ,\sG)}\cA_\fm^{I,x})\simeq \Hom_{\Pro\Coh(\widehat{L}_{\xi_{\fm,v}})}(\cO_{Y}(\alpha_j),\overline\fI^1_\fm(-\chi)).$$
Therefore it suffices to show that
$$\Hom_{\Pro\Coh(\widehat{L}_{\xi_{\fm,v}})}(\cO_{Y}(\alpha_j),\overline\fI^1_\fm(-\chi))\quad\text{and}\quad\Hom_{\Pro\Coh(\widehat{L}_{\xi_{\fm,v}})}(\cO_{Y},\overline\fI^1_\fm(-\chi))$$ 
are concentrated in degree 0.

We will use the following argument frequently.
\begin{lemma}\label{lemma-hom}
	Let $\cC$ be a $\Lambda$-linear category. If $A,B,C,D,E$ are objects in $\cC$ such that we are given two fiber sequences
	$$A\to B\to C\ \text{ and }\ C\to D\to E.$$
	Let $X\in \cC$ be an object. If the $\Lambda$-module $\Hom_\cC(?,X)$ sits in degree $0$ for $?=A$,$B$,$D$, and $E$, then $\Hom_{\cC}(C,X)$ also sits in degree 0.
\end{lemma}
\begin{proof}
	The fiber sequence 
	$$\Hom_\cC(C,X)\to\Hom_\cC(B,X)\to\Hom_{\cC}(A,X)$$
	implies that $\Hom_\cC(C,X)$ sits in degrees $[0,1]$. The fiber sequence
	$$\Hom_\cC(E,X)\to\Hom_\cC(D,X)\to\Hom_{\cC}(C,X)$$
	implies that $\Hom_\cC(C,X)$ sits in degrees $[-1,0]$. Together we see that $\Hom_\cC(C,X)$ sits in degree 0.
\end{proof}

\begin{lemma}
    There are isomorphisms
    $$\Cone(\bigoplus_{\ker^1(\QQ,\sG)}\cA_\fm^{I,x} \xrightarrow{\alpha_{x,x'}}\bigoplus_{\ker^1(\QQ,\sG)}\cA_\fm^{I,x'})\simeq \Cone(H_{j,j+1}^{x'}\xrightarrow{\beta_{x,x'}} H_{j,j+1}^x)[1],$$
    $$\Cone(\bigoplus_{\ker^1(\QQ,\sG)}\cA_\fm^{I,x'} \xrightarrow{\beta_{x,x'}}\bigoplus_{\ker^1(\QQ,\sG)}\cA_\fm^{I,x})\simeq \Cone(H_{j+1,j}^{x}\xrightarrow{\alpha_{x,x'}} H_{j+1,j}^{x'})[-1],$$
\end{lemma}
\begin{proof}
    Twisting the short exact sequence (\ref{eq-exact-seq-1}) by the line bundle $\cO(-\alpha_j)$, we obtain a short exact sequence
    $$0\to \cO_Y\xrightarrow{v_j}\cO_{X_{P_x}}(-\alpha_j)\to \cO_{X_{P_{x'}}}(-\alpha_j)\to 0.$$
    Note that $\Hom_{\Pro\Coh(\widehat{L}_{\xi_{\fm,v}})}(\cO_{X_{P_{x'}}}(-\alpha_j),\overline\fI^1_\fm(-\chi))=H_{j+1,j}^{x'}$ and $\Hom_{\Pro\Coh(\widehat{L}_{\xi_{\fm,v}})}(\cO_{X_{P_{x}}}(-\alpha_j),\overline\fI^1_\fm(-\chi))=H_{j+1,j}^{x}$.  Hence we have
    $$\Cone(H_{j,j+1}^{x'}\xrightarrow{\beta_{x,x'}} H_{j,j+1}^x)\simeq \Hom_{\Pro\Coh(\widehat{L}_{\xi_{\fm,v}})}(\cO_{Y}(\alpha_j),\overline\fI^1_\fm(-\chi)).$$
    The first isomorphism follows.

    Twisting the short exact sequence (\ref{eq-exact-seq-2}) by the line bundle $\cO(\alpha_j)$, we obtain a short exact sequence
    $$0\to \cO_{X_{P_{x'}}}(\alpha_j)\xrightarrow{u_j}\cO_{X_{P_x}}(\alpha_j)\to \cO_Y(\alpha_j)\to 0.$$
    The second isomorphism follows similarly.
\end{proof}

Now Theorem \ref{thm-main} follows from Lemma \ref{lemma-hom}, using that all the modules $H^{x}_{i,j}$'s and $\cA_\fm^{I,x}$'s are concentrated in degree $0$.

\section{Coherent sheaf associated to automorphic forms}\label{section-coh-sheaf}
Let notations be as in \S\ref{subsection-finish-proof}. Denote $\fA_\fm\coloneqq\LL^\unip((i_{1})_*\cA_\fm)$. In this section we study the structure of the coherent sheaf $\fA_\fm$.

Denote $Z_{{}^LG,F}^{\widehat\unip}\coloneqq\Gamma(\Loc^{\widehat\unip}_{{}^LG,F},\cO_{\Loc^{\widehat\unip}_{{}^LG,F}})$. Let $A$ be a coherent sheaf on $\Loc^{\widehat\unip}_{{}^LG,F}$. By adjunction, there is a morphism
$$\Gamma(\Loc^{\widehat\unip}_{{}^LG,F},A)\otimes_{Z_{{}^LG,F}^{\widehat\unip}}\cO_{\Loc^{\widehat\unip}_{{}^LG,F}}\to A.$$
If $A$ is concentrated in degree 0, then the above morphism factors through a morphism
$$H^0\bigg(\Gamma(\Loc^{\widehat\unip}_{{}^LG,F},A)\otimes_{Z_{{}^LG,F}^{\widehat\unip}}\cO_{\Loc^{\widehat\unip}_{{}^LG,F}}\bigg)\to A.$$ 
The main result of this section is the following theorem.

\begin{thm}\label{thm-coh-generic}
	Assume we are in the situation of Theorem \ref{thm-main}. Then we have:
	\begin{enumerate}
		\item The object $\fA_\fm$ is a genuine coherent sheaf on $\Loc^{\widehat\unip}_{{}^LG,F}$, i.e. it is concentrated in degree $0$.
		\item There is a natural isomorphism 
	$$H^0\bigg(\Gamma(\Loc^{\widehat\unip}_{{}^LG,F},\fA_\fm)\otimes_{Z_{{}^LG,F}^{\widehat\unip}}\cO_{\Loc^{\widehat\unip}_{{}^LG,F}}\bigg)\xrightarrow{\sim}\fA_\fm.$$
	\end{enumerate} 
\end{thm}

\begin{rmk}
	The space $\Gamma(\Loc^{\widehat\unip}_{{}^LG,F},\fA_\fm)$ of global sections of $\fA_\fm$ is expected to be equal to the co-Whittaker model of $\cA_\fm$. This can be proved once one can show that the unipotent local Langlands functor $\LL^\unip_G$ sends the unipotent part of the Iwahori--Whittaker representation to the structure sheaf $\cO_{\Loc^{\widehat{\unip}}_{{}^LG,F}}$.
    Therefore the theorem is saying that the space $\cA_\fm$ is determined by its co-Whittaker model. 
\end{rmk}

The rest of this section is devoted to prove Theorem \ref{thm-coh-generic}. 

\subsection{Coherent description of the identity stratum}

Let $\xi\in (\hat{T}\git W)(\Lambda)$ be a regular semisimple element. By Proposition \ref{prop-fix-parameter}, there is a fully faithful embedding
$$(\LL^\unip)^\Adm_{\xi}\colon\Rep^{\widehat\unip}(G(F),\Lambda)^\Adm_{\xi}\hookrightarrow \Coh(\widehat{L}_{\xi}).$$
We fix an element $x_\std\in\hat{T}(\Lambda)$ mapping to $\xi$ as in \S\ref{subsection-construct-of-morphism}. By Proposition \ref{prop-completion}, there is an isomorphism 
$$\widehat{L}_{\xi}\simeq \Spf\bigg(\frac{\Lambda[v_i]_{i\in Q_{\xi}}[[u_1,\dots,u_n]]}{(u_iv_i)_{i\in Q_{\xi}}}\bigg)/\hat{T}.$$
Let $\Coh_{1}(\widehat{L}_{\xi})\subseteq \Coh(\widehat{L}_{\xi})$ denote the essential image of $(\LL^\unip)^\Adm_{\xi}$. We give a coherent description of $\Coh_{1}(\widehat{L}_{\xi})$.

Let $Z=(\widehat{L}_{\xi})_{\mathrm{red}}$ be the underlying reduced stack of $\widehat{L}_{\xi}$. Thus $Z=\Spec(\Lambda[v_i]_{i\in Q_{\xi}})/\hat{T}$ is a smooth Artin stack. The natural embedding $Z\hookrightarrow \widehat{L}_{\xi}$ has a section $r\colon\widehat{L}_{\xi}\to Z$ defined by the ring homomorphism
$$\Lambda[v_i]_{i\in Q_{\xi}}\xrightarrow{v_i\mapsto v_i}\frac{\Lambda[v_i]_{i\in Q_{\xi}}[[u_1,\dots,u_n]]}{(u_iv_i)_{i\in Q_{\xi}}}.$$
Let $s\colon \BB\hat{T}\hookrightarrow Z$ be the closed embedding defined by $v_i=0$, $i\in Q_{\xi}$. Consider the following functors
$$\Coh(\widehat{L}_{\xi})\xrightarrow{r_*}\Coh(Z)\xrightarrow{s^*}\Coh(\BB\hat{T}).$$
We identify $\Coh(\BB\hat{T})$ with the category of perfect $\XX^\bullet(\hat{T})$-graded $\Lambda$-modules. For an object $A$ in $\Coh(?)$ for $?=\BB\hat{T}$, $Z$, or $\widehat{L}_{\xi}$, and $\lambda\in \XX^\bullet(\hat{T})$, let
$$\gr^\lambda A\coloneqq\Gamma(?,A\otimes\cO(-\lambda))$$
denote the $\lambda$-component of $A$. 

\begin{notation}\label{not-1}
	Let $P,Q\subseteq Q_{\xi}$ with $P\cap Q=\emptyset$. Denote $\cF_{P,Q}\coloneqq \cO_{\widehat{L}_{\xi}}/(u_i,v_j)_{i\in P,j\in Q}$, which is the pro-coherent sheaf on $\widehat{L}_{\xi}$ concentrated in degree 0. Note that $\cF_{\emptyset,Q}=\cO_{X_Q}$ and $\cF_{Q_{\xi},\emptyset}=\cO_Z$.
		
	If $P\subseteq Q_{\xi}$ is a subset, we denote $\alpha_P=\sum_{i\in P}\alpha_i$.
\end{notation}

\begin{prop}\label{prop-1-stratum}
	The subcategory $\Coh_{1}(\widehat{L}_{\xi})\subseteq \Coh(\widehat{L}_{\xi})$ consists of objects $A\in \Coh(\widehat{L}_{\xi})$ such that $\gr^\lambda(s^*r_*(A))\neq 0$ only if $\lambda= \alpha_P$ for some subset $P\subseteq Q_{\xi}$.
\end{prop}
\begin{proof}
	Denote $\Coh_{1}(\widehat{L}_{\xi})'\coloneqq\{A\in \Ind\Coh(\widehat{L}_{\xi})|\gr^\lambda(s^*r_*(A))\neq 0\text{ only if }\lambda=\alpha_P\text{ for }P\subseteq Q_{\xi}\}.$ The functor $s^*$ is computed by the Koszul complex
	$$s^*(A)=[A(\alpha_{Q_{\xi}}) \to\cdots\to\bigoplus_{P\subseteq Q_{\xi},|P|=2}A(\alpha_P)\to \bigoplus_{j\in Q_{\xi}}A(\alpha_j)\xrightarrow{(v_j)} A], \quad A\in\Coh(Z).$$
    Therefore $s^*r_*$ commutes with limits and colimits. Moreover, the graded piece $\gr^\lambda(s^*r_*A)$ is computed by the complex
    $$[\gr^{\lambda-\alpha_{Q_\xi}}(A)\to\cdots\to\bigoplus_{P\subseteq Q_\xi,|P|=2}\gr^{\lambda-\alpha_P}(A)\to\bigoplus_{j\in Q_\xi}\gr^{\lambda-\alpha_j}(A)\to\gr^{\lambda}(A)].$$
    We can define $\Ind\Coh_{1}(\widehat{L}_{\xi})'\subseteq \Ind\Coh(\widehat{L}_{\xi})$ and $\Pro\Coh_{1}(\widehat{L}_{\xi})'\subseteq\Pro\Coh(\widehat{L}_{\xi})$ by the same conditions.

	An object $V\in\Rep^{\widehat\unip}(G(F),\Lambda)_{\xi}^\Adm$ can be written as a retract of a finite colimit of $\cInd_I^{G(F)}\Lambda$. Therefore $\LL((i_1)_*V)$ can be written as a retract of a finite colimit of $(j_{\xi})^*\CohSpr$. By Proposition \ref{prop-Springer}, to check that $\Coh_{1}(\widehat{L}_{\xi})\subseteq\Coh_{1}(\widehat{L}_{\xi})'$, it suffices to check that each $\cO_{X_P},P\subseteq Q_{\xi}$ lies in $\Pro\Coh_{1}(\widehat{L}_{\xi})'$. This follows from a direct computation.
	
	To check that $\Coh_{1}(\widehat{L}_{\xi})'=\Coh_{1}(\widehat{L}_{\xi})$, it suffices to find objects in $\Coh_{1}(\widehat{L}_{\xi})$ that generate $\Coh_{1}(\widehat{L}_{\xi})'$ as idempotent complete $\Lambda$-linear category. Denote $\widehat{I}_{\xi}\coloneqq\cInd_I^{G(F)}\otimes_{Z(\cH)}Z(\cH)_{\xi}^\wedge$ as a pro-object in $\Rep^{\widehat\unip}(G(F),\Lambda)$. Let $\langle\widehat{I}_{\xi}\rangle\subseteq \Pro\Rep^{\widehat\unip}(G(F),\Lambda)$ be the idempotent complete $\Lambda$-linear category generated by $\widehat{I}_{\xi}$. Then we have 
    $$\Rep^{\widehat\unip}(G(F),\Lambda)^\Adm_{\xi}=\langle \widehat{I}_\xi\rangle\cap\Rep^{\widehat\unip}(G(F),\Lambda)^\Adm.$$ Form this we deduce that $\Coh_1(\widehat{L}_{\xi})=\langle (j_{\xi})^*\CohSpr\rangle\cap \Coh(\widehat{L}_{\xi})$. Here $\langle (j_{\xi})^*\CohSpr\rangle\subseteq \Pro\Coh(\widehat{L}_{\xi})$ is the idempotent complete $\Lambda$-linear category generated by $(j_{\xi})^*\CohSpr$. 
	
	We claim that for each $P\subseteq Q_{\xi}$, the object $\cO_Z(\alpha_P)$ lies in $\Coh_1(\widehat{L}_{\xi})$. We know that $\cF_{\emptyset,P}=\cO_{X_P}\in \langle (j_{\xi})^*\CohSpr\rangle$ for any $P\subseteq Q_{\xi}$. We further claim that for $P,Q,U\subseteq Q_{\xi}$ with $P\cap Q=\emptyset, U\subseteq P$, the object $\cF_{P,Q}(\alpha_U)$ lies in $\Pro\Coh_1(\widehat{L}_{\xi})$. The case $P=\emptyset$ is known. For $i\notin P$. We have short exact sequences
	$$0\to \cF_{P\sqcup\{i\},Q}(\alpha_{U\sqcup\{i\}})\xrightarrow{v_i} \cF_{P,Q}(\alpha_U)\to \cF_{P,Q\sqcup\{i\}}(\alpha_U)\to 0$$
	$$0\to  \cF_{P,Q\sqcup\{i\}}(\alpha_U)\xrightarrow{u_i} \cF_{P,Q}(\alpha_U)\to\cF_{P\sqcup\{i\},Q}(\alpha_U)\to 0.$$
	This proves the claim by induction on $|P|$. In particular, the objects $\cO_Z(\alpha_P)$ lies in $\Coh_1(\widehat{L}_{\xi})$.
	
	It suffices to show that the objects $\cO_Z(\alpha_P)$ for $P\subseteq Q_{\xi}$ generate $\Coh_1(\widehat{L}_{\xi})'$. Notice that 
    $$\gr^{\mu}(s^*r_*(\cO_Z(\alpha_P)))=\begin{cases}
        \Lambda\quad&\text{if }\mu=\alpha_P, \\
        0\quad&\text{otherwise}.
    \end{cases}$$
	We first show that the functor $s^*r_*\colon \Coh(\widehat{L}_{\xi})\to\Coh(\BB\hat{T})$ is conservative: The functor $r_*$ is clearly conservative. To show that $s^*\colon\Coh(Z/\hat{T})\to \Coh(\BB\hat{T})$ is conservative, we reduce to show that the pullback $\Coh(\AAA^1/\GG_m)\to \Coh(\BB\GG_m)$ is conservative. This is well-known.
	
	Let $A\in \Coh_1(\widehat{L}_{\xi})'$ be a non-zero object. Denote $\cE(A)=\{P\subseteq Q_{\xi}|\gr^{\alpha_P}(A)\neq 0\}$. Therefore $\cE(A)$ is non-empty. Let $P\in \cE(A)$ be a minimal element. We have $\gr^{\alpha_P}(A)=\gr^{\alpha_P}(s^*r_*A)$. Take $M=H^{m}(\gr^{\alpha_P}(A))$ with minimal $m$ such that $H^{m}(\gr^{\alpha_P}(A))$ is non-zero. Then $M$ contains a $\Lambda[[u_1,\dots,u_n]]$-submodule isomorphic to $\Lambda$.
    It induces a morphism $\cO_Z(\alpha_P)[-m]\to A.$ Let $B$ be the cone of the this morphism. Then $\gr^{\alpha_Q}(B)=\gr^{\alpha_Q}(A)$ if $Q\neq P$, $H^{\neq m}(\gr^{\alpha_P}(s^*r_*(B)))=H^{\neq m}(\gr^{\alpha_P}(s^*r_*(A)))$, and 
    $$\dim H^m(\gr^{\alpha_P}(s^*r_*(B)))=\dim H^m(\gr^{\alpha_P}(s^*r_*(A)))-1.$$ 
    Therefore by induction, we can write $A$ as a finite extension of $\cO_Z(\alpha_p)[m]$'s. This finishes the proof.
\end{proof}

Objects in $\Coh(\widehat{L}_{\xi})$ can be viewed as $\XX^\bullet(\hat{T})$-graded $S$-modules satisfying some nilpotent conditions, where $S=\frac{\Lambda[v_i]_{i\in Q_{\xi}}[[u_1,\dots,u_n]]}{(u_iv_i)_{i\in Q_{\xi}}}$. In fact, if $A\in\Coh(\widehat{L}_{\xi})$, then $\bigoplus_{\lambda\in \XX^\bullet(T)}\gr^\lambda(A)$ is a $\XX^\bullet(\hat{T})$-graded $S$-module. Let $X^\mathrm{pos}_{\xi}\subseteq \XX^\bullet(\hat{T})$ be the cone generated by $\alpha_i$ for $i\in Q_{\xi}$. Define the order $\leq$ on $X^\mathrm{pos}_{\xi}$ as follows: $\lambda\leq \lambda'$ if $\lambda'=\lambda+\sum_{i\in Q_{\xi}}m_i\alpha_i$ for $m_i\geq 0$.
\begin{cor}\label{cor-property-coh1} Let $A\in \Coh_1(\widehat{L}_{\xi})$ be an object.
	\begin{enumerate}
	\item $\gr^\lambda(A)=0$ if $\lambda\notin X^\mathrm{pos}_{\xi}.$
	\item For $\lambda\in X^\mathrm{pos}_{\xi}$, write $\lambda=\sum_{i\in Q_{\xi}} m_i\alpha_i$ with $m_i\geq 0$. Write $P=\{i\in Q_{\xi}|n_i\neq 0\}$. Then the morphism
		$$\prod_{i\in P}v_i^{m_i-1}\colon \gr^{\alpha_P}(A)\to \gr^{\lambda}(A)$$
		is an isomorphism.
	\item For $P\subseteq Q_{\xi}$ and $i\in P$, the action of $u_i$ on $\gr^{\alpha_P}(A)$ is trivial.
\end{enumerate}
\end{cor}
\begin{proof}
	The first two statements follows by checking the generators $\cO_Z(\alpha_P)$ of $\Coh_1(\widehat{L}_\xi)$. (3) follows from (2) as we have a constraint $u_iv_i=v_iu_i=0$, but the morphism
	$$v_i\colon \gr^{\alpha_P}(A)\to \gr^{\alpha_P+\alpha_i}(A)$$
	is an isomorphism. This forces that $u_i=0$ on $\gr^{\alpha_P}(A)$.
\end{proof}

Informally, this means an object $A\in \Coh_1(\widehat{L}_{\xi})$ is ``determined'' by the $ \gr^{\alpha_P}(A)$ for $P\subseteq Q_{\xi}$.

\subsection{Proof of Theorem \ref{thm-coh-generic}}
Denote $\xi=\xi_{\fm,v}$. The object $\fA_\fm$ lies in $\Coh(\widehat{L}_{\xi})$, so it suffices to prove the statements over $\widehat{L}_{\xi}$.
The morphism $r_*\colon\Coh(\widehat{L}_{\xi})\to \Coh(Z)$ is $t$-exact. Because $\fA_\fm$ lies in $\Coh_1(\widehat{L}_{\xi})$, to prove that it is concentrated in degree 0, it suffices to show that $\gr^{\alpha_P}(\fA_\fm)$ is concentrated in degree 0 for any $P\subset Q_{\xi}$.

\begin{lemma}\label{lemma-ext-comp-1}
	Let $A\in\Coh_1(\widehat{L}_{\xi})$. Let $P\subseteq Q_{\xi}$, then there is an isomorphism
	$$\Hom_{\Pro\Coh(\widehat{L}_{\xi})}\big(\cF_{P,\emptyset}(\alpha_P),A\big)\cong \gr^{\alpha_P}(A).$$
\end{lemma}
\begin{proof}
	If $P=\emptyset$, this is clear. We prove by induction on $|P|\geq 1$ that:
	\begin{enumerate}[(i)]
		\item If $m_i\geq 0$ for $i\in Q_{\xi}\backslash P$ and $m_i=1$ for $i\in P$, then there is a canonical isomorphism
			$$\Hom_{\Pro\Coh(\widehat{L}_{\xi})}\big(\cF_{P,\emptyset}(\sum_{i\in Q_{\xi}}m_i\alpha_i),A\big)\cong \gr^{\sum_{i\in Q_{\xi}}m_i\alpha_i}(A).$$
		\item If $m_i\geq 0$ for $i\in Q_{\xi}\backslash P$, $m_i\geq 1$ for $i\in P$, and there is at least one $i\in P$ with $m_i\geq 2$. Then
		$$\Hom_{\Pro\Coh(\widehat{L}_{\xi})}\big(\cF_{P,\emptyset}(\sum_{i\in Q_{\xi}}m_i\alpha_i),A\big)=0.$$
	\end{enumerate}
	
	We first consider the case $P=\{i\}$. There is a long exact sequence
	$$\cdots\xrightarrow{u_i}\cO_{\widehat{L}_{\xi}}(2\alpha_i)\xrightarrow{v_i}\cO_{\widehat{L}_{\xi}}(\alpha_i)\xrightarrow{u_i}\cO_{\widehat{L}_{\xi}}(\alpha_i)\xrightarrow{v_i}\cO_{\widehat{L}_{\xi}}\xrightarrow{u_i}\cO_{\widehat{L}_{\xi}} \to \cF_{\emptyset,\{i\}}\to 0.$$
	The claim follows from Corollary \ref{cor-property-coh1}.
	
	Assume that the statement is true for $P'=P\backslash\{j\}$. We have a long exact sequence
	$$\cdots\xrightarrow{u_j}\cF_{P',\emptyset}(2\alpha_i)\xrightarrow{v_j}\cF_{P',\emptyset}(\alpha_i)\xrightarrow{u_j}\cF_{P',\emptyset}(\alpha_i)\xrightarrow{v_j}\cF_{P',\emptyset}\xrightarrow{u_j}\cF_{P',\emptyset}\to \cF_{P,\emptyset}\to 0.$$
	Then the claim follows from induction hypothesis.
\end{proof}

\begin{proof}[Proof of Theorem \ref{thm-coh-generic} (1)]
By Corollary \ref{cor-property-coh1} and Lemma \ref{lemma-ext-comp-1}, it suffices to show that 
$$\Hom_{\Pro\Coh(\widehat{L}_{\xi})}\big(\cF_{P,\emptyset}(\alpha_P),\fA_\fm\big)$$ 
sits in degree 0 for any $P\subseteq Q_{\xi}$. Take $a=[\frac{n}{2}]$ in Theorem \ref{thm-coh-formula-Sh}. Note that
	$$(j_{\xi})^*V_{\mu_{a,1}}\cong \bigoplus_{\lambda\in \mathrm{Wt}(V_{\mu_{a,1}})}V_{\mu_{a,1}}(\lambda)\otimes_\Lambda\cO_{\widehat{L}_{\xi}}(\lambda).$$
	where $\lambda$ runs through weights appearing in $V_{\mu_{a,1}}=\wedge^a(\std)\otimes\wedge^{n-a}(\std)$ and $V_{\mu_{a,1}}(\lambda)$ is the $\lambda$-weight space.
	It follows that
	\begin{equation}\label{eq-sh-a}\tag{$\star$}
	    \begin{aligned}
	   &\ \ \overline{R\Gamma(\sSh_{K^vI}^\ext(\sG_a,\sX_a)_{\overline\QQ_p} ,\Lambda)}_{\fm}[d] \\ \simeq &\bigoplus_{ x\in \varsigma^{-1}(\xi)} \bigoplus_{\lambda\in \mathrm{Wt}(V_{\mu_{a,1}})}(V_{\mu_{a,1}}(\lambda))^*\otimes_\Lambda\Hom_{\Pro\Coh(\widehat{L}_{\xi})}(\cO_{X_{P_x}}(\lambda-\chi),\overline{\fI}_{\fm}^a(-\chi)).
	\end{aligned}
	\end{equation}
	Here $\chi=(1^n)\in \XX^\bullet(\hat{T})$. We note that $\lambda\in \mathrm{Wt}(V_{\mu_{a,1}})$ if and only if $\mathrm{dom}(\lambda-\chi)\leq (1^a,0^{\epsilon},(-1)^a)$ in $\XX^\bullet(\hat{T})^+$, where $\epsilon=1$ if $n$ is odd and $\epsilon=0$ if $n$ is even. Here $\mathrm{dom}(\lambda-\chi)$ is the dominant Weyl translate of $\lambda-\chi$.
	
	By \cite[Theorem 1.5]{YZ-torsion-vanishing}, the cohomology $\overline{R\Gamma(\sSh_{K^vI}^\ext(\sG_a,\sX_a)_{\overline\QQ_p} ,\Lambda)}_{\fm}[d]$ sits in degree 0. Therefore for $\lambda\in \XX^\bullet(\hat{T})$ with $\lambda\leq (1^a,0^{\epsilon},(-1)^a)$, the $\Hom$-space
	$$\Hom_{\Pro\Coh(\widehat{L}_{\xi})}(\cO_{X_{P_x}}(\lambda),\overline{\fI}_{\fm}^a(-\chi))$$
	sits in degree 0. If $A\in \Pro\Coh_1(\widehat{L}_{\xi})$, then by Proposition \ref{prop-coh-formula-set}, we have
	$$\Hom_{\Pro\Coh(\widehat{L}_{\xi})}(A,\overline{\fI}_{\fm}^a(-\chi))\cong \Hom_{\Pro\Coh(\widehat{L}_{\xi})}(A,\fA_\fm).$$
	For $P\subseteq Q_{\xi}$, the object $\cF_{P,\emptyset}(\alpha_P)$ lies in $\Pro\Coh_1(\widehat{L}_{\xi})$. Hence it suffices to show that 
	$$\Hom_{\Pro\Coh(\widehat{L}_{\xi})}\big(\cF_{P,\emptyset}(\alpha_P),\overline{\fI}_{\fm}^a(-\chi)\big)$$ 
	sits in degree 0 for each $P\subseteq Q_{\xi}$.
	
	Let $P,Q\subset Q_{\xi}$ with $P\cap Q=\emptyset$. Let $U\subseteq Q_{\xi}$ with $P\subseteq U$. We claim that the $\Hom$-space
	$$\Hom_{\Pro\Coh(\widehat{L}_{\xi})}(\cF_{P,Q}(\alpha_U),\overline{\fI}_{\fm}^a(-\chi))$$
	sits in degree 0. We do induction on $|P|$. If $P=\emptyset$, then $\cF_{\emptyset,Q}(\alpha_U)=\cO_{X_Q}(\alpha_U)$ appears in the formula (\ref{eq-sh-a}) because $\mathrm{dom}(\alpha_U)\leq (1^a,0^{\epsilon},(-1)^a)$. Thus the claims hold when $P=\emptyset$. Assume that the claims hold for $P'=P\backslash\{i\}$. Using the short exact sequences
	$$0\to \cF_{P,Q}(\alpha_U)\xrightarrow{v_i}\cF_{P\backslash\{i\},Q}(\alpha_{U\backslash\{i\}})\to \cF_{P\backslash\{i\},Q\sqcup\{i\}}(\alpha_{U\backslash\{i\}})\to 0$$
	$$0\to \cF_{P\backslash\{i\},Q\sqcup\{i\}}(\alpha_{U})\xrightarrow{u_i}\cF_{P\backslash\{i\},Q}(\alpha_U)\to \cF_{P,Q}(\alpha_U)\to 0$$
	and Lemma \ref{lemma-hom}, we see that the statement holds for $P$. 
\end{proof}

\begin{lemma}\label{lemma-detect-gen}
	Let $A\in\Coh_1(\widehat{L}_{\xi})^\heartsuit$. The natural morphism $H^0(\gr^0(A)\otimes_{\Lambda[[u_1,\dots,u_n]]}\cO_{\widehat{L}_{\xi}})\to A$ is an isomorphism if and only if for all $P\subseteq Q_{\xi}$ and $j\in P$, the morphism
	$$v_j\colon \gr^{\alpha_{P\backslash \{j\}}}(A)\to \gr^{\alpha_{P}}(A)$$
	identifies $\gr^{\alpha_{P}}(A)$ with the cokernel of $u_j\colon \gr^{\alpha_{P\backslash \{j\}}}(A)\to  \gr^{\alpha_{P\backslash \{j\}}}(A))$.
\end{lemma} 
\begin{proof}
	Note that the composition $ \gr^{\alpha_{P\backslash \{j\}}}(A)\xrightarrow{u_i} \gr^{\alpha_{P\backslash\{j\}}}(A)\xrightarrow{v_i}\gr^{\alpha_{P}}(A)$ is zero. It induces a natural morphism from $\mathrm{coker}(\gr^{\alpha_{P\backslash \{j\}}}(A)\xrightarrow{u_j}  \gr^{\alpha_{P\backslash \{j\}}}(A))$ to $\gr^{\alpha_{P}}(A)$. The statement is clear after noticing that
	$$\gr^{\alpha_P}\big(H^0(\gr^0(A)\otimes_{\Lambda[[u_1,\dots,u_n]]}\cO_{\widehat{L}_{\xi}}))\big)=H^0(\gr^0(A)\otimes_{\Lambda[[u_1,\dots,u_n]]}\Lambda[[u_j]]_{j\notin P}).$$
	Here we view $\Lambda[[u_j]]_{j\notin P}$ as the quotient of $\Lambda[[u_1,\dots,u_n]]$ by $(u_i)_{i\in P}$.
\end{proof}

\begin{lemma}\label{lemma-ext-comp-2}
	Let $A\in\Coh_1(\widehat{L}_{\xi})^\heartsuit$. Assume that
	$$\Hom_{\Pro\Coh(\widehat{L}_{\xi})}(\cF_{P,\emptyset},A)$$
	sits in degree $0$ for any $P\subseteq Q_{\xi}$. Then the natural morphism $H^0(\gr^0(A)\otimes_{\Lambda[[u_1,\dots,u_n]]}\cO_{\widehat{L}_{\xi}})\to A$ is an isomorphism.
\end{lemma}
\begin{proof}
	By Lemma \ref{lemma-detect-gen}, we need to show that for any $P\subseteq Q_{\xi}$ and $i\in Q_{\xi}\backslash P$, there is a natural isomorphism $\gr^{\alpha_{P\sqcup\{i\}}}(A)\cong \mathrm{coker}(\gr^{\alpha_P}(A)\xrightarrow{u_i}\gr^{\alpha_P}(A))$.
	By Corollary \ref{cor-property-coh1}, we compute that 
	$$\Hom_{\Pro\Coh(\widehat{L}_{\xi})}(\cF_{\{i\},\emptyset}(\alpha_{P}),A)=[\gr^{\alpha_P}(A)\xrightarrow{u_i}\gr^{\alpha_P}(A)\xrightarrow{v_i}\gr^{\alpha_{P\sqcup\{i\}}}(A)]$$
    in degrees $[0,2]$.
	It suffices to show that $\Hom_{\Pro\Coh(\widehat{L}_{\xi})}(\cF_{\{i\},\emptyset}(\alpha_{P}),A)$ sits in degree $0$ for each $P,i$ as above.
	
	We claim a stronger statement: For $P,U\subseteq Q_{\xi}$, $P\cap U=\emptyset$ :
	\begin{enumerate}[(i)]
		\item $\Hom_{\Pro\Coh(\widehat{L}_{\xi})}(\cF_{P,\emptyset}(\alpha_U),A)$ sits in degree 0.
		\item If $m_i\geq 1$ for $i\in U$, then the morphism
			$$\prod_{i\in U}v_i^{m_i-1}\colon\Hom_{\Pro\Coh(\widehat{L}_{\xi})}(\cF_{P,\emptyset}(\alpha_U),A)\to  \Hom_{\Pro\Coh(\widehat{L}_{\xi})}(\cF_{P,\emptyset}(\alpha_P+\sum_{i\in U}m_j\alpha_i),A)$$
			is an isomorphism.
	\end{enumerate}
	If $P=\emptyset$, then the statements are clear by assumptions. Assume $|P|\geq 1$ and take $i\in P$. We prove (ii) by induction on $|P|$. This follows from the long exact sequence
	$$\cdots\xrightarrow{u_i}\cF_{P\backslash\{i\},\emptyset}(\alpha_U+2\alpha_i)\xrightarrow{v_i}\cF_{P\backslash\{i\},\emptyset}(\alpha_U+\alpha_i)\xrightarrow{u_i}\cF_{P\backslash\{i\},\emptyset}(\alpha_U+\alpha_i)\xrightarrow{v_i}\cF_{P\backslash\{i\},\emptyset}(\alpha_U)\xrightarrow{u_i} \cF_{P,\emptyset}(\alpha_U)\to 0.$$
	Then we do induction on $|U|$ to prove (i). The case $|U|=\emptyset$ follows from assumptions. If $|U|\geq 1$, take $i\in U$. We have long exact sequence
	$$\cdots\xrightarrow{u_i}\cF_{P,\emptyset}(\alpha_U+\alpha_i)\xrightarrow{v_i}\cF_{P,\emptyset}(\alpha_U)\xrightarrow{u_i}\cF_{P,\emptyset}(\alpha_U)\xrightarrow{v_i}\cF_{P,\emptyset}(\alpha_{U\backslash\{i\}})\xrightarrow{u_i}\cF_{P,\emptyset}(\alpha_{U\backslash\{i\}})\to \cF_{P\sqcup\{i\},\emptyset}(\alpha_{U\backslash\{i\}})\to 0.$$
	Using (ii), we see that $\Hom_{\Pro\Coh(\widehat{L}_{\xi})}(\cF_{P\sqcup\{i\},\emptyset}(\alpha_{U\backslash\{i\}}),A)$ is computed by a bi-complex 
	$$[\Hom_{\Pro\Coh(\widehat{L}_{\xi})}(\cF_{P,\emptyset}(\alpha_{U\backslash\{i\}}),A)\xrightarrow{u_i}\Hom_{\Pro\Coh(\widehat{L}_{\xi})}(\cF_{P,\emptyset}(\alpha_{U\backslash\{i\}}),A)\xrightarrow{v_i}\Hom_{\Pro\Coh(\widehat{L}_{\xi})}(\cF_{P,\emptyset}(\alpha_{U}),A)].$$
	By induction hypothesis, both $\Hom_{\Pro\Coh(\widehat{L}_{\xi})}(\cF_{P,\emptyset}(\alpha_{U\backslash\{i\}}),A)$ and $\Hom_{\Pro\Coh(\widehat{L}_{\xi})}(\cF_{P\sqcup\{i\},\emptyset}(\alpha_{U\backslash\{i\}}),A)$ sits in degree 0. Looking at the spectral sequence, we see that $\Hom_{\Pro\Coh(\widehat{L}_{\xi})}(\cF_{P,\emptyset}(\alpha_{U}),A)$ sits in degree 0.
\end{proof}

\begin{proof}[Proof of Theorem \ref{thm-coh-generic} (2)]
Let notations be as in the Proof of Theorem \ref{thm-coh-generic} (1).
For $P\subseteq Q_{\xi}$, the object $\cO_{P,\emptyset}$ lies in $\Pro\Coh_1(\widehat{L}_{\xi})$. By Lemma \ref{lemma-ext-comp-2}, it suffices to show that
$$\Hom_{\Pro\Coh(\widehat{L}_{\xi})}(\cF_{P,\emptyset},\overline\fI^a_\fm(-\chi))$$
sits in degree 0 for any $P\subseteq Q_{\xi}$.

We claim that for $P,Q,U\subseteq Q_{\xi}$ with $P\cap Q=P\cap U=\emptyset$, the $\Hom$-space
$$\Hom_{\Pro\Coh(\widehat{L}_{\xi})}(\cF_{P,Q}(-\alpha_U),\overline\fI^a_\fm(-\chi))$$
sit in degree 0. We do induction on $|P|$. If $P=\emptyset$, then $\cF_{\emptyset,Q}(-\alpha_U)$ appears in the formula (\ref{eq-sh-a}). Thus the claim is true for $P=\emptyset$. Assume that the claim holds for $P'=P\backslash\{i\}$. We have two short exact sequences
$$0\to\cF_{P\backslash\{i\},Q\sqcup\{i\}}(-\alpha_U)\xrightarrow{u_i}\cF_{P\backslash\{i\},Q}(-\alpha_U)\to \cF_{P,Q}(-\alpha_U)\to 0$$
$$0\to\cF_{P,Q}(-\alpha_U)\xrightarrow{v_i} \cF_{P\backslash\{i\},Q}(-\alpha_{U\sqcup\{i\}})\to \cF_{P\backslash\{i\},Q\sqcup\{i\}}(-\alpha_{U\sqcup\{i\}})\to 0.$$
For the terms not equals to $\cF_{P,Q}(-\alpha_U)$, the claim is true by induction hypothesis. Using Lemma \ref{lemma-hom}, it follows that the claim holds for $P$.
\end{proof}

\section{Raising the level}\label{section-level-raising}
In this section we give an arithmetic application of our main result. Let notations be as in Theorem \ref{thm-main}. 

Let $L$ be a finite extension of $\QQ_\ell$. Let $T$ be a finite set of finite places of $\QQ$ containing $\ell$ but not containing $p$.
 Let $K_T=\prod_{w\in T}K_w$ be an open compact subgroup of $\prod_{w\in T}\sG(\QQ_w)$. 
Let $\zeta\colon K_T\to \GL(V_\zeta)$ be a continuous representation on a finite dimensional $L$-vector space $V_\zeta$. Assume that $(V_\zeta,\zeta)$ is locally algebraic when restricted to $K_\ell$. Let $M_\zeta$ be a $K^v$-stable $O_L$-lattice in $V_\zeta$.

\begin{defn}
	Define the space of automorphic form of $\sG$ with coefficients of in $O_L$ and type $M_\zeta$ by
	$$\cA_{O_L}(K_T,M_\zeta)=\{f\colon \sG(\QQ)\backslash\sG(\AAA_f)\to M_\zeta|f(gu)=\xi(u^{-1})f(g),\ u\in K_T,g\in\sG(\AAA_f)\}.$$
	Denote $\cA_L(K_T,V_\zeta)=\cA_{O_L}(K_T,M_\zeta)\otimes_{O_L}L$.
\end{defn}

Note that $\cA_{O_L}(K_T,M_\zeta)$ carries a smooth action of $\sG(\AAA_f^T)=\prod_{w\notin T}'\sG(\QQ_w)$. Let $\pi$ be an irreducible sub-$\sG(\AAA_f^T)$-representation of $\cA_{\overline\QQ_\ell}(K_T,V_\zeta)$. We can enlarge $L$ so that $\pi$ is defined over $L$. By \cite[Proposition 3.4.2]{CHT08}, there is a continuous Galois representation $\rho_\pi\colon \Gal(\overline\sF/\sF)\to \GL_n(L)$ that is compatible with local Langlands correspondences for $\GL_n$ at prime-to-$\ell$ split places, and is de Rham at $\ell$-adic places with Hodge--Tate weights determined by $\zeta$. Let $\overline{\rho}_\pi\colon \Gal(\overline{L}/L)\to \GL_n(\Lambda)$ be the semisimple reduction of $\rho_\pi$. 

Let $K^{T\sqcup\{p\}}\subseteq \sG(\AAA_f^{T\sqcup\{p\}})$ be a neat open compact subgroup such that $\pi^{K^{T\sqcup\{p\}}}\neq 0$. Let $K_p^v\subseteq H(\QQ_p)$ be an open compact subgroup with $(\pi_p^v)^{K_p^v}\neq 0$. We can assume that $K^{T\sqcup\{p\}}$ factorizes into $\prod_{w\notin T\sqcup\{p\}}K_w$. Let $S$ be the a finite set of places of $\QQ$ containing $p$ such that if $w\notin S\cup T$, then $\sG$ is a unramified over $\QQ_w$ and $K_w$ is a hyperspecial subgroup. Let 
$$\TT^{S\cup T}_{O_L}=\bigotimes'_{w\notin S\cup T}O_L[K_w\backslash\sG(\QQ_w)/K_w]$$ 
be the global Hecke algebra acting on $\cA_{O_L}(K_T,M_\zeta)^{K^{T\sqcup\{p\}}K_p^v}$. The automorphic representation $\pi$ defines a maximal ideal $\fm_\pi$ of $\TT^{S\cup T}_{O_L}$. Note that $\fm_\pi$ is also a maximal ideal of $\TT^{S\cup T}_{O_L}\otimes_{O_L}\Lambda$, and thus it make sense to say that $\fm_\pi$ is cohomologically generic as in Definition \ref{def-Gal-generic}.

Denote $\pi=\boxtimes'_{w\notin T}\pi_w$ where $\pi_w$ is an irreducible smooth representation of $\sG(\QQ_w)$. Let $\pi_p=\pi_v\boxtimes\pi_p^v$ be the decomposition with respect to $\sG(\QQ_p)=\GL_n(F)\times H(\QQ_p)$. Assume that $\rho_\fm$ is unramified at $v$. Then $\pi_v$ is a unramified principal series. Let $\xi_{\pi_v}\in(\hat{T}\git W)(L)$ be the Satake parameter of $\pi_v$. Then $q^{\frac{1-n}{2}}\cdot\rho_\pi(\Frob_v)$ lies in the conjugacy class $\xi_{\pi_v}$ as in Proposition \ref{prop-local-global-comp-hyperspecial}. Let $\overline{\xi}_{\pi_v}\in (\hat{T}\git W)(\Lambda)$ be the reduction of $\xi_{\pi_v}$. Assume that $\overline{\xi}_{\pi_v}$ is regular semisimple. If $\xi\in (\hat{T}\git W)(\overline\ZZ_\ell)\subseteq (\hat{T}\git W)(\overline\QQ_\ell)$ is a point with reduction $\overline{\xi}_{\pi_v}$, then the reduction map defines an inclusion $Q_{\xi}\hookrightarrow Q_{\overline{\xi}_{\pi_v}}$. Here, $Q_{\xi}$ is the set of pairs eigenvalues $(a,b)$ of $\xi$ with $a=qb$ as before.

Let $V$ be an unipotent irreducible smooth representation of $G(F)$ with coefficients in $\overline{\QQ}_\ell$. Denote by $\cH_{\overline{\QQ}_\ell}=\overline{\QQ}_\ell[I\backslash G(F)/I]$ the affine Hecke algebra over $\overline{\QQ}_\ell$. Then the action of $Z(\cH_{\overline{\QQ}_\ell})\cong \cO(\hat{T}_{\overline\QQ_\ell}\git W)$ on $V^I$ is set theoretically supported at a point $\xi_V\in (\hat{T}\git W)(\overline\QQ_\ell)$. We call $\xi_V$ the \emph{central character} of $V$. Assume that $\xi_V$ is regular semisimple. Let 
$$\rec_V\colon W_F\to \GL_n(\overline\QQ_\ell)$$ 
be the Weil group representation associated to $V$ by local Langlands correspondence. Then $\rec_V(\Frob_v)$ lies in the conjugacy class $\xi_V$. Pick a lift $x=\diag(x_1,\dots,x_n)\in \varsigma^{-1}(\xi_V)$. After conjugation, we can assume $\rec_V(\Frob_v)=x$. Let $t\in I_F/P_F$ be a tame generator. Then $\rec_V(t)$ is unipotent. Denote $N=\log(\rec_V(t))$. Then $N$ is a nilpotent matrix satisfying $\mathrm{Ad}_x(N)=q^{-1}N$. Denote
$$\II_V^x=\{(i,j)|1\leq i,j\leq n, \text{ the $(i,j)$-entry of $N$ is non-zero}\}.$$
The embedding $\II_V^x\hookrightarrow E_{\xi_V}\times E_{\xi_V}, (i,j)\mapsto (x_j,x_i)$ has image contained in $Q_{\xi_V}$. Denote by $\II_V$ the image of $\II_V^x\hookrightarrow Q_{\xi_V}$. Then $\II_V$ is independent of the choice of the point $x$. We call $\II_V$ the \emph{monodromy type} of $V$.

\begin{thm}\label{thm-level-raising}
	We assume the following assumptions:
	\begin{enumerate}
		\item $\ell$ is banal at $v$. 
		\item $\sF_v$ is unramified over $\QQ_p$ and $\sF_v\neq \QQ_p$.
		\item The maximal ideal $\fm_\pi\subseteq \TT^{S\cup T}_{O_L}$ is cohomologically generic.
		\item $\overline{\xi}_{\pi_v}$ is regular semisimple.
	\end{enumerate}
	Let $P\subseteq Q_{\overline{\xi}_{\pi_v}}$ be a subset. Then there is an irreducible subrepresentation $\pi'\subseteq\cA_{\overline{\QQ}_\ell}(K_T,V_\zeta)$ of $\sG(\AAA^T_f)$ with coefficients in $\overline{\QQ}_\ell$ such that:
	\begin{enumerate}
        \item $\pi'$ has non-trivial ${K^{T\sqcup\{p\}}K_p^v}$-invariant elements.
		\item After enlarging $L$, let $\rho_{\pi'}\colon \Gal(\overline\sF/\sF)\to\GL_n(L)$ be the Galois representation associated to $\pi'$. Then the semisimple reduction $\overline{\rho}_{\pi'}\colon\Gal(\overline\sF/\sF)\to\GL_n(\Lambda)$ of $\rho_{\pi'}$ is isomorphic to $\overline{\rho}_\pi$.
		\item Let $\pi'=\boxtimes'\pi_w$ and $\pi'_p=\pi'_v\boxtimes\pi_p^{\prime v}$ as before. Then $\overline{\xi}_{\pi_v}$ is equal to the reduction of $\xi_{\pi'_v}\in (\hat{T}\git W)(\overline\QQ_\ell)$. Under the canonical inclusion $Q_{\xi_{\pi'_v}}\hookrightarrow Q_{\overline{\xi}_{\pi_v}}$, the monodromy type $\II_{\pi'_v}$ of $\pi'_v$ has image $P$.
	\end{enumerate}
\end{thm}
\begin{proof}
Let $\overline{M}_\zeta=M_\zeta\otimes_{O_L}\Lambda$ be the reduction of $M_\zeta$. Take a normal open compact subgroup $K_T'\subseteq K_T$ such that the action of $K_T'$ on $\overline{M}_\zeta$ is trivial. We obtain an prime-to-$v$ level $K^{\prime v}=K^{T\sqcup\{p\}}K_p^vK_T'$. There is an isomorphism
$$\cA_{O_L}(K_T',M_\zeta)^{K^{T\sqcup\{p\}}K_p^v}\otimes_{O_L}\Lambda\cong \cA(K^{\prime v})\otimes_\Lambda \overline{M}_\zeta$$
of $G(F)\times K_T/K_T'$-modules.

\begin{lemma}\label{lemma-rep-char-0}
	Let $V$ be an unipotent irreducible smooth representation of $G(F)$ with coefficients in $\overline{\QQ}_\ell$. Let $\xi_V\in(\hat{T}\git W)(\overline\QQ_\ell)$ be the central character of $V$. Let $\II_V\subseteq Q_{\xi_V}$ be the monodromy type of $V$. Then
	$$\dim_{\overline\QQ_\ell} V^{I,x}=\left\{\begin{aligned}
	&0, \text{ if }P_x\neq \II_V \\
	&1, \text{ if }P_x=\II_V.
\end{aligned}\right.$$
\end{lemma}
\begin{proof}
	Let $x\in \hat{T}(\overline\QQ_\ell)$ be a lift of $\xi$. The composition series of $\nInd_{B(F)}^{G(F)}\chi_x$ contains all the irreducible representations with central character $\xi$ exactly once. By Corollary \ref{cor-parabolic-R}, we see that $\dim_{\overline\QQ_\ell} V^{I,x}$ is either 0 or 1. Then it follows from the Corollary \ref{cor-parabolic-R} and Bernstein--Zelevinsky classification. 
\end{proof}

We know that $\pi^{K^{T\sqcup\{p\}}K_p^v}$ is contained in the localization $\big(\cA_{O_L}(K_T,M_\zeta)^{K^{T\sqcup\{p\}}K_p^v}\big)_{\fm}\otimes_{O_L}\overline{\QQ}_\ell$. Moreover, if $\pi'\subseteq\cA_{O_L}(K_T,M_\zeta)$ is an irreducible subrepresentation such that:
\begin{enumerate}[(i)]
	\item $(\pi')^{K^{T\sqcup\{p\}}K_p^v}\neq 0$,
	\item $(\pi')^{K^{T\sqcup\{p\}}K_p^v}\subseteq \big(\cA_{O_L}(K_T,M_\zeta)^{K^{T\sqcup\{p\}}K_p^v}\big)_{\fm}\otimes_{O_L}\overline{\QQ}_\ell$.
\end{enumerate}
Then the Galois representation $\rho_{\pi'}$ associated to $\pi'$ satisfies $\overline\rho_{\pi'}\simeq \rho_\fm\simeq \overline{\rho}_\pi$.

Let $x,x'\in \varsigma^{-1}(\overline{\xi}_{\pi_v})$ with $P_{x'}\subsetneq P_x$. By definition, the morphisms $\alpha_{x,x'}$ and $\beta_{x,x'}$ in Theorem \ref{thm-main} are given by Hecke operators
$$\alpha^{\mathrm{Hk}}_{x,x'}\in H^0(\Hom_{G(F)}(\widehat\delta_{x'},\widehat\delta_{x}))\quad\text{and}\quad\beta^{\mathrm{Hk}}_{x,x'}\in H^0(\Hom_{G(F)}(\widehat\delta_{x},\widehat\delta_{x'})).$$
Moreover, the composition $\alpha_{x,x'}\circ\beta_{x,x'}\colon V^{I,x'}\to V^{I,x'}$ is nilpotent for $V\in \Rep^{\widehat\unip}(G(F),\Lambda)^{\Adm,\heartsuit}_{\xi}$ by Proposition \ref{prop-comp-unip}.

If $R$ is an algebra over $\ZZ_\ell[\sqrt{q}]$, we denote by $\cH_R$, $\cR_R$ the corresponding Hecke algebras with coefficients in $R$. Denote $\widehat\delta_{x,O_L}=\cInd_I^{G(F)}O_L\otimes_{\cR_{O_L}}(\cR_{O_L})^\wedge_{x}$. Let 
$$\tilde\alpha^{\mathrm{Hk}}_{x,x'}\in H^0(\Hom_{G(F)}(\widehat\delta_{x',O_L},\widehat\delta_{x,O_L}))\quad\text{and}\quad\tilde\beta^{\mathrm{Hk}}_{x,x'}\in H^0(\Hom_{G(F)}(\widehat\delta_{x,O_L},\widehat\delta_{x',O_L}))$$
be lifts of $\alpha^{\mathrm{Hk}}_{x,x'}$ and $\beta^{\mathrm{Hk}}_{x,x'}$ respectively. We obtain morphisms
$$\tilde\alpha_{x,x'}\colon V^{I,x}\to V^{I,x'}$$
$$\tilde\alpha_{x',x}\colon V^{I,x'}\to V^{I,x}$$
for any $V\in\Rep^{\widehat\unip}(G(F),O_L)$ with the action of $Z(\cH_{O_L})=\cO(\hat{T}_{O_L}\git W)$ on $V^I$ factors through the completion $Z(\cH_{O_L})^\wedge_{\xi}$. 

By Theorem \ref{thm-main}, the morphism
$$\alpha_{x,x'}\colon \cA(K^{\prime v})_\fm^{I,x}\to \cA(K^{\prime v})_\fm^{I,x'}$$
is surjective and the morphism
$$\beta_{x,x'}\colon \cA(K^{\prime v})_\fm^{I,x'}\to \cA(K^{\prime v})_\fm^{I,x}$$
is injective.  By \cite[Lemma 3.3.1]{CHT08}, $\cA(K^{\prime v})$ is finite free over $\Lambda[K_T/K_T']$. It follows that the morphism
$$\tilde\alpha_{x,x'}\colon\big(\cA_{O_L}(K_T,M_\zeta)_\fm^{K^{T\sqcup\{p\}}K_p^v}\big)^{I,x}\to \big(\cA_{O_L}(K_T,M_\zeta)_\fm^{K^{T\sqcup\{p\}}K_p^v}\big)^{I,x'}$$
is surjective after modulo $\varpi_L$, and
$$\tilde\beta_{x,x'}\colon\big(\cA_{O_L}(K_T,M_\zeta)_\fm^{K^{T\sqcup\{p\}}K_p^v}\big)^{I,x'}\to \big(\cA_{O_L}(K_T,M_\zeta)_\fm^{K^{T\sqcup\{p\}}K_p^v}\big)^{I,x}$$
is injective after modulo $\varpi_L$. 
Note that $(\cA_{O_L}(K_T,M_\zeta)_\fm^{K^{T\sqcup\{p\}}K_p^v})^I$ finite free over $O_L$. Hence the direct summand $\big(\cA_{O_L}(K_T,M_\zeta)_\fm^{K^{T\sqcup\{p\}}K_p^v}\big)^{I,x}$ are also finite free over $O_L$ for any $x\in Q_{\xi}$.

We have the following easy lemma.

\begin{lemma}
	Let $A$ and $B$ be finite free $O_L$-modules with homomorphisms $f\colon A\to B$ and $g\colon B\to A$. Assume $\overline{f}\colon A/\varpi_LA\to B/\varpi_LB$ is surjective and $\overline{g}\colon B/\varpi_LB\to A/\varpi_LA$ is injective. Also assume that $\overline{f}\circ\overline{g}$ is nilpotent. Then $\mathrm{rank}_{O_L}(A)>\mathrm{rank}_{O_L}(B)$ unless they are both zero.
\end{lemma}
\begin{proof}
    It is easy to see that $\dim(A/\varpi_LA)>\dim(B/\varpi_LB)$ unless they are both zero.
\end{proof}

Since the composition $\tilde\alpha_{x,x'}\circ \tilde\beta_{x,x'}$ is nilpotent after modulo $\varpi_L$, the above lemma implies that
$$\mathrm{rank}_{O_L}\big(\big(\cA_{O_L}(K_T,M_\zeta)_\fm^{K^{T\sqcup\{p\}}K_p^v}\big)^{I,x}\big)>\mathrm{rank}_{O_L}\big(\big(\cA_{O_L}(K_T,M_\zeta)_\fm^{K^{T\sqcup\{p\}}K_p^v}\big)^{I,x'}\big).$$

By \cite[Proposition 3.3.2]{CHT08}, $\cA_{O_L}(K_T,M_\zeta)_\fm\otimes_{O_L}\overline\QQ_\ell$ is a semisimple $\sG(\AAA_f^T)$-module. Fix a subset $P\subseteq Q_{\overline{\xi}_{\pi_v}}$. By Lemma \ref{lemma-rep-char-0}, there is an irreducible $G(F)$-submodule $V$ of $\cA_{O_L}(K_T,M_\zeta)_\fm^{K^{T\sqcup\{p\}}K_p^v}\otimes_{O_L}\overline{\QQ}_\ell$ such that $$\dim V^{I,x}=\left\{\begin{aligned}
	&1 \text{ if }P_x=P\\
	&0 \text{ if }P_x\subsetneq P.
\end{aligned}\right.$$
By Lemma \ref{lemma-rep-char-0}, we see that the monodromy type $\II_V$ of $V$ is equal to $P$. The $G(F)$-representation $V$ appears as a factor of some irreducible $\sG(\AAA_f^T)$-module $\pi'\subseteq\cA_{\overline{\QQ}_\ell}(K_T,V_\zeta)$. Then $\pi'$ is what we want.
\end{proof}

\bibliographystyle{alpha}
\bibliography{bib}
\end{document}